\documentclass[a4paper,11pt]{amsart}
\usepackage{graphicx} 
\input{Macros}

\usepackage[english]{babel}
\usepackage{tabu}
\usepackage{mathrsfs}

\newcommand{\fm}{\mathfrak{m}}

\newcommand{\cO}{\mathcal{O}}

\newcommand{\Lceil}{\Big\lceil}
\newcommand{\Rceil}{\Big\rceil}

\DeclareMathOperator{\totdis}{totaldiscrep}

\makeatletter
\let\save@mathaccent\mathaccent
\newcommand*\if@single[3]{%
  \setbox0\hbox{${\mathaccent"0362{#1}}^H$}%
  \setbox2\hbox{${\mathaccent"0362{\kern0pt#1}}^H$}%
  \ifdim\ht0=\ht2 #3\else #2\fi
  }
\newcommand*\rel@kern[1]{\kern#1\dimexpr\macc@kerna}
\newcommand*\widebar[1]{\@ifnextchar^{{\wide@bar{#1}{0}}}{\wide@bar{#1}{1}}}
\newcommand*\wide@bar[2]{\if@single{#1}{\wide@bar@{#1}{#2}{1}}{\wide@bar@{#1}{#2}{2}}}
\newcommand*\wide@bar@[3]{%
  \begingroup
  \def\mathaccent##1##2{%
    \let\mathaccent\save@mathaccent
    \if#32 \let\macc@nucleus\first@char \fi
    \setbox\z@\hbox{$\macc@style{\macc@nucleus}_{}$}%
    \setbox\tw@\hbox{$\macc@style{\macc@nucleus}{}_{}$}%
    \dimen@\wd\tw@
    \advance\dimen@-\wd\z@
    \divide\dimen@ 3
    \@tempdima\wd\tw@
    \advance\@tempdima-\scriptspace
    \divide\@tempdima 10
    \advance\dimen@-\@tempdima
    \ifdim\dimen@>\z@ \dimen@0pt\fi
    \rel@kern{0.6}\kern-\dimen@
    \if#31
      \overline{\rel@kern{-0.6}\kern\dimen@\macc@nucleus\rel@kern{0.4}\kern\dimen@}%
      \advance\dimen@0.4\dimexpr\macc@kerna
      \let\final@kern#2%
      \ifdim\dimen@<\z@ \let\final@kern1\fi
      \if\final@kern1 \kern-\dimen@\fi
    \else
      \overline{\rel@kern{-0.6}\kern\dimen@#1}%
    \fi
  }%
  \macc@depth\@ne
  \let\math@bgroup\@empty \let\math@egroup\macc@set@skewchar
  \mathsurround\z@ \frozen@everymath{\mathgroup\macc@group\relax}%
  \macc@set@skewchar\relax
  \let\mathaccentV\macc@nested@a
  \if#31
    \macc@nested@a\relax111{#1}%
  \else
    \def\gobble@till@marker##1\endmarker{}%
    \futurelet\first@char\gobble@till@marker#1\endmarker
    \ifcat\noexpand\first@char A\else
      \def\first@char{}%
    \fi
    \macc@nested@a\relax111{\first@char}%
  \fi
  \endgroup
}
\makeatother

\title{Base-change of locally stable families in positive characteristic}
\author{Marta Benozzo and Quentin Posva}
\date{}

\address{Laboratoire de Math\'ematiques d'Orsay, Universit\'e Paris--Saclay, Orsay, 91400, France}
\email{marta.benozzo@universite-paris-saclay.fr}

\address{Institut de Mathématiques, Université de Neuchâtel, Rue Émile-Argand 11, 2000 Neuchâtel, Switzerland} 
\email{quentin.posva@unine.ch}

\begin{document}

\maketitle

\begin{quote}
\textsc{Abstract:} We investigate the permanence of local stability for one-parameter families $X\to C$ under finite flat base-changes, when the base-field has positive characteristic $p>0$. Building on previous work of Hu--Zong, we show that it suffices to consider base-changes by Frobenius morphisms. In that case, we show that the situation is governed by the discrepancies of the pairs $(X,X_c)$, together with some differential invariants of the vertical divisors whose multiplicity in their fiber is divisible by $p$. 
While the behaviour of these invariants remains in general mysterious, we establish upper bounds under some $F$-splitting assumptions.
\end{quote}

\setcounter{tocdepth}{1}
\tableofcontents

\section{Introduction}
Let $k$ be an algebraically closed field of characteristic $p>0$. In the moduli theory of varieties over $k$, and in particular for the definition of moduli stacks and functors, one has to specify a notion of family over positive dimensional bases. The first important cases to consider are flat fibrations $X\to C$ where $C$ is a smooth $k$-curve. One usually asks that all pairs $(X,X_c)\to C$ be log canonical, where $c$ runs through the closed points of $C$: we say that $X\to C$ is \emph{locally stable}. This condition was originally introduced for the moduli theory of general type varieties in characteristic $0$: in that context it can be reformulated in several ways, and it is stable under base-change, see \cite{Kollar_Families_of_varieties_of_general_type}. In positive characteristic, much less is known, and the aim of this article is to study the following: 

\begin{question}\label{main_question}
    Let $X\to C$ be a locally stable family over a smooth $k$-curve $C$. Let $C'\to C$ be a finite flat morphism, where $C'$ is a smooth curve. Is the base-change $X'=X\times_C C'\to C'$ a locally stable family?
\end{question}

Some technicalities that are specific to positive characteristics already appear: it can happen that the base-change $X'$ is non-normal (\autoref{example:base-change_insep_node}), and it is not known whether the generic fiber of $X\to C$ is geometrically log canonical (see \autoref{rmk:test_only_vertical_divisors}). To avoid them, we impose conditions on the geometric fibers of $X\to C$: the closed ones must be demi-normal (with separable nodes when $p=2$), and the generic one must be geometrically log canonical. We call this stronger notion \emph{strong local stability} (\autoref{def:local_stability}). This is in accordance with the philosophy of moduli theory, as the geometric fibers should also belong to the relevant moduli stack. 

Strong local stability ensures that $X'$ is normal (\autoref{prop:demi_normality_loc_stable_family}) and that $K_{X'}$ is $\bQ$-Cartier (\autoref{lemma:Q_cartier_cond_base_changes}), so we can focus on the main question: what can we say about the discrepancies of the pairs $(X',X'_{c'})$?
This depends a lot on the morphism $C'\to C$. If it has only tame ramifications, then a well-known Riemann--Hurwitz computation shows that $X'\to C'$ is locally stable (\autoref{prop:tame_base_change}). If it is separable with wild ramifications, then the ramification divisor in the Riemann--Hurwitz formula becomes problematic. If $C'\to C$ is purely inseparable, then Riemann--Hurwitz does not apply: instead, the theory of $1$-foliations offers a good substitute.

In \cite{Hu_Zong_Base_change_local-stability_positive_char}, Hu and Zong show that permanence of (strong) local stability under base-change along Artin--Schreier covers $C'\to C$ (i.e.\ the extension of function fields is Galois of degree $p$) follows from permanence under base-change along purely inseparable covers (\footnote{
    The second author gives another proof of that statement in \cite{Posva_Local_stability_comparison}.
}). They mention, without proof, that \autoref{main_question} is therefore reduced to purely inseparable base-changes. We provide a proof of that statement:

\begin{theorem}[\autoref{p-purelyinsep_vs_wild}]
Assume that strong local stability is preserved by base-changes along purely inseparable morphisms of smooth curves. Then it is preserved by base-changes along arbitrary finite flat morphisms of smooth curves.
\end{theorem}

The idea is as follows: given an arbitrary $C'\to C$, factor it as a sequence of purely inseparable extensions, tame covers and Artin--Schreier covers. In general, this is only possibly for the Galois closure $C''$ of $C'/C$, so deducing local stability of $X'\to C'$ from local stability of $X''\to C''$ requires some descent statement. We prove the following two cases, which are sufficient for our purpose:

\begin{theorem}[\autoref{prop:descent_along_pur_insep} and \autoref{prop:descent_loc_stab_along_AS}]
Suppose that $C'\to C$ is either the $k$-linear Frobenius, or an Artin--Schreier cover. If $X'\to C'$ is locally stable, then $X\to C$ is locally stable.
\end{theorem}

The Artin--Schreier case requires general formulas about discrepancies of $\mathbb{Z}/p$-quotients, which we write down in \autoref{appendix:Z/p_qt}.

With this reduction in place, we study the discrepancies of the base-change $X^{(1)}\to C$ of $X\to C$ along the Frobenius morphism of $C$. Since we assume the generic fiber of $X\to C$ to be geometrically log canonical, we only need to study divisors that are vertical over $C$. As already said, the Riemann--Hurwitz formula does not apply, and the theory of $1$-foliations steps in. The projection $X^{(1)}\to X$ factors the Frobenius of $X^{(1)}$, so we have an induced $C$-morphism $X\to X^{(1)}$ which is the quotient by the relative tangent sheaf $T_{X/C}$ (see \autoref{section:div_and_Frob_base_change}). The so-called adjunction formula (see for example \cite[Proposition 2.5.14]{Posva_Singularities_quotients_by_1_foliations}) relates the discrepancies of $X\to C$ and those of $X^{(1)}\to C$ in terms of the discrepancies of $T_{X/C}$. The latter discrepancies are examined in \autoref{section:Frob_base_change}: to explain the formula that we ultimately obtain, we need to introduce one additional notion.

Let $Y\to X$ be a normal birational model, and let $E$ be a prime divisor of $Y$ whose image $\mathbf{0}\in C$ is a closed point. Following \cite{Benozzo_CBF_in_pos_char}, we say that $E$ is \emph{wild} if $\ord_EY_\mathbf{0}$ is divisible by $p$, and \emph{tame} otherwise. If $E$ is wild, we introduce a non-negative integer $\gamma_w(E)$, called the \emph{wild coefficient} of $E$, defined as follows. If $t\in \sO_{C,\mathbf{0}}$ and $s\in \sO_{Y,E}$ are uniformizers, then we can write
$t=us^{\ord_EY_\mathbf{0}}$ with $u$ a unit in $\sO_{Y,E}$. We let 
        $$\gamma_w(E)=\max\left\{r\geq 0\mid s^r \text{ divides }
        \ du \ \text{ in }\ \Omega^1_{\sO_{Y,E}/k}\right\}.$$
See \autoref{section:fibrations} for details. Then we have:

\begin{theorem}[\autoref{prop:formula_discrep_qt} and \autoref{prop:discrepancy_tame_divisor}]\label{thm:formula_discrepancy}
Assume that $X\to C$ is strongly locally stable, and let $E$ be a divisor over $X$, with image a closed point $\mathbf{0}\in C$. Let $E^{(1)}$ be the corresponding divisor over $X^{(1)}$. If $E$ is tame, then
        $$a\left(E^{(1)};X^{(1)},X_\mathbf{0}^{(1)}\right)\geq a(E;X,X_\mathbf{0}).$$
If $E$ is wild, then
        $$a\left(E^{(1)};X^{(1)},X_\mathbf{0}^{(1)}\right)
        = \begin{cases}
        p\cdot a(E;X,X_\mathbf{0})-(p-1)\cdot \gamma_w(E) &
        \text{if }E\text{ is }T_{X/C}\text{-invariant;} \\
        a(E;X,X_\mathbf{0})-\frac{p-1}{p}\cdot\gamma_w(E) &
        \text{otherwise.}
        \end{cases}$$
\end{theorem}

This shows that the permanence of (strong) local stability along Frobenius base-changes is controlled by the relations between discrepancies and wild coefficients. To the best of our knowledge, wild coefficients have not been studied so far in the literature, and bounding them in terms of discrepancies seems a delicate task. The computations of \autoref{section:invariance} indicate that a ``general'' wild divisor should be non-$T_{X/C}$-invariant, and the arguments of \autoref{section:study_wild_coeff} show that the invariant ones are more complicated to study. Let us also mention that $\gamma_w(E)$ is necessarily divisible by $p$ when $E$ is non-$T_{X/C}$-invariant (\autoref{cor:wild_coeff_p_div_in_non_inv_case}).

In principle, the Minimal Model Program (MMP) provides another route to \autoref{main_question}: if $X\to C$ is strongly locally stable, then $X_\mathbf{0}$ is slc and inversion of adjunction would imply that $X^{(1)}\to C$ is locally stable (see \autoref{lemma:inversion}). Inversion of adjunction is usually established as a consequence of the birational MMP (see the proof of \cite[Theorem 4.9]{Kollar_Singularities_of_the_minimal_model_program}), but currently only a few cases of inversion of adjunction in positive characteristic are known (see \autoref{section:inv_adjunction}). The precise relation between the birational MMP over a DVR, and the differential properties of locally stable families, is not clear yet. Some surprising consequences are indicated in \autoref{prop:discrep_formula_at_lc_places}.

There is one situation where inversion of adjunction along $X^{(1)}_\mathbf{0}$ holds in every dimension and characteristic: when $X_\mathbf{0}$ is $F$-pure (see \autoref{prop:permanence_F_pure_case}, which is a direct application of \cite{Polstra_Simpson_Tucker_F-pure_inversion_of_adjunction} and \cite{Hara_Watanabe_F_sing_vs_lt_and_lc}). 
The $F$-pure case is characterized by the existence of special differential operators locally on $X$: in \autoref{section:study_wild_coeff} we study wild coefficients in the presence of such operators. We obtain:

\begin{theorem}[\autoref{prop:bound_wild_coeff_in_F_pure_case}]
Suppose that $X\to (\mathbf{0}\in C)$ is strongly locally stable, and that $X$ is $1$-$F$-split along $X_\mathbf{0}$ 
(\autoref{def:F-split_along_div}). Then, if $E$ is a wild divisor over $X$,
        $$\gamma_w(E)\leq \begin{dcases}
            \ \prescript{p}{}\Lceil
            (p-1)\cdot a(E;X,X_\mathbf{0})\Rceil
            & \text{if }E\text{ is not }T_{X/C}\text{-invariant}, \\
            \ \prescript{p}{}\Lceil
            (p-1)\cdot a(E;X,X_\mathbf{0})\Rceil +p-2 & \text{otherwise,}
        \end{dcases}$$
where $\prescript{p}{}\lceil x\rceil=
\min\{n\in p\mathbb{Z}\mid n\geq x \}$.
\end{theorem}

If $\varphi$ is the splitting operator given by the assumption, then we obtain the above theorem by estimating the order of $\varphi(u^{p-1})-u^{p-1}$ along $E$ where $u$ is the unit introduced in the definition of $\gamma_w(E)$ (\autoref{lemma:bound_non_inv_case} and \autoref{lemma:bound_inv_case}), and by bounding from above the minimal integer $a$ such that $s^a\varphi$ is a regular operator at the generic point of $E$ (\autoref{prop:bound_on_pole}). The latter estimate is established using an argument contained in \cite{Hara_Watanabe_F_sing_vs_lt_and_lc}, and this is where the factor $p-1$ appears.

In view of \autoref{thm:formula_discrepancy}, this upper bound is in general too large by a factor of $p$. Our methods yield small enough bounds in a few cases, mostly for non-invariant divisors with discrepancies that are not too large: see \autoref{cor:application_non_inv_div} and \autoref{rmk:application_inv_div}. Nevertheless, substantial improvements of these arguments should be possible in view of \autoref{prop:permanence_F_pure_case}. Going beyond the $F$-split case appears to be more challenging.

In this article, we work more generally with families of pairs $(X,\Delta)\to C$. We also consider the case of \emph{demi-normal} families $(X,\Delta)\to C$: in that case we give a definition of strong local stability (\autoref{def:local_stability}) that is arguably quite restrictive, but which allows reduction to the normal case (\autoref{prop:red_to_normal_case}).

\begin{remark}
We only consider base-change by the Frobenius morphism. Base-change by iterated Frobenius can be addressed ``one step at the time" by decomposing the iterated Frobenius, or directly, granted that the necessary framework of higher foliations is in place---cf.\ the forthcoming
\cite{foliations_in_progress}.
See also \autoref{prop:loc_stability_lost_forever} and \autoref{rmk:F-splitting_higher_order}.
\end{remark}

\medskip
This article is organized as follows. In \autoref{section:definitions} we give our definitions related to local stability. 
In \autoref{section:prelim_observations} we gather preliminary technical results that are likely well-known to experts (Zariski and \'{e}tale localization on the base $C$, normality and $\bQ$-Gorensteiness of base-change, reduction of the slc to the log canonical case). Then we introduce wild divisors and wild coefficients in \autoref{section:fibrations}, and the $1$-foliations arising from Frobenius base-change in \autoref{section:div_and_Frob_base_change}. In \autoref{section:descent} we prove our descent results, using some material about $\mathbb{Z}/p$-quotients gathered in \autoref{appendix:Z/p_qt}. In \autoref{section:permanence_under_base_change} we study the discrepancies of base-changes: in \autoref{section:inv_adjunction} we discuss the cases in which inversion of adjunction is available; in \autoref{section:reduction_to_pur_insep} we reduce to Frobenius base-changes, building on \cite{Hu_Zong_Base_change_local-stability_positive_char, Posva_Local_stability_comparison}; and in \autoref{section:Frob_base_change} we establish the formulas of \autoref{thm:formula_discrepancy} and some general properties of wild coefficients. In \autoref{section:study_wild_coeff}, we start with some examples of families of non-$F$-pure surface singularities (\autoref{section:examples}), then we present some generalities about pullbacks of differential operators (\autoref{section:diff_operators}) which we apply to study wild coefficients in the $F$-split case in \autoref{section:wild_coeff_F-split_case}.

\subsection{Acknowledgments}
We thank Samuele Ciprietti, Stéphane Druel, Anne Fayolle and Alapan Mukhopadhyay for stimulating conversations and Tai-Hsuan Chung for numerous discussions. 
We thank Iacopo Brivio for comments on a previous draft.
The first author was supported by the European Union’s Horizon 2020 research and innovation programme under the Marie Skłodowska-Curie grant agreement No 101034255, Université Paris--Saclay and by the Simons Collaboration on Perfection in Algebra, Geometry, and Topology, award ID MP-SCMPS-00001529-10, Sorbonne Université. 
The second author is supported by the Ambizione grant num.\ 
PZ00-2-233513 from the Swiss National Science Foundation, and thanks the Université de Neuchâtel for its hospitality.

\section{Definitions}\label{section:definitions}

\subsection{Notations and conventions}
We work over an algebraically closed field $k$ of characteristic $p>0$.
We follow the standard terminology of birational geometry used in \cite{Kollar_Singularities_of_the_minimal_model_program}, with the caveat that varieties are reduced equidimensional but not necessarily irreducible (to accommodate semi-log canonical varieties). For basics on demi-normal varieties and on slc pairs in positive characteristics, we refer to \cite{Posva_Gluing_for_surfaces_and_threefolds}. For the theory of $1$-foliations, we refer to \cite{Posva_Singularities_quotients_by_1_foliations}.

\begin{definition}[see {\cite[Definition 2.2]{Kollar_Families_of_varieties_of_general_type}}]\label{def:family_of_pairs}
Let $C$ be a (germ of) smooth curve over $k$. A \emph{family of pairs} $f\colon (X,\Delta)\to C$ is the data of:
    \begin{itemize}
        \item a variety $X$ and a flat separated closed essentially of finite type morphism $f\colon X\to C$ with pure-dimensional, geometrically connected and geometrically reduced fibers (notice that this implies $f_*\sO_X=\sO_C$ if $f$ is proper);
        \item an effective $\bQ$-divisor $\Delta$ on $X$ whose support does not contain any irreducible component of any fiber $X_c$, and no irreducible component of $X_c\cap \Supp(\Delta)$ is contained in $\Sing(X_c)$; notice that this implies that there exists an open subset $U \subseteq X$ such that $\codim_X (X \setminus U) \geq 2$ and $\Delta|_{U}$ is $\bQ$-Cartier.
        \item We assume that $X$ is Gorenstein in codimension one (given the previous assumptions, this is a condition on the generic fiber), so it has a well-defined canonical divisor $K_X$.
    \end{itemize}
If $\Delta=0$, we also call $X\to C$ a \emph{family of varieties}.
\end{definition}

\begin{remark}
The fact that $f$ is closed is not always required in the literature. It is a technical assumption whose goal is to ensure that properties of closed fibers propagate to the generic one. It is not difficult in applications to do without it, but for general proofs it is easier to assume it from the start.
\end{remark}

\begin{remark}\label{rmk:relative_div_check_only_closed_fibers}
We claim that: the condition ``no irreducible component of $X_c\cap \Supp(\Delta)$ is contained in $\Sing(X_c)$" can be tested on closed points $c\in C(k)$ only.

Indeed, let $\eta$ be the generic point of $C$ and let $Z_\eta$ be a codimension one component of $\Sing(X_\eta)$. Then its closure $Z\subset X$ is a codimension one component of $\Sing(X)$, and it is flat over $C$. Since $f$ is closed and flat, for every $c\in C(k)$ the intersection $Z_c=Z\cap X_c$ has codimension one in $X_c$. The fiber $X_c$ cannot be regular at any point of $Z_c$, so $Z_c\subseteq \Sing(X_c)$. Therefore, if $Z_\eta\subseteq \Supp(\Delta_\eta)$, we obtain $Z\subseteq \Supp(\Delta)$ and thus $Z_c\subseteq X_c\cap \Supp(\Delta)$, which contradicts the condition.
\end{remark}

\begin{remark}
In \autoref{def:family_of_pairs} we do not assume that $K_X+\Delta$ is $\bQ$-Cartier. In particular, $(X,\Delta)$ is not a \emph{pair} in the usual sense. This is the convention of \cite{Kollar_Families_of_varieties_of_general_type}. A family of pairs with normal fibers, in our sense, is called a \emph{family of couples} in other sources, see e.g.\ \cite{Bernasconi_Brivio_Filipazzi_Deformations_of_3folds}.
\end{remark}

\begin{definition}[Tame ramifications]
Let $C$ be a (germ of) smooth curve over $k$, and let $\pi\colon C'\to C$ be a finite flat morphism from another (germ of) smooth curve over $k$. We say that $\pi$ has \emph{tame ramification} if the following holds: For every closed point $c'\in C'$ with image $\pi(c')=c\in C$, if $\xi\in \sO_{C,c}$ and $\xi'\in \sO_{C',c'}$ are uniformizers, we can write
$\xi=u\xi'^m$ where $u\in \sO_{C',c'}^\times$ and $m$ is coprime with the characteristic $p$. (Up to passing to the henselization of $\sO_{C',c'}$ and changing $\xi'$, we can assume that $u=1$.)
\end{definition}

\begin{definition}[Base-change of divisors]\label{def:base_change_divisors}
Let $f\colon (X,\Delta)\to C$ be a family of pairs, and let $C'\to C$ be a flat finite morphism of (germs of) smooth curves over $k$. We define $(X',\Delta')=(X,\Delta)\times_CC'$ by $X'=X\times_CC'\xrightarrow{g} X$ and $\Delta'=g^*\Delta$ (\footnote{
    To be more precise about the definition of $\Delta'$: let $U\subset X$ be the open subset on which $\Delta$ is $\bQ$-Cartier. By \autoref{rmk:relative_div_check_only_closed_fibers}, the codimension of $X\setminus U$ in $X$ is $\geq 2$. Let $U'=g^{-1}(U)$: then $U'$ is also big in $X'$. We pullback $\Delta|_U$ to $U'$ as $\bQ$-Cartier $\bQ$-Weil divisor, and we let $\Delta'$ be its closure on $X'$.
}).
\end{definition}

Note that if $(\mathbf{0}'\in C')\to (\mathbf{0}\in C)$ has only tame ramifications, then by Riemann--Hurwitz we have $K_{X'}+X'_{\mathbf{0}'}+\Delta'=g^*(K_X+X_\mathbf{0}+\Delta)$ 
\cite[2.41.4]{Kollar_Singularities_of_the_minimal_model_program}.
This equality does not hold for arbitrary base-changes, but the difference can be described: see the proof of \autoref{lemma:Q_cartier_cond_base_changes}.

\begin{remark}
With the notations of the above definition, let us verify that $(X',\Delta')\to C'$ is a family of pairs:
    \begin{itemize}
        \item The conditions on the morphism $f'\colon X'\to C'$ are clearly verified. 
        
        \item Clearly $g^{-1}\Supp(\Delta)$ does not contain any irreducible component of a fiber. If $c'\in C'$ is a closed point mapping to $c\in C$, then as closed subsets of $X'_{c'}$ the intersection $g^{-1}(\Supp(\Delta))\cap X'_{c'}$ is equal to the pullback of $\Supp(\Delta)\cap X_c$ along $X'_{c'}\cong X_c$. So $\Supp(\Delta')\cap X'_{c'}$ does not contain any irreducible component of $\Sing(X'_{c'})$. The case of closed fibers implies the case of the generic one, see \autoref{rmk:relative_div_check_only_closed_fibers}. 

        \item The morphism $C'\to C$ is Gorenstein \cite[0C12]{Stacks_Project}, so the finite morphism $X'\to X$ is Gorenstein \cite[0C07]{Stacks_Project}. It follows that $X'$ is Gorenstein in codimension one \cite[0C11]{Stacks_Project}.
    \end{itemize}
\end{remark}

Let us also record the following well-known fact:
\begin{lemma}\label{lemma:pur_insep_map_curves}
Let $\pi\colon C'\to C$ be a flat finite purely inseparable $k$-morphism of (germs of) smooth $k$-curves. Then $\pi$ is an iterate of the $k$-linear Frobenius morphism of $C$.
\end{lemma}
\begin{proof}
This follows from the proof of \cite[0CCY]{Stacks_Project}.
\end{proof}

Next we give our definitions of locally stable families over smooth $k$-curves. In characteristic $0$ there are several equivalent ways of defining them, thanks to (inversion of) adjunction. However the analoguous definitions in positive characteristic are in general not equivalent \cite{Kollar_Families_of_3folds_in_pos_char, Posva_Pathological_MMP_sing_alpha_p_qts}, so we differentiate them.

\begin{definition}[Local stability]\label{def:local_stability}
Let $C$ be a (germ of) smooth curve over $k$. A \emph{locally stable family of pairs} (resp.\ \emph{locally demi-stable family of pairs}) over $C$ is a family of pairs $f\colon (X,\Delta)\to C$ such that:
    \begin{enumerate}
        \item for each closed point $c\in C$, the pair $(X,\Delta+X_c)$ is log canonical (resp.\ semi-log canonical).
    \end{enumerate}
We say that $(X,\Delta)\to C$ is \emph{strongly locally stable}
(resp.\ \emph{strongly locally demi-stable}) if in addition:
    \begin{enumerate}\setcounter{enumi}{1}
        \item the generic fiber $(X_{K(C)},\Delta_{K(C)})$ is geometrically log canonical 
        (resp.\ geometrically semi-log canonical, see \autoref{def:geom_slc} below), and
        \item each closed fiber $X_c$ is demi-normal (i.e.\ $S_2$ and at worst nodal at codimension one points); and if $p=2$, the nodes of $X_c$ are all separable \cite[\S 3.1]{Posva_Gluing_for_surfaces_and_threefolds}.
    \end{enumerate}
\end{definition}

We mostly care about strong local (demi-)stability, as it appears to be better suited for moduli theory. Inseparable nodes cause additional technical problems in characteristic $2$, so for simplicity we exclude them. For base-change results, some restrictions on closed fibers seem necessary to control the singularities of the base-change of the total space; we will sometimes work with locally (demi-)stable families whose closed fibers are demi-normal with only separable nodes (i.e.\ families satisfying $(a)$ and $(c)$ in \autoref{def:local_stability}). By contrast, in several cases descent of local stability does not require any such additional assumptions.

\begin{remark}[Adjunction]\label{rmk:adjunction}
Assume that $(X,\Delta)\to C$ is locally stable, and pick a closed point $c\in C$. Then $(X_c^\nu, \Diff_{X^\nu_c}(\Delta))$ is log canonical, where $X^\nu_c$ is the normalization of $X_c$: this follows from \cite[Lemma 4.8]{Kollar_Singularities_of_the_minimal_model_program}. If $X_c$ is $S_2$, then in fact $(X_c,\Delta_c)$ is slc. 
\end{remark}

\begin{remark}
Assume that $(X,\Delta)\to C$ is strongly locally demi-stable. We do not know whether its normalization $(\widetilde{X},\widetilde{\Delta}+D)\to C$ is strongly locally stable, where $\widetilde{\Delta}$ is the strict transform of $\Delta$ and $D$ is the conductor divisor. The problem is whether the closed fibers of $\widetilde{X}\to C$ are $S_2$: this is true for families of surfaces in characteristic $p>5$ by
\cite{Arvidsson_Bernasconi_Patakfalvi_Properness_moduli_surfaces}, but the general case is open.
\end{remark}

\subsection{Geometrically slc singularities}
Let us fix a field $\Omega$ of characteristic $p>0$ with algebraic closure $\overline{\Omega}$. A \emph{variety} over $\Omega$ is a pure-dimensional, geometrically reduced scheme that is separated of finite type over $\Omega$, whose irreducible components are geometrically irreducible. 

\begin{definition}[see {\cite[\S 2.2]{Bernasconi_Tanaka_On_dP_fibrations_in_pos_char}, \cite[\S 2.1]{Benozzo_CBF_in_pos_char}}]\label{def:geom_slc}
Let $X$ be a demi-normal variety over $\Omega$ (see \cite[Definition 5.1]{Kollar_Singularities_of_the_minimal_model_program}), and let $\Delta$ be a $\bQ$-divisor on $X$. Assume that:
    \begin{itemize}
        \item $K_X+\Delta$ is $\bQ$-Cartier and $\Supp(\Delta)$ does not contain any irreducible component of $\Sing(X)$ that has codimension one in $X$, and
        \item $\Omega$ is algebraically closed in the function field of every irreducible component of $X$.
    \end{itemize}
We say that $(X,\Delta)$ is \emph{geometrically log canonical} if $(X_{\overline{\Omega}},\Delta_{\overline{\Omega}})$ is log canonical. 
We say that $(X,\Delta)$ is \emph{geometrically semi-log canonical} if:
    \begin{enumerate}
        \item $(X_{\overline{\Omega}},\Delta_{\overline{\Omega}})$ is semi-log canonical, and
        \item if $E\subset X$ is the conductor of the normalization of $X$, then $E_{\overline{\Omega}}$ is the conductor of the normalization of $X_{\overline{\Omega}}$.
    \end{enumerate}
\end{definition}

\begin{remark}
A log canonical singularity over $\Omega$ is not necessarily geometrically log canonical, even if it is geometrically normal. See \cite[Example 4.6]{Sato_General_hyperplane_section_of_lc_3folds_pos_char} for some examples. In the notations of \emph{op.cit.}, note that whenever we specialize $a$ to an element of a perfect field, the normal singularity we obtain is not log canonical: so these examples cannot appear as generic fibers in a locally stable family.
\end{remark}

\begin{remark}
We comment on the condition that $X_{\overline{\Omega}}$ is demi-normal with conductor $E_{\overline{\Omega}}$. For varieties, demi-normality is equivalent to being semi-normal (see \cite[Definition 10.11]{Kollar_Singularities_of_the_minimal_model_program}), $S_2$ and $G_1$ \cite[Remark 3.2.3, Lemma 3.2.4]{Posva_Gluing_for_surfaces_and_threefolds}. The latter two conditions are preserved by field extensions (by \cite[6.7.1]{EGA_IV.2} and \cite[0C03]{Stacks_Project} respectively). However semi-normality might not survive inseparable field extensions, see \cite[Remark 8.6]{Greco_Traverso_Seminormal_schemes}; note that, thanks to the $S_2$ condition, semi-normality of $X_{\overline{\Omega}}$ is decided in codimension one. Furthermore, even if $X_{\overline{\Omega}}$ is semi-normal, and thus demi-normal, it could happen that its conductor (which is reduced of pure codimension one) is larger than $\Supp(E_{\overline{\Omega}})$, or that $E_{\overline{\Omega}}$ is not reduced. 

In summary, for $X$ demi-normal with conductor $E$, the following two conditions are equivalent:
    \begin{itemize}
        \item $X_{\overline{\Omega}}$ is demi-normal with conductor $E_{\overline{\Omega}}$;
        \item the codimension one points of $X$ are geometrically semi-normal, the regular ones are geometrically regular, and $E$ is geometrically reduced.
    \end{itemize}
\end{remark}

\begin{lemma}\label{lemma:normal_of_geom_slc}
Assume that $(X,\Delta)$ is geometrically semi-log canonical. If $X^\nu$ is the normalization of $X$, with conductor $D\subset X^\nu$, then $(X^\nu,\nu^{-1}_*\Delta+D)$ is geometrically log canonical.
\end{lemma}
\begin{proof}
Let $Y=(X^\nu)_{\overline{\Omega}}$: it is reduced and $S_2$. The normalization $\nu\colon X^\nu\to X$ base-changes to 
$\nu_{\overline{\Omega}}\colon Y\to X_{\overline{\Omega}}$.
Since $(X_{\overline{\Omega}})^\nu$ is normal, the morphism $(X_{\overline{\Omega}})^\nu\to X$ factors through $X^\nu$ \cite[035Q]{Stacks_Project}, and so we have an induced finite morphism
$\phi\colon (X_{\overline{\Omega}})^\nu\to Y$ fitting into a commutative diagram
        $$\begin{tikzcd}
            (X_{\overline{\Omega}})^\nu \arrow[r, "\phi"] \arrow[d, "\overline{\nu}" left] & 
            Y \arrow[dl, "\nu_{\overline{\Omega}}"]
            \\
            X_{\overline{\Omega}}
        \end{tikzcd}$$
Let $E\subset X$ be the conductor of $X^\nu\to X$. By \autoref{def:geom_slc} the normalization morphism $\overline{\nu}$ is an isomorphism exactly above the complement of $E_{\overline{\Omega}}$. The same holds for $\nu_{\overline{\Omega}}$, since $\nu\colon X^\nu\to X$ is an isomorphism exactly above the complement of $E$. Since $X_{\overline{\Omega}}$ has nodal singularities in codimension one, it is easily seen that $\phi$ must be an isomorphism in codimension one. As $Y$ and $(X_{\overline{\Omega}})^\nu$ are $S_2$, it follows that $\phi$ is an isomorphism.

Clearly $(\nu^{-1}_*\Delta)_{\overline{\Omega}}$ is equal to the strict transform of $\Delta_{\overline{\Omega}}$ on $Y$. Moreover $D=\nu^{-1}(E)$ so, by flatness of the field extension:
        $$D_{\overline{\Omega}}=\nu_{\overline{\Omega}}^{-1}(E_{\overline{\Omega}})$$
which, as $X_{\overline{\Omega}}$ is demi-normal, is equal to the conductor of $Y\to X_{\overline{\Omega}}$. Therefore the normalization of $(X_{\overline{\Omega}},\Delta_{\overline{\Omega}})$ is equal to $(X^\nu,\nu^{-1}_*\Delta+D)\times_\Omega \overline{\Omega}$. This shows that the latter is lc, in other words that $(X^\nu,\nu^{-1}_*\Delta+D)$ is geometrically lc.
\end{proof}

\section{Preliminary observations}\label{section:prelim_observations}

Throughout this section, $C$ denotes a (germ of) smooth curve over $k$.

\subsection{Localization on the base}
The following result is well-known to the experts: we sketch its proof for convenience.
\begin{lemma}\label{lemma:localiz_on_the_base}
Let $(X,\Delta)\to C$ be a locally stable family of pairs. Then the following are equivalent:
    \begin{enumerate}
        \item For every finite flat morphism $C'\to C$, where $C'$ is a smooth curve, the base-change family 
        $(X,\Delta)\times_C C'\to C'$ is locally stable.
        \item For every closed point $c\in C(k)$ and every normal ring $\sO'$ that is a finite flat extension of $\sO_{C,c}$, the base-change family
        $(X,\Delta)\times_C \Spec(\sO')\to \Spec(\sO')$ is locally stable.
        \item For every closed point $c\in C(k)$ and every normal ring $\widetilde{\sO}'$ that is a finite flat extension of $\sO_{C,c}^h$ (the henselization of $\sO_{C,c}$), the base-change family
        $(X,\Delta)\times_C \Spec(\widetilde{\sO}')\to \Spec(\widetilde{\sO}')$ is locally stable.
    \end{enumerate}
Moreover, the equivalences still hold if we replace everywhere ``locally stable" by ``strongly locally stable".
\end{lemma}
\begin{proof}
The relevant properties of fibers (being geometrically integral or geometrically (demi)-normal) ascend and descend faithfully flat base-changes. So the content of the statement is about comparison of discrepancies: the main point is that they are \'{e}tale-local invariants. Here are some details.

$(a)\Leftrightarrow (b)$: The item $(b)$ follows from $(a)$ applied to $C'$, the normalization of $C$ in $\Frac(\sO')$. Conversely, assume that $(b)$ holds and let $E$ be an exceptional divisor of $X\times_CC'$. If the center of $E$ on $X\times_CC'$ is contained in the closed fiber over $c'\in C'$, then we can read $a(E;(X,\Delta)\times_C C')$ after base-changing to $\Spec(\sO_{C',c'})$, so we apply $(b)$ with $\sO'=\sO_{C',c'}$. If the center of $E$ is contained on the generic fiber, then the discrepancy can be read after base-changing to the local ring of any closed point of $C'$, and $(b)$ applies again.

$(c)\Rightarrow (b)$: Let $\sO'$ be a finite flat local extension of $\sO_{C,c}$. By \cite[0CBF]{Stacks_Project} we have a canonical isomorphism
        $$(\sO')^h\cong \sO'\otimes_{\sO_{C,c}}\sO_{C,c}^h.$$
So by $(c)$ and \cite[2.16]{Kollar_Singularities_of_the_minimal_model_program} we obtain that $(X,\Delta)\times_C \Spec(\sO')\to \Spec(\sO')$ is locally stable.

$(b)\Rightarrow (c)$: Let $\widetilde{\sO}'$ be a finite flat extension of $\sO_{C,c}^h$. 
By \autoref{lemma:approxim_ext_henselization} below we can find an \'{e}tale extension $R$ of $\sO_{C,c}$ and a finite flat extension $\sO'$ of $R$ such that $\sO'\otimes_{R}\sO_{C,c}^h\cong \widetilde{\sO}'$. The induced map $\sO_{C,c}\to \sO'$ is finitely generated and quasi-finite: so by Zariski Main Theorem \cite[02LR]{Stacks_Project} we can factor it as
        $$\sO_{C,c}\overset{\varphi}{\longrightarrow}
        \mathcal{Q} \overset{j}{\longrightarrow}
        \sO'$$
where $\varphi$ is finite and $j$ is a localization. 
Replacing $\mathcal{Q}$ by $\mathcal{Q}/\ker(j)$, we can also assume that $\mathcal{Q}$ is flat over the DVR $\sO_{C,c}$. By $(b)$ we obtain that $X\times_C\Spec(\mathcal{Q})\to \Spec(\mathcal{Q})$ is locally stable, so \emph{a fortiori} $(X,\Delta)\times_C\Spec(\sO')\to \Spec(\sO')$ is locally stable. By \cite[2.16]{Kollar_Singularities_of_the_minimal_model_program} it follows that $(X,\Delta)\times_C \Spec(\widetilde{\sO}')\to \Spec(\widetilde{\sO}')$ is locally stable.
\end{proof}

\begin{lemma}\label{lemma:approxim_ext_henselization}
Let $A$ be a Noetherian local ring, with henselization $A^h$. Let $B$ be a finite flat extension of $A^h$. Then there is an \'{e}tale extension $A\hookrightarrow C$, inducing an isomorphism on residue fields, and a finite flat extension $C\hookrightarrow C'$ such that $C'\otimes_CA^h\cong B$.
\end{lemma}
\begin{proof}
Recall that $A^h$ is a filtered colimit of \'{e}tale extensions of $A$. Thus by \cite[05N9]{Stacks_Project} there exists an \'{e}tale extension $C$ of $A$, and a finitely generated $C$-algebra $C'$ such that $C'\otimes_C A^h=B$. It remains to show that $C'$ is finite flat over $C$. The map $C\to A^h$ is faithfully flat, so flatness of $A^h \to B$ descends to $C\to C'$. Finiteness descends by combining \cite[07RG and 02LA]{Stacks_Project}. 
\end{proof}

\begin{remark}
Consider the following statement:
    \begin{enumerate}[label=(\alph*)]
    \setcounter{enumi}{3}
        \item For every closed point $c\in C(k)$ and every normal ring $\mathfrak{O}$ that is a finite flat extension of $\widehat{\sO}_{C,c}$, the base-change family 
        $(X,\Delta)\times_C \Spec(\mathfrak{O})\to \Spec(\mathfrak{O})$ is locally stable.
    \end{enumerate}
Then the equivalent conditions of \autoref{lemma:localiz_on_the_base} follow from the above one (the point is that completion is a regular morphism and preserves discrepancies), but the converse is unknown in general. The major issue is that we do not have an approximation result like \autoref{lemma:approxim_ext_henselization} for finite extensions of $\widehat{\sO}_{C,c}$. 
Nonetheless, if log resolutions of $(X,\Delta+X_c)$ exist, then all these conditions are equivalent \cite[2.16]{Kollar_Singularities_of_the_minimal_model_program}.
\end{remark}

\begin{remark}\label{rmk:test_only_vertical_divisors}
Let $(X,\Delta)\to C$ be a family of pairs with $K_X+\Delta$ $\bQ$-Cartier. It is natural to expect that, to test whether $(X,\Delta)\to C$ is locally stable, it suffices to consider divisors that are \emph{vertical} over $C$. This is known when there exists a proper log resolution of $(X,\Delta+X_c)$ for all $c\in C(k)$. Indeed, the existence of such a resolution implies that the log canonical locus of $(X,\Delta+X_c)$ is open for each $c$ (by \cite[Corollary 2.11]{Kollar_Singularities_of_the_minimal_model_program}), while it contains every closed fiber of $X\to C$; as $X\to C$ is closed, the log canonical locus must be equal to $X$.
To the best of our knowledge, the general case is still open.
\end{remark}

\subsection{Demi-normality of base-changes}
Let $C'\to C$ be a finite flat morphism where $C'$ is a smooth (germ of) curve. 

\begin{proposition}\label{prop:demi_normality_loc_stable_family}
Let $f\colon X\to C$ be a
family of varieties whose closed fibers are demi-normal with only separable nodes. 
Then $X'=X\times_CC'$ is demi-normal with only separable nodes. Furthermore, if $X$ is normal then $X'$ is also normal.
\end{proposition}
\begin{proof}
By \cite[12.1.6]{EGA_IV.3} the set $\{x\in X\mid X_{f(x)} \text{ is }S_2\text{ at }x\}$ is open. By hypothesis it contains every closed point of $X$, so it is equal to $X$. It follows that the generic fiber of $f$ is $S_2$. By \cite[6.7.1]{EGA_IV.2} the generic fiber of $f'=f\times_{C}C'$ is also $S_2$. As the closed fibers of $f'$ are $S_1$, it follows that $X'$ is $S_2$.
Since the fibers of $X'\to C'$ are geometrically reduced, we get that $X'$ is reduced. As the closed fibers of $f'$ are demi-normal, by \cite[Proposition 4.11]{Bernasconi_Posva_Equising_lifting_K-trivial_slc_surfaces} we obtain that $X'$ is demi-normal.

Let us show that $X'$ has only separable nodes. For demi-normal schemes, by \cite[Lemma 2.3.7]{Posva_Gluing_surfaces_mixed_char}, having only separable nodes is equivalent to being weakly normal. So it follows from the assumptions that the closed fibers of $f'$ are weakly normal. A straightforward adaptation of the arguments of \cite[Lemma 4.10, Proposition 4.11]{Bernasconi_Posva_Equising_lifting_K-trivial_slc_surfaces} shows that $X'$ is also weakly normal, as claimed.

Now assume that $X$ is normal: then it is geometrically unibranch, and the same holds for the generic fiber of $f'$ \cite[6.15.6]{EGA_IV.2}. Suppose that there is a codimension one nodal singularity $\eta'\in X'$: it is a separable node as we have seen, and it must map to the generic point of $C'$. So the separable node $\sO_{X',\eta'}$ is geometrically unibranch: this implies that it is normal, contradiction. So we obtain that $X'$ is normal.
\end{proof}

\autoref{prop:demi_normality_loc_stable_family} fails if we allow inseparable nodes in the fibers:

\begin{example}\label{example:base-change_insep_node}
In characteristic $p=2$, consider the family of surfaces $f\colon X\to \bA^1_t$ given by $X=V(x^2+y^2z+t)\subset \bA^4$ and where $f$ is the natural projection. Then $X$ is regular but the fiber $X_0$ is irreducible, demi-normal with an inseparable node \cite[\S 3.1]{Posva_Gluing_for_surfaces_and_threefolds}. Let $\mu\colon Y\to X$ be the blow-up defined by the ideal $(x,y,t)$. Straightforward computations show that $Y_0$ is reduced snc, and that $\mu^*(K_X+X_0)=K_Y+Y_0$. Therefore $X\to \bA^1_t$ is locally stable with regular total space. However, the base-change of $X$ along the $k$-linear Frobenius morphism of $\bA^1_t$ is not normal.
\end{example}

\subsection{$\bQ$-Cartier conditions}
We let $C'\to C$ as in the previous subsection.

\begin{lemma}\label{lemma:Q_cartier_cond_base_changes}
Let $(X,\Delta)\to C$ be a family of pairs,
with base-change $(X',\Delta')=(X,\Delta)\times_CC'$. Assume that $K_X+\Delta$ is $\bQ$-Cartier and that $X'$ is normal: then $K_{X'}+\Delta'$ is $\bQ$-Cartier.

If in addition $X\to C$ is projective, then $K_X+\Delta$ is ample over $C$ if and only if $K_{X'}+\Delta'$ is ample over $C'$.
\end{lemma}
\begin{proof}
It follows from \cite[V, Proposition 6.6]{Lang_Algebra} that we can factor $C'\to C$ as 
        $$C'\longrightarrow C_0\longrightarrow C$$
where $C_0$ is a smooth (germ of) $k$-curve, $C_0\to C$ is separable and $C'\to C_0$ is purely inseparable. It suffices to consider both parts separately; moreover, by \autoref{lemma:pur_insep_map_curves}, for the purely inseparable part it suffices to consider the case of a single $k$-linear Frobenius. Furthermore, we are allowed to shrink $C$ and $X$ Zariski-locally. Notice that $X$ is normal since $X'$ is normal and $X'\to X$ is faithfully flat \cite[033G]{Stacks_Project}.

\medskip
\textsc{Case 1: $C'\to C$ is the $k$-linear Frobenius of $C$.} We may assume that $X=\Spec(R)$, that $C=\Spec(\sO)$, and that
        $$R\cong \frac{\sO[x_1,\dots,x_n]}{(f_1,\dots,f_r)}$$
where the $f_i$s are polynomials in the $x_j$s with coefficients in $\sO$. Then $X'=\Spec(R')$ with
        $$R'\cong \frac{\sO^{1/p}[x_1,\dots,x_n]}{(f_1,\dots,f_r)},$$
where $\sO^{1/p}$ denotes the sub-ring of an algebraic closure of $\Frac(\sO)$ consisting of $p$-th roots of elements of $\sO$. Since $C'$ is smooth over $k$, after shrinking $C$ if necessary, the $\sO^{1/p}$-module $\Der_k(\sO^{1/p})$ is free, say generated by the $p$-closed derivation $\partial$. It extends to a $p$-closed derivation of $\sO^{1/p}[x_1,\dots,x_n]$ by setting $\partial(x_i)=0$ for every $i$. Since the coefficients of the $f_i$s are $p$-th powers in $\sO^{1/p}$, we have $\partial(f_i)=0$ and so $\partial$ descends to a $p$-closed derivation $\psi$ on $R'$. Let $\mathfrak{F}$ be the $1$-foliation that is freely generated by $\psi$: then $K_\mathfrak{F}\cong \sO_X$ is Cartier and the projection $q\colon X'\to X$ is the quotient by $\mathfrak{F}$. 

By definition $q^*\Delta=\Delta'$.
So it follows from \cite[Proposition 4.2.3]{Posva_Singularities_quotients_by_1_foliations} that we have an equality of $\bQ$-divisors
        $$q^*(K_X+\Delta)=K_{X'}+\Delta'+(p-1)K_\mathfrak{F},$$
whence $K_{X'}+\Delta'$ is $\bQ$-Cartier.

Assume now that $X$ is projective over $C$: then, after shrinking $C$ if necessary, we can find a \emph{global} presentation
        $$X=V(f_1,\dots,f_r)\subset \bP^n_{C},
        \quad f_i\text{ homogeneous},$$
and a similar argument as above shows that $\mathfrak{F}$ is globally generated by $\psi$. So $K_\mathfrak{F}\cong \sO_X$, and it follows at once that $K_{X'}+\Delta'$ is ample over $C'$ if and only if $K_X+\Delta$ is ample over $C$.

\medskip
\textsc{Case $2$: $C'\to C$ is separable.} Since $q\colon X'\to X$ is the base-change of $C'\to C$, the ramification divisor $R$ of $X'\to X$ is a sum of multiples of (scheme-theoretic) fibers. In particular it is Cartier. Moreover $q^*\Delta=\Delta'$, as the generic points of $\Delta$ are all contained in the generic fiber. The ramification formula \cite[2.41.2]{Kollar_Singularities_of_the_minimal_model_program} reads
        $$q^*(K_X+\Delta)=K_{X'}+\Delta'-R,$$
so we deduce that $K_{X'}+\Delta'$ is $\bQ$-Cartier. Since $\sO(R)\cong \sO_X$ locally over $C$, we see as in the previous case that $K_{X'}+\Delta'$ is ample over $C'$ if and only if $K_X+\Delta$ is ample over $C$.
\end{proof}

\subsection{Base-change with tame ramifications}
The following is well-known:
\begin{proposition}\label{prop:tame_base_change}
Let $C'\to C$ be a finite flat morphism of smooth curves over $k$, with only tame ramification. Let $f\colon (X,\Delta)\to C$ be a family of pairs. Then $f\colon (X,\Delta)\to C$ is locally stable (resp.\ strongly locally stable) if and only if $f'\colon (X,\Delta)\times_C C'\to C'$ is locally stable (resp.\ strongly locally stable).
\end{proposition}
\begin{proof}
We only sketch the argument. By \autoref{lemma:localiz_on_the_base} we may assume that $C$ and $C'$ are spectra of DVRs, with respective closed points $\mathbf{0}$ and $\mathbf{0}'$. By Riemann--Hurwitz the divisor $K_X+\Delta+X_\mathbf{0}$ pullbacks to $K_{X'}+\Delta'+X'_{\mathbf{0}'}$, so one is $\bQ$-Cartier if and only if the other is. If $Y\to X$ is any normal birational model, let $Y'$ be the normalization of $Y\times_CC'$: the projection $Y'\to Y$ has only tame ramifications in codimension one, as we explain below. Then the result follows from the Riemann--Hurwitz-based computations of \cite[2.42]{Kollar_Singularities_of_the_minimal_model_program}.

To show that $Y'\to Y$ has tame ramifications, let $E$ be a prime divisor of $Y$ and let $\sO=\sO_{Y,E}$. If $E$ dominates $C$ then $\sO\to \sO\otimes_{\sO_C}\sO_{C'}$ is \'{e}tale and there is nothing to show. So assume that $E$ maps to $\mathbf{0}$. Working \'{e}tale-locally, we may assume that 
$\sO_{C,\mathbf{0}}=k[t]_{(t)}$,
$\sO_{C',\mathbf{0}'}=k[t']_{(t')}$ with $t=t'^a$ where $a$ is prime to $p$,
and $\sO_{Y,E}=k[x,\mathbf{y}]_{(x)}$ with $t=ux^n$ where $u\in \sO_{Y,E}^\times$. Then
        $$\sO'=\sO_{Y,E}\otimes_{\sO_C}\sO_{C'}\cong
        k[x,\mathbf{y}, T']_{(x)}/(T'^a-ux^n).
        $$
We want to compute the normalization of $\sO'$. Since \'{e}tale extensions commute with normalizations, we may assume that $u^{1/a}$ belongs to $\sO'$. Then 
        $$\sO'^\nu\cong k[x',\mathbf{y}]_{(x')}
        \quad \text{with } x'^a=x
        \text{ and } x'^n=u^{-1/a}T',$$
see \cite[Proposition 2.2.1]{Posva_Resolution_of_1_foliations}.
This shows that $\sO\to\sO'^\nu$ has tame ramifications.
\end{proof}

\subsection{Reduction to the normal case}
Let us fix a locally demi-stable family $(X,\Delta)\to C$. We assume that $X$ is demi-normal but not normal: let $\nu\colon \widetilde{X}\to X$ be the normalization, with conductors $D\subset \widetilde{X}$ and $E\subset X$. Since $X$ is demi-normal, $E$ and $D$ are reduced of pure codimension one, flat over $C$. We have $\nu^*(K_X+\Delta)=K_{\widetilde{X}}+\widetilde{\Delta}+D$ where $\widetilde{\Delta}$ is the strict transform of $\Delta$, and $(\widetilde{X},\widetilde{\Delta}+D)\to C$ is locally stable.

Let $C'\to C$ be a finite flat morphism of smooth $k$-curves. Write $(X',\Delta')=(X,\Delta)\times_CC'$ and $(\widetilde{X}',\widetilde{\Delta}'+D')=(\widetilde{X},\widetilde{\Delta}+D)\times_CC'$. It is clear that $\widetilde{\Delta}'$ is the strict transform of $\Delta'$. We also let $D'=D\times_CC'$ and $E'=E\times_CC'$.

\begin{proposition}\label{prop:red_to_normal_case}
With the above notation, assume that:
    \begin{enumerate}
        \item $(X,\Delta)\to C$ is strongly locally demi-stable, and
        \item $(\widetilde{X}',\widetilde{\Delta}'+D')\to C'$ is locally stable.
    \end{enumerate}
Then $(X',\Delta')\to C'$ is strongly locally demi-stable.
\end{proposition}
\begin{proof}
By \autoref{prop:demi_normality_loc_stable_family} we have that $X'$ is demi-normal. By assumption $(X'_{K(C')},\Delta'_{K(C')})$ is geometrically slc, and the closed fibers of $X'\to C'$ are demi-normal with only separable nodes. So we have to show $(X',\Delta')\to C'$ is locally demi-stable.

By assumption $\widetilde{X}'$ is normal, so $\widetilde{X}'\to X'$ is the normalization. Let $\mathfrak{E}'\subset X'$ and $\mathfrak{D}'\subset \widetilde{X}'$ be the conductors of the normalization. Every irreducible component of $\mathfrak{E}'$ dominates $C'$, since $X'$ is demi-normal with reduced fibers over $C'$.
By \autoref{def:geom_slc} we see that $E'$ is reduced and equal to $\mathfrak{E}'$ on the generic fiber: so $\mathfrak{E}'=E'$. Since $(\widetilde{X}',\widetilde{\Delta}'+D')$ is lc, we see that $D'$ is generically reduced: since $D'\to D$ is finite flat and $D$ is $S_1$, we obtain that $D'$ is $S_1$ \cite[Theorem 23.9]{Matsumura_Commutative_Ring_Theory}, and so it is reduced. By construction $D'$ is the preimage of $E'$ in $\widetilde{X}'$, so in fact $D'=\mathfrak{D}'$.

Let $\sD$ be the normalization of $D$, with involution $\tau\colon \sD\to \sD$ induced by the morphism $D\to E$. Similarly, let $\sD'$ be the normalization of $D'$ with involution $\tau'$ of $\sD'$ induced by $D'\to E'$. Since $D'\to E'$ is the base-change of $D\to E$, we see that $\tau'$ is the unique extension of the involution $\tau\times_CC'$ of $\sD\times_CC'$.

Consider the commutative diagram
        $$\begin{tikzcd}
        \sD'\arrow[d, "i'"] \arrow[r, "h"] &
        \sD\times_CC' \arrow[r, "g"] & \sD \arrow[d, "i"] \\
        \widetilde{X}'\arrow[rr, "\phi"] && \widetilde{X}.
        \end{tikzcd}$$
The involution $\tau$ preserves the log canonical divisor 
$K_\sD+\Diff_\sD(\widetilde{\Delta}+D)=
i^*(K_{\widetilde{X}}+\widetilde{\Delta}+D)$. So $\tau\times_CC'$ preserves
$g^*(K_\sD+\Diff_\sD(\widetilde{\Delta}+D))$, and thus $\tau'$
preserves 
$h^*g^*(K_\sD+\Diff_\sD(\widetilde{\Delta}+D))$. By adjunction and commutativity of the diagram, we have
        \begin{eqnarray*}
            h^*g^*(K_\sD+\Diff_\sD(\widetilde{\Delta}+D))&=&
            i'^*\phi^*(K_{\widetilde{X}}+\widetilde{\Delta}+D) \\
            &=&
            i'^*(K_{\widetilde{X}'}+\widetilde{\Delta}'+D') \\
            &=&
            K_{\sD'}+\Diff_{\sD'}(\widetilde{\Delta}'+D')
        \end{eqnarray*}
so $\tau'$ is an involution of the pair $(\sD', \Diff_{\sD'}(\widetilde{\Delta}'+D'))$.

To conclude the proof, it remains to show that $K_{X'}+\Delta'$ is $\bQ$-Cartier.
This holds on the generic fiber, since field extensions preserve such properties. Take a closed point $c\in C'$: no lc center of $(X',\Delta')$ is contained in the fiber $X'_c$, as otherwise $(X',\Delta'+X_c')$ would not be slc. 
Combine the fact that $\tau'$ is an involution of $(\sD', \Diff_{\sD'}(\widetilde{\Delta}'+D'))$ with the argument in \cite[Proof of Proposition 5.38]{Kollar_Singularities_of_the_minimal_model_program}, until ``\emph{... is finite by the easy (9.55)}" in the middle of page 212: we obtain that $K_{X'}+\Delta'$ is $\mathbb{Q}$-Cartier. Note that Koll\'{a}r's argument refers to \cite[Corollary 5.33]{Kollar_Singularities_of_the_minimal_model_program} and ultimately to \cite[Theorem 5.32]{Kollar_Singularities_of_the_minimal_model_program} for the construction of geometric quotients by finite equivalence relations: in our positive characteristic context, existence of such quotients holds by the easier \cite[Theorem 6]{Kollar_Quotients_by_finite_equivalence_relations}.
\end{proof}

\section{Fibrations and wild divisors}\label{section:fibrations}
In this section, we introduce the notion of wild divisors and wild coefficients.
We fix a smooth $k$-curve $C$, 
and we let $f\colon \sX\to C$ be a flat finite type morphism with $\sX$ normal.
We define
        \begin{equation}\label{eqn:relative_log_tangent}
        T_{\sX/C}=\ker\left(
            df\colon T_{\sX/k}\longrightarrow f^*T_{C/k}
        \right).
        \end{equation}
Since $f^*T_{C/k}$ is a line bundle, we see that $T_{\sX/C}$ is coherent and saturated. On the open set where $\sX\to C$ is relatively smooth, we can describe the sections of $T_{\sX/C}$ as those of $T_{\sX/k}$ which are $f^{-1}\sO_C$-linear (equivalently, which annihilate $f^{-1}\sO_C$). In particular, $T_{\sX/C}$ is generically closed under Lie brackets and $p$-th powers. As it is saturated, this is true globally. Hence $T_{\sX/C}$ is a $1$-foliation on $\sX$.

Next we compute the canonical divisor of $T_{\sX/C}$. For this we need to identify the image of $df\colon T_{\sX/k}\to f^*T_{C/k}$, at least in codimension one. Let $\eta$ be a codimension one point of $\sX$.
    \begin{enumerate}
        \item Assume that $\eta$ is contained in the generic fiber of $f$. We have an exact sequence
                \begin{equation}\label{eqn:horiz_divisor}
                         0\to T_{\sX/C,\eta}
                        \to T_{\sX/k,\eta}
                        \to f^*T_{C/k}\otimes \fm_{\sX,\eta}^e 
                        \to 0.
                    \end{equation}
        where $\fm_{\sX,\eta}$ is the maximal ideal of $\sO_{\sX, \eta}$ and $e=0$ if and only if $f$ is smooth at $\eta$.
    
        \item Assume that $\eta$ is the generic point of a closed fiber of $f$. If $t$ is a uniformizer of $\sO_{C,f(\eta)}$, we can write $t=ux_0^n$ where $x_0$ is a uniformizer of $\sO_{\sX,\eta}$, $u\in \sO_{\sX,\eta}^\times$ and $n\geq 1$. Let $x_1,\dots,x_d\in \sO_{\sX,\eta}$ be such that $x_0,\dots,x_d$ is a differential basis of $\sO_{\sX,\eta}$, with dual basis $\partial_{x_0},\dots,\partial_{x_i}$ of $T_{\sX,\eta}$. Let $u_i:= \partial_{x_i}(u)$.
        We distinguish two cases:
            \begin{itemize}
                \item \emph{The characteristic $p$ does not divide $n$.}---
                Then we have
                        $$dt=nux_0^{n-1}dx_0+ x_0^n\sum_{i=0}^du_idx_i$$
                from which we deduce
                        $$\partial_{x_0}(t)=nux_0^{n-1}+u_0x_0^n,\quad 
                        \partial_{x_i}(t)=u_ix_0^n \quad (i\geq 1)$$
                and so
                        $$df(\partial_{x_0})=x_0^{n-1}(nu+u_0x_0)\partial_t, \quad 
                        df(\partial_{x_i})=u_ix_0^n\partial_t 
                        \quad (i\geq 1).$$
                As $nu+u_0x_0$ is invertible, we obtain an exact sequence
                    \begin{equation}\label{eqn:vertical_tame_divisor}
                         0\to T_{\sX/C,\eta} 
                        \to T_{\sX/k,\eta} 
                        \to f^*T_C\otimes \mathfrak{m}_{\sX,\eta}^{n-1}\to 0.
                    \end{equation}

                \item \emph{The characteristic $p$ 
                divides $n$.}---
                Then $dt=x_0^n\sum_{i=0}^du_idx_i$ and so
                $df(\partial_{x_i})=u_ix_0^n\partial_t$
                for $i\geq 0$. Whenever $u_i\neq 0$, we write $u_i=x_0^{a_i}v_i$ where $a_i\geq 0$ and $v_i\in \sO_{\sX,\eta}^\times$. Note that at least one $u_i$ is non-zero: for otherwise $u\in K(\sX)^p$, and so $t\in K(\sX)^p$, which implies that the field extension $K(C)\subset K(\sX)$ is not separable, contradicting geometric reducedness of the generic fiber of $f$ \cite[Corollary 2.5]{Schroeer_Fibrations_with_nonreduced_geom_fibers}.

                Let $\gamma_w(\eta)=\min\{a_i\mid i: \ u_i\neq 0\}$. Then we have an exact sequence
                    \begin{equation}\label{eqn:vertical_wild_divisor}
                         0\to T_{\sX/C,\eta} 
                        \to T_{\sX/k,\eta} 
                        \to f^*T_{C/k}\otimes \mathfrak{m}_{\sX,\eta}^{n+\gamma_w(\eta)}\to 0.
                    \end{equation}
            \end{itemize}
    \end{enumerate}

\begin{definition}[see {\cite[Definition 3.4]{Benozzo_CBF_in_pos_char}}]\label{def:wild_and_tame_divisors}
Suppose that $E=\overline{\{\eta\}}$ is an irreducible component of a closed fiber of $f$, and let $n=\ord_E\sX_{f(\eta)}$ be as above. We say that $E$ is a \emph{tame component} if $n$ is not divisible by the characteristic $p$; otherwise we say that $E$ is a \emph{wild component}.

In any case, we define the \emph{wild coefficient} $\gamma_w(E)$ of $E$ by:
        $$\gamma_w(E)=\begin{cases}
        \gamma_w(\eta) & \text{if }E\text{ is wild,}\\
        -1 & \text{otherwise}.
        \end{cases}$$
\end{definition}

\begin{example}
Suppose that $t=(1+x_0^c)x_0^p$ where $c$ is not divisible by $p$, and let $x_1,\dots,x_d$ be differential coordinates together with $x_0$. Then $d(1+x_0^c)=cx_0^{c-1}dx_0$ so $\gamma_w(E)=c-1$. Note also that $T_{\sX/C}$ is generated by $\partial_{x_1},\dots,\partial_{x_d}$, so $E$ is $T_{\sX/C}$-invariant (see \autoref{section:div_and_Frob_base_change}).
\end{example}

\begin{proposition}[see {\cite[Theorem 3.6]{Benozzo_CBF_in_pos_char}}]\label{prop:can_div_relative_tangent_sheaf}
Let $\sX\to C$ be as above. Then
        \begin{equation}\label{eqn:can_divisor_rel_log_foliation}
        K_{T_{\sX/C}}
        =
        K_{\sX/C} -R(f)-W(f)-\Gamma
        \end{equation}
where
        $$R(f)=\sum_{c\in C(k)}f^{*}(c)-\red(f^{*}(c)), 
        \quad 
        W(f)=\sum_{E \text{ vertical}}(\gamma_w(E)+1)E$$
and $\Gamma$ is an effective divisor that has no vertical component over $C$. Moreover, if the generic fiber of $f$ is geometrically normal, then $\Gamma=0$.
\end{proposition}   
\begin{proof}
Take the reflexified determinants of the short exact sequence defined by $df\colon T_{\sX/k}\to f^*T_{C/k}$, and use \autoref{eqn:horiz_divisor}, \autoref{eqn:vertical_tame_divisor} and \autoref{eqn:vertical_wild_divisor}.   
As for the ``Moreover'' part, see \cite[Step 2 of Theorem 3.6]{Benozzo_CBF_in_pos_char}.
\end{proof}
        
In particular, if $(\sX,\Delta)\to C$ is strongly locally stable, then we have $K_{T_{\sX/C}}=K_{\sX/C}$. 

\begin{remark}\label{rmk:alternative_def_wild_coeff}
The wild coefficients are uniquely determined by the fibration $\sX\to C$ and are independent of the choices of differential coordinates. More precisely, assume that $E$ is a wild divisor. If $t$ is a local parameter of $\sO_{C,f(E)}$ and $s$ a local parameter of $\sO_{\sX,E}$, then $t=us^n$ with $n=\ord_E \sX_{f(E)}$ and
        $$\gamma_w(E)   =\length_{\sO_{\sX,E}}\frac{\left(\sO_{\sX,E}\cdot du\right)^\text{sat}}{\sO_{\sX,E}\cdot du}$$
where the saturation is taken inside $\Omega^1_{\sO_{\sX,E}/k}$.        
More generally,
for any ideal $\mathfrak{a}$ of $\sO_{\sX,E}$, define the ideal
    $$\mathfrak{d}(\mathfrak{a})
     =\left\{\sum_i\partial_i(a_i)\mid a_i\in \mathfrak{a}, \           \partial_i\in \Der_k(\sO)\right\}
    \lhd \sO_{\sX,E}.$$
Then
    $$\gamma_w(E)            =\ord_E\mathfrak{d}\left(\fm_{C,f(E)}\sO_{\sX,E}\right)
    - \ord_E\sX_{f(E)}.$$
In particular, the wild coefficient of $E$ can be computed after passing to the completion or to the (strict) henselization of $\sO_{\sX,E}$.
\end{remark}

\begin{lemma}\label{lemma:wild_coeff_not-1_mod_p}
If $E$ is a vertical wild divisor, then $\gamma_w(E)\not\equiv -1$ modulo $p$.
\end{lemma}
\begin{proof}
Choose a uniformizer $x_0\in \sO$, and let us go to the completion $\widehat{\sO}\cong \kappa\llbracket x_0\rrbracket$, where $\kappa$ is the residue field of $\sO$. We expand $u$ as
        $$u=v^p+\sum_{i\not\in p\bZ}a_ix_0^i, \quad 
        v\in \kappa\llbracket x_0\rrbracket, \ a_i\in \kappa.$$
Then 
        $$du=\sum_{i\not\in p\bZ}ia_ix_0^{i-1}dx_0
        +\sum_{i\not\in p\bZ}x_0^ida_i.$$
So the biggest power of $x_0$ dividing $du$ is $x_0^{i-1}$ for $i\notin p\bZ$, which yields the statement.
\end{proof}

\section{Foliations induced by Frobenius base-change}\label{section:div_and_Frob_base_change}
In this section we gather some preliminary results about $1$-foliations induced by base-change along Frobenius morphisms.

Let $(\mathbf{0}\in C)$ be a germ of smooth curve over $k$, and let $f\colon (X,\Delta)\to C$ be a family of pairs.
Let $F_{C/k}\colon C\to C$ be the $k$-linear Frobenius. We let $X^{(1)}=X\times_{(f,F_{C/k})}C$. If $F_{X^{(1)}/k}\colon X^{(1)}\to X^{(1)}$ is the $k$-linear Frobenius of the latter, then we have a commutative diagram
        \begin{equation}\label{eqn:notations_for_base_change}
        \begin{tikzcd} 
            X^{(1)}\arrow[d, "f^{(1)}" left] \arrow[rr, "(F_{C/k})_X" below] 
            \arrow[rrr, bend left, "F_{X^{(1)}/k}"] &&
            X \arrow[r, "F_{X/C}" below]\arrow[d, "f" left] 
            & X^{(1)} \arrow[dl, "f^{(1)}", bend left] \\
            C \arrow[rr, "F_{C/k}" below] && C &
        \end{tikzcd}
        \end{equation}
where all horizontal arrows are purely inseparable morphisms of height one. We will write $\Delta^{(1)}=\Delta\times_{C,F_{C/k}}C=(F_{C/k})_X^*\Delta$, so that $(X^{(1)},\Delta^{(1)})\to C$ is also a family of pairs.

\begin{notation}
In case $X$ and $X^{(1)}$ are normal, we let
    \begin{itemize}
        \item $\mathfrak{F}$ be the rank one $1$-foliation of $X^{(1)}$ defined by $(F_{C/k})_X$, i.e.\
            $$\mathfrak{F}=
            \text{saturation of }\ker\left(d(F_{C/k})_X \colon T_{X^{(1)}}\to (F_{C/k})_X^*T_X\right)$$ 
        (it is the same $\mathfrak{F}$ that featured in the proof of \autoref{lemma:Q_cartier_cond_base_changes});
        \item the corank one $1$-foliation of $X$ 
        defined by $F_{X/C}$ is equal to $T_{X/C}$, according to the following lemma.
    \end{itemize}
\end{notation}

\begin{lemma}[{see \cite[Lemma 3.17]{Benozzo_CBF_in_pos_char}}]\label{lemma:relative_foliation_and_central_fiber}
Assume that $X$ and $X^{(1)}$ are normal. Then
$T_{X/C}$, defined as in \autoref{eqn:relative_log_tangent}, is a $1$-foliation which defines the morphism $F_{X/C}$ and leaves the central fiber $X_\mathbf{0}$ invariant.
\end{lemma}
\begin{proof}
We can regard $(F_{C/k})_X\colon X^{(1)}\to X$ as the normalization of $X$ in the field $L$, compositum of $K(C)^{1/p}$ and $K(X)$. Then the question is to identify the set of derivations of $K(X)$ defining the sub-extension $L^p\hookrightarrow K(X)$: as $L^p$ is the compositum of $K(X)^p$ and $K(C)$, these derivations are exactly the ones which annihilate $K(C)$. The associated reflexive sub-sheaf of $T_X$ is $T_{X/C}$, which by definition leaves $X_\mathbf{0}$ invariant.
As noticed after \autoref{eqn:relative_log_tangent}, it is a $1$-foliation.
\end{proof}

\begin{lemma}\label{lemma:can_div_Frob_is_Cartier}
Assume that $X^{(1)}$ is normal. Then the divisor $K_\mathfrak{F}$ is Cartier.
\end{lemma}
\begin{proof}
This was observed during the proof of \autoref{lemma:Q_cartier_cond_base_changes}: the argument is given in the first paragraph of Case 1, and it holds under our current assumptions. 
\end{proof}

\begin{lemma}\label{lemma:log_rel_tangent_is_Q-Gor}
Assume that $X^{(1)}$ is normal. Then $K_{T_{X/C}}=K_{X/C}$.
\end{lemma}
\begin{proof}
Since $X^{(1)}_{K(C)}$ is normal we conclude from Serre's criterion of normality, \cite[6.7.1]{EGA_IV.2} and \cite[22.5.8]{EGA_IV.1} that $X_{K(C)}$ is geometrically normal. Therefore $K_{T_{X/C}}=K_{X/C}$ by \autoref{prop:can_div_relative_tangent_sheaf}. 
\end{proof}

Let $E$ be a prime divisor over $X$, appearing on a normal birational model $Y\to X$. Let $Y^{(1),\nu}$ be the normalization of $Y\times_{C,F_{C/k}}C$: it is endowed with a birational morphism $Y^{(1),\nu}\to X^{(1)}$ that makes the relevant diagram commutative. Since the projection $Y^{(1),\nu}\to Y$ is also the normalization of $Y$ in the function field of $X^{(1)}$, it is a universal homeomorphism. So there is a unique prime divisor $E^{(1)}$ on $Y^{(1),\nu}$ mapping to $E$. It is easily seen that:

\begin{lemma}
The function
        $$\{\text{prime divisor over }X\}\longrightarrow
        \{\text{prime divisor over }X^{(1)}\},
        \quad E\mapsto E^{(1)}$$
is a bijection.
\hfill \qedsymbol
\end{lemma}

\begin{definition}\label{def:deltas}
Let $E$ be a prime divisor over $X$ appearing on a birational model $Y$. Then the $k$-linear Frobenius of $Y^{(1),\nu}$ factors as
        $$F_{Y^{(1),\nu}/k}\colon 
        Y^{(1),\nu}\overset{h}{\longrightarrow}
        Y \overset{g}{\longrightarrow}
        Y^{(1),\nu}$$
where $h$ is the quotient by (the birational pullback of) $\mathfrak{F}$, and $g$ is the quotient by (the birational pullback of) $T_{X/C}$.
By abuse of notation, we say $E$ is \emph{$T_{X/C}$-invariant} (resp.\ \emph{$\mathfrak{F}$-invariant}) if it is invariant with respect to the birational pullback of $T_{X/C}$ (resp.\ of $\mathfrak{F}$).
We write
    $$\delta_{T_{X/C}}(E)=\begin{cases}
        1 & \text{if }E\text{ is }T_{X/C}\text{-invariant,}\\
        p & \text{otherwise}
    \end{cases}
    \quad  \quad 
    \delta_\mathfrak{F}(E^{(1)})=\begin{cases}
        1 & \text{if }E^{(1)}\text{ is }\mathfrak{F}\text{-invariant,}\\
        p & \text{otherwise}.
    \end{cases}$$
\end{definition}

\begin{lemma}\label{lemma:inv_dichotomy}
The divisor $E$ is $T_{X/C}$-invariant if and only if $E^{(1)}$ is not $\mathfrak{F}$-invariant.
In particular, $\delta_{T_{X/C}}(E)\cdot \delta_\mathfrak{F}(E^{(1)})=p$ holds.
\end{lemma}
\begin{proof}
Let us use the notation of \autoref{def:deltas}. As 
$F_{Y^{(1),\nu}/k}^*E^{(1)}=pE^{(1)}$, exactly one of the following happens: $(g^*E^{(1)}=E, h^*E=pE^{(1)})$, or $(g^*E^{(1)}=pE, h^*E=E^{(1)})$. So the result follows from \cite[Lemma 4.2.2]{Posva_Singularities_quotients_by_1_foliations}.
\end{proof} 

\begin{lemma}\label{lemma:orders_central_fibers}
Let $E$ be a divisor appearing on a normal birational model $Y$ of $X$. Then:
        $$\ord_EY_\mathbf{0}= 
        \delta_{T_{X/C}}(E)\cdot \ord_{E^{(1)}}Y^{(1),\nu}_\mathbf{0} 
        = \frac{p}{\delta_\mathfrak{F}(E^{(1)})}\cdot 
        \ord_{E^{(1)}}Y^{(1),\nu}_\mathbf{0}.
        $$
\end{lemma}
\begin{proof}
In view of \autoref{lemma:inv_dichotomy}, it suffices to prove the first equality. Let $h\colon Y\to Y^{(1),\nu}$ be the quotient by (the birational pullback of) $T_{X/C}$. By \autoref{lemma:relative_foliation_and_central_fiber} we have
$h^*Y_\mathbf{0}^{(1),\nu}=Y_\mathbf{0}$, while 
$h^*E^{(1)}=\delta_{T_{X/C}}(E)\cdot E$ by \cite[Lemma 4.2.2]{Posva_Singularities_quotients_by_1_foliations}. So around the generic point of $E$ we have
        $$\ord_E(Y_\mathbf{0})\cdot E=
        Y_\mathbf{0}=
        h^*Y_\mathbf{0}^{(1),\nu}=
        \delta_{T_{X/C}}(E)\cdot \ord_{E^{(1)}}(Y_\mathbf{0}^{(1),\nu})\cdot E$$
as claimed.
\end{proof}

\begin{example}\label{example:non_inv_wild_div}
We indicate a simple way to produce non-$T_{X/C}$-invariant wild vertical divisors. Take a general regular closed point of $X_\mathbf{0}$, pick local coordinates $x,y,z_2\dots,z_d$ such that $t=x$ where $t$ is a local parameter of $\sO_{C,\mathbf{0}}$. We blow-up $(x,y)$ with weights $(p,1)$: the schematic chart with coordinates $x',y',z_2,\dots,z_d$ is given by $x=x'y'^p$, $y=y'$. Then $t=x'y'^p$, so the relative tangent sheaf over $C$ is generated by $\partial_{y'}$ and the $\partial_{z_i}$s. In particular, the exceptional divisor $E=V(y')$ is wild and non-$T_{X/C}$-invariant (with wild coefficient $0$).---
See \autoref{e-nonzero_wildcoeffs} for examples of non-$T_{X/C}$-invariant wild vertical divisors with non-zero wild coefficient.
\end{example} 

\section{Descent}\label{section:descent}
In this section we study descent of local stability along Frobenius morphisms and Artin--Schreier covers.
Throughout the section, we let $(\mathbf{0}\in C)$ be a germ of smooth curve over $k$ and $f\colon (X,\Delta)\to C$ be a family of pairs. 

\subsection{Descent along Frobenius}
We use the notations introduced in \autoref{section:div_and_Frob_base_change}.

\begin{proposition}\label{prop:discrep_qt_by_F}
Let $E^{(1)}$ be a divisor over $X^{(1)}$, appearing on a birational model $\mu'\colon Y^{(1),\nu}\to X^{(1)}$. Assume that $X$ and $X^{(1)}$ are normal, and that $K_X+\Delta$ is $\bQ$-Cartier. Then
        \begin{eqnarray*}a(E;X,\Delta+X_\mathbf{0})
        &=&
        \frac{1}{\delta_\mathfrak{F}(E^{(1)})}\cdot 
        \left[
        a(E^{(1)};X^{(1)},\Delta^{(1)}+X_\mathbf{0}^{(1)})
        \right.\\
        && \left. \quad \quad \quad \quad \quad \quad 
        +(p-1)\cdot \left(
        a(E^{(1)};\mathfrak{F})-\ord_{E^{(1)}}Y^{(1),\nu}_\mathbf{0}
        \right)
        \right].
        \end{eqnarray*}
\end{proposition} 
\begin{proof}
By \autoref{lemma:Q_cartier_cond_base_changes} the divisor $K_{X^{(1)}}+\Delta^{(1)}$ is $\bQ$-Cartier.
We have a commutative diagram
        $$\begin{tikzcd}
            Y^{(1),\nu} \arrow[r, "q'"] \arrow[d, "\mu'"] &
            Y \arrow[d, "\mu"] \\
            X^{(1)} \arrow[r, "q"] & X
        \end{tikzcd}$$
where $q'$ is the quotient by $\mu'^*\mathfrak{F}$ and $\mu$ is birational. We have $\Delta^{(1)}=q^*\Delta$ by definition, and so $q'^*(\mu^{-1}_*\Delta)=\mu'^{-1}_*\Delta^{(1)}$. Recall that $K_\mathfrak{F}$ is Cartier. Using \cite[Proposition 4.2.3]{Posva_Singularities_quotients_by_1_foliations} we compute on one hand
        \begin{eqnarray*} 
        \mu'^*q^*(K_X+\Delta)&=&
        K_{Y^{(1),\nu}}+\mu'^{-1}_*\Delta^{(1)}+(p-1)K_{\mu'^*\mathfrak{F}} \\
        && \quad
        -\left[(p-1) a(E^{(1)};\mathfrak{F})\cdot E^{(1)}
        +a(E^{(1)};X^{(1)},\Delta^{(1)})
        \right]
        \end{eqnarray*}
and on the other hand
        $$q'^*\mu^*(K_X+\Delta)=
        K_{Y^{(1),\nu}}+\mu'^{-1}_*\Delta^{(1)}+(p-1)K_{\mu'^*\mathfrak{F}}
        -a(E;X,\Delta)\delta_\mathfrak{F}(E^{(1)})\cdot E^{(1)}.$$
The two pullbacks are equal, so we find
        $$a(E;X,\Delta)=\delta_\mathfrak{F}(E^{(1)})^{-1}\cdot
        \big[
        a(E^{(1)};X^{(1)},\Delta^{(1)})+(p-1)a(E^{(1)};\mathfrak{F})
        \big].$$
To throw in the central fiber, we subtract $\ord_EY_\mathbf{0}$ on both sides. We have seen in \autoref{lemma:orders_central_fibers} that $\ord_EY_\mathbf{0}=p\cdot\delta_\mathfrak{F}(E^{(1)})^{-1} \cdot\ord_{E^{(1)}}Y^{(1),\nu}_\mathbf{0}$, and so by rearrangement we find the desired formula.
\end{proof}

\begin{proposition}\label{prop:descent_along_pur_insep}
Assume that $f^{(1)}\colon (X^{(1)},\Delta^{(1)})\to C$ is locally stable (resp.\ strongly locally stable).
Then $f\colon (X,\Delta)\to C$ is locally stable (resp.\ strongly locally stable).
\end{proposition}
\begin{proof}
Note that $X$ is normal, since $X^{(1)}$ is normal and $X^{(1)}\to X$ is faithfully flat \cite[033G]{Stacks_Project}. 

We start by establishing that $K_X+\Delta$ is $\bQ$-Cartier. By \autoref{lemma:relative_foliation_and_central_fiber}, the log adjunction formula \cite[Proposition 4.2.3]{Posva_Singularities_quotients_by_1_foliations} and the fact that $F_{X/C}^*\Delta^{(1)}=p\Delta$, we find
        $$F_{X/C}^*(K_{X^{(1)}}+\Delta^{(1)})=
        (K_X+\Delta)+(p-1)\cdot (K_{T_{X/C}}+\Delta).$$
Since $X^{(1)}$ is normal, we have $K_{T_{X/C}}+\Delta=K_X+\Delta-f^*K_C$ by \autoref{lemma:log_rel_tangent_is_Q-Gor}. Therefore we get
        $$F_{X/C}^*(K_{X^{(1)}}+\Delta^{(1)})+(p-1)\cdot f^*K_C=
        p\cdot (K_X+\Delta)$$
which shows that $K_X+\Delta$ is $\bQ$-Cartier.

Next we investigate the local stability of $f\colon X\to C$.
By \autoref{lemma:localiz_on_the_base} we may assume that $C=\Spec(\sO)$ is the spectrum of a DVR with closed point $\mathbf{0}$. We denote the closed fiber of $f^{(1)}$ by $X^{(1)}_\mathbf{0}$. 
The strongly locally stable case follows from the locally stable one, since the geometric fibers of $f$ and $f^{(1)}$ are the same, so we prove the locally stable case.
Let $E$ be an exceptional divisor over $X$ appearing on a normal birational model $Y$, and $E^{(1)}$ the corresponding exceptional divisor over $X^{(1)}$. To compute $a(E;X,\Delta+X_\mathbf{0})$, we distinguish two cases.

\medskip
\textsc{Case 1: $E$ dominates $C$.}
In this case we can localize over the function field $\kappa$ of $C$ to compute the discrepancy, and the base-change $F_{C/k}\colon C'=C\to C$ becomes the field extension $\kappa\hookrightarrow \kappa'=\kappa^{1/p}$. Notice that $Y_{\kappa'}$ is $S_2$ and generically reduced (as $X_{\kappa'}$ is generically reduced), so it is reduced. The morphism $h\colon Y_{\kappa'}\to Y$ is Gorenstein, so $Y_{\kappa'}$ is Gorenstein in codimension one and its canonical divisor is the pullback of $K_Y$. Let $Y'$ be the normalization of $Y_{\kappa'}$, with induced morphism $\varphi=h\circ\nu\colon Y'\to Y$. Notice that $E^{(1)}$ is a divisor on $Y'$ with support $\varphi^*E$.

To conclude this case, we follow the argument given in \cite[Proposition 2.3]{Bernasconi_Tanaka_On_dP_fibrations_in_pos_char}. By \cite[Theorem 3.1]{LMMPJoeZsolt} we have 
        $$\nu^*K_{Y_{\kappa'}}=K_{Y'}-G$$
where $G\geq 0$. So for any $\bQ$-divisor $F$ on $Y$,
        $$\varphi^*(K_Y+F)=K_{Y'}+\varphi^*F-G.$$
It follows that 
    $$a(E;X_\kappa, \Delta_\kappa)\geq a(E^{(1)};X_{\kappa'},\Delta_{\kappa'})\geq -1,$$
as desired.

\medskip
\textsc{Case 2: $E$ does not dominate $C$.} 
We assume for the rest of the proof that $E$ maps to $\mathbf{0}$. To treat this case, we use $1$-foliation techniques.
Recall that the morphism $(F_{C/k})_X\colon X^{(1)}\to X$ is the quotient by $\mathfrak{F}$.



We may localize $X$ at the center of $E$: then $X\times_CC^{(1)}$ becomes the localization of $X^{(1)}$ at the center of $E^{(1)}$. As in the proof of \autoref{lemma:Q_cartier_cond_base_changes}, we can write $X=\Spec(R)$ where
        $$R=\frac{\sO[x_1,\dots,x_N]}{(f_1,\dots,f_r)}$$
and the $f_i$s are polynomials in the variables $x_1,\dots,x_N$ and coefficients in $\sO$. Thus
        $$R^{(1)}=R\otimes_{\sO}\sO^{1/p}=\frac{\sO^{1/p}[x_1,\dots,x_N]}{(f_1,\dots,f_r)},$$
where $\sO^{1/p}$ is defined as in the proof of \autoref{lemma:Q_cartier_cond_base_changes}.
Let $t$ be a uniformizer of $\sO^{1/p}$. Consider the derivation $D_t$ of $\sO^{1/p}[x_1,\dots,x_N]$ which sends $t$ to $1$ and the $x_i$s to $0$. The coefficients of the $f_j$s belong to $\sO$, hence they are $p$-th powers in $\sO^{1/p}$: so $D_t(f_j)=0$ for every $j$. This implies that $D_t$ descends to a derivation of $R^{(1)}$ which we denote $\partial_t$. 
Then $\partial_t$ generates $\mathfrak{F}$.
In particular $K_\mathfrak{F}$ is Cartier, and $\sO(-K_\mathfrak{F})$ is generated by $\partial_t$.

Now, let $T$ be the local ring along $E^{(1)}$ on a normal birational model $\phi\colon Y^{(1),\nu}\to X^{(1)}$ where the divisor appears. It is a DVR, say with uniformizer $s$. Since $E^{(1)}$ maps to $\mathbf{0}$, we can write $t=us^n$ in $T$ where $n=\ord_{E^{(1)}}Y^{(1),\nu}_\mathbf{0}>0$ and $u\in T^\times$. By definition $a=a(E^{(1)};\mathfrak{F})$ is the smallest integer with the property that
        $$\partial_t=s^{-a}\psi \quad \text{for some } \psi\in \Der(T).$$
We distinguish two cases:
    \begin{itemize}
        \item \emph{$n$ and $p$ are coprime.} Then we can pass to an \'{e}tale extension of $T$ (or to its henselization) and assume that $u=1$. We have
                $$1=\partial_t(s^n)
                =ns^{n-1}\partial_t(s).$$
        Therefore $1=ns^{-a+n-1}\psi(s)$. As $\psi(s)\in T$ we must have $a\geq n-1$, with equality if and only if $\psi(s)\notin(s)$, which is equivalent to $E^{(1)}$ not being $\mathfrak{F}$-invariant.
        By \autoref{prop:discrep_qt_by_F} and the fact that $a(E^{(1)};X^{(1)},\Delta^{(1)}+X^{(1)}_\mathbf{0})\geq -1$, we obtain that $a(E;X,\Delta+X_\mathbf{0})\geq -1$.

        \item \emph{$p$ divides $n$.} We have $1=\partial_t(us^n)=s^n\partial_t(u)$, so $1=s^{-a+n}\psi(u)$. So $a(E^{(1)};\mathfrak{F})\geq n$ in any case, and as before we deduce $a(E;X,\Delta+X_\mathbf{0})\geq -1$.
    \end{itemize}
This concludes the proof.
\end{proof}

As an application, we prove that if local stability is lost after some purely inseparable base-change, then it is lost for all further purely inseparable base-changes. To formulate it precisely, we extend \autoref{eqn:notations_for_base_change}:
for $e\geq 1$ we let $X^{(e)}=X\times_{(f,F_{C/k}^e)}C$ be the base-change of $X\to C$ by the $e$-th iterate of the $k$-linear Frobenius of $C$. It is easy to see that we have a commutative diagram
        $$\begin{tikzcd}
        \dots \arrow[r]&  X^{(e+1)} \arrow[d, "f^{(e+1)}"] \arrow[r] & X^{(e)} \arrow[r]\arrow[d, "f^{(e)}"] & \dots \arrow[r] &
        X^{(1)} \arrow[d, "f^{(1)}"] \arrow[r] & X \arrow[d, "f"] \\
        \dots \arrow[r] & C \arrow[r, "F_{C/k}"] & 
        C \arrow[r,"F_{C/k}"] & \dots \arrow[r] & C \arrow[r, "F_{C/k}"] & C
        \end{tikzcd}$$
where every square is Cartesian. 
We let $\Delta^{(e)}$ be the pullback of $\Delta$ through $X^{(e)}\to X$, as in \autoref{def:base_change_divisors}.

\begin{proposition}\label{prop:loc_stability_lost_forever}
Suppose that the closed fibers of $X\to C$ are demi-normal with only separable nodes, that the generic one is geometrically normal, and that $K_X+\Delta$ is $\bQ$-Cartier. If the base-change $(X^{(e_0)},\Delta^{(e_0)})\to C$ is \emph{not} locally stable for some $e_0\geq 1$, then $(X^{(e)},\Delta^{(e)})\to C$ is not locally stable for any $e\geq e_0$.
\end{proposition}
\begin{proof}
Thanks to \autoref{prop:demi_normality_loc_stable_family} and \autoref{lemma:Q_cartier_cond_base_changes}, we see that every $X^{(e)}$ is normal and every $K_{X^{(e)}}+\Delta^{(e)}$ is $\bQ$-Cartier. Then the contraposition of the statement follows by repeated applications of \autoref{prop:descent_along_pur_insep}.
\end{proof}

\subsection{Descent along Artin--Schreier extensions}
Let $C'\to C$ be a finite flat morphism of smooth curves, such that the extension of function fields $K(C)\hookrightarrow K(C')$ is Galois of degree $p$.
As usual, we write $f'=f\times_CC'\colon (X',\Delta')=(X,\Delta)\times_CC'\to C'$.

\begin{lemma}\label{lemma:action_on_normalization}
There is a $C$-linear $\bZ/p$-action on $C'$ with geometric quotient $C$.
\end{lemma}
\begin{proof}
By the Artin--Schreier theorem \cite[VI, Theorem 6.4]{Lang_Algebra} there is a $\bZ/p$-action on $K(C')$ with field of invariants $K(C)$; this action is necessarily $K(C)$-linear. Let $\sigma$ be a generator of the action: we claim that it extends to an $\sO_C$-linear automorphism of $\sO_{C'}$. The extension is necessarily unique, so we may assume that $C=\Spec(A)$ is affine. Then $C'=\Spec(A')$, and pick $x\in A'$. Since $A'$ is the integral closure of $A$ in $K(C')=\Frac(A')$, there is a monic equation
        $$x^n+\sum_{i=0}^{n-1}a_ix^i=0, \quad 
        a_i\in A.$$
Applying $\sigma$, we find $\sigma(x)^n+\sum_i a_i\sigma(x)^i=0$ and thus $\sigma(x)\in A'$. 

Returning to the general situation, we obtain that the $\bZ/p$-action extends to $\sO_{C'}$. The sub-sheaf of invariants is a sheaf of integrally closed $\sO_C$-algebras with generic stalk $K(C)$, so it must be equal to $\sO_C$.
\end{proof}

Observe that $G=\bZ/p$ acts on $X'=X\times_CC'$ via the second factor, making $f'$ equivariant, and that the projection $q\colon X'\to X$ is the geometric quotient.

\begin{lemma}\label{lemma:descent_Q-Cartier_AS_ext}
Assume that $X'$ is normal and that $K_{X'}+\Delta'$ is $\bQ$-Cartier. Then $X$ is normal and $K_X+\Delta$ is $\bQ$-Cartier.
\end{lemma}

\begin{proof}
Since $q$ is faithfully flat and $X'$ is normal, $X$ is also normal \cite[033G]{Stacks_Project}.
As seen in Case 2 of the proof of \autoref{lemma:Q_cartier_cond_base_changes}, for any positive integer $m>0$ we have
        $$q^*m(K_X+\Delta)=m(K_{X'}+\Delta')-mR$$
where $R$ is the ramification divisor of the $G$-action on $X'$.  Therefore, if $U\subset X$ is the big open subset where $K_X+\Delta$ is $\bQ$-Cartier and $m$ is divisible enough, we find that $\sO_{X'}(m(K_{X'}+\Delta')-mR)$ is $G$-linearized over $U'=q^{-1}(U)$. Since $q$ is finite and flat, $U'$ is big in $X'$: hence, as $\sO(m(K_{X'}+\Delta')-mR)$ is reflexive, its $G$-linearization over $U'$ extends to a $G$-linearization over $X'$.
Replacing $m$ by a multiple if necessary, we may also assume that $\mathscr{L}=\sO_{X'}(m(K_{X'}+\Delta')-mR)$ is invertible. 

We define the presheaf $\mathscr{L}^G$ of $G$-invariant sections of $\mathscr{L}$ as follows: if $V\subset X$ is open, then $G$ acts on the $\sO_{X'}(q^{-1}V)$-module $\Gamma(q^{-1}V,\mathscr{L})$. We let $\Gamma(V,\mathscr{L}^G)$ be the sub-group of $G$-invariants elements of $\Gamma(q^{-1}V,\mathscr{L})$: it is naturally an $\sO_X(V)$-module. We leave it to the reader to check that $\mathscr{L}^G$ is a quasi-coherent $\sO_X$-module.

\begin{claim}
The module-theoretic pullback $q^*\sO_X(m(K_X+\Delta))$ is $S_1$, and the $G$-invariant sub-module $\mathscr{L}^G=\sO_X(m(K_X+\Delta))$.
\end{claim}
\begin{proof}\renewcommand{\qedsymbol}{$\lozenge$}
On an integral quasi-projective scheme, a coherent module is $S_1$ if and only if it embeds locally into a coherent locally free module \cite[Lemma 1.5]{Hartshorne_Generalized_div_on_Gor_schemes}. This applies to $\sO_X(m(K_X+\Delta))$ since it is even $S_2$. As $q$ is flat, the module-theoretic pullback $q^*(-)$ preserves injective maps, and so we obtain that $q^*\sO_X(m(K_X+\Delta))$ is $S_1$.

It is clear that $\mathscr{L}^G=\sO_X(m(K_X+\Delta))$ holds over $U$. To prove that the equality is also true over $X$, it suffices to show that $\mathscr{L}^G$ is reflexive: that is, whenever $V$ is open and $Z\subset V$ is closed of codimension at least $2$, the restriction map $\Gamma(\mathscr{L}^G,V)\to \Gamma(\mathscr{L}^G,V\setminus Z)$ is bijective. We have
$\mathscr{L}^G\hookrightarrow q_*\mathscr{L}$ and the target is $S_2$ \cite[Proposition 5.4]{Kollar_Mori_Birational_geometry_of_algebraic_varieties}, hence reflexive. 
In particular, sections of $\mathscr{L}^G$ defined on $V\setminus Z$ can be extended uniquely to elements of $q_*\mathscr{L}(V)$. But since being $G$-invariant is a property that can be checked generically, it is clear that these extensions belong to $\mathscr{L}^G$, whence $\mathscr{L}^G$ is indeed reflexive.
\end{proof}

Now take any point $x'\in X'$, and let $V'$ be an open neighbourhood of $x'$ where $\sO_{X'}(m(K_{X'}+\Delta'))$ is freely generated by one section $s$. Replacing $V'$ by $\cap_{\sigma\in G}\sigma(V')$, we may assume that $V'=q^{-1}(V)$ for an open neighbourhood $V$ of $q(x')$. Then each $\sigma(s)$ is a section of $\mathscr{L}$ over $V'$ which generates. Replacing $s$ by $\bigotimes_\sigma \sigma(s)$ and $m$ by $pm$, we may therefore assume that $s$ is $G$-invariant. Hence $s$ descends to a section of the $\sO_X$-module $\mathscr{L}^G=\sO_X(m(K_X+\Delta))$ over $V$. 

Consider the induced map 
$s\colon \sO_V\to \sO_X(m(K_X+\Delta))|_V$. We claim it is an isomorphism. Since $q$ is faithfully flat, it suffices to show this after pulling back by $q$. We have the induced maps
        $$\begin{tikzcd}
        \sO_{V'} \arrow[r, "q^*s"] &
        q^*\left(\sO_X(m(K_X+\Delta))|_V\right) \arrow[r, "\alpha"] &
        q^{[*]}\left(\sO_X(m(K_X+\Delta))|_V\right)
        =\mathscr{L}|_{V'}
        \end{tikzcd}$$
where $\alpha$ is the double-dual map. The composition $\sO_{V'}\to \mathscr{L}|_{V'}$ corresponds to the section $s\in \Gamma(V', \mathscr{L})$: it is surjective, since $s$ generates $\mathscr{L}$ over $V'$. Since $q^*\sO_X(m(K_X+\Delta))$ is $S_1$, the map $\alpha$ is injective \cite[Lemma 1.5]{Hartshorne_Generalized_div_on_Gor_schemes}. So we find that $q^*s$ is surjective. The map $q^*s$ is also injective, because $q^*\sO_X(m(K_X+\Delta))$ is $S_1$.

So we have found that $K_X+\Delta$ is $\bQ$-Cartier around $q(x')$. As $q$ is surjective, the proof is complete.
\end{proof}

\begin{proposition}\label{prop:descent_loc_stab_along_AS}
With the above notation, assume that $f'\colon (X',\Delta')\to C'$ is locally stable (resp.\ strongly locally stable). Then $f\colon (X,\Delta)\to C$ is locally stable (resp.\ strongly locally stable).
\end{proposition}
\begin{proof}
By \autoref{lemma:descent_Q-Cartier_AS_ext}, $X$ is normal and $K_X+\Delta$ is $\bQ$-Cartier. The geometric fibers of $f$ and $f'$ are the same, so it suffices to compare discrepancies to prove the statement. 

We start with some reductions. By \autoref{lemma:localiz_on_the_base} we may assume that $C=\Spec(\sO)$ is the spectrum of a DVR with algebraically closed residue field $k$ and closed point $\mathbf{0}$. Then $C'=\Spec(\sO')$ is the spectrum of a semi-local ring, and by \autoref{lemma:action_on_normalization} 
there is an action of $\bZ/p$ on $\sO'$ with invariant ring $\sO$. Since $\sO$ is local, this action is transitive on the set of maximal ideals of $\sO'$ \cite[V, \S 2, Th\'{e}or\`{e}me 2]{Bourbaki_AC_V-VII}.
If $\sO'$ is \emph{not} local, then the action has no inertia (because $\bZ/p$ is simple), and so $\sO'$ is \'{e}tale over $\sO$ \cite[IX, \S 1, Corollaire 1]{Raynaud_Anneaux_locaux_henseliens}. In that case the descent result follows from \cite[Proposition 2.15]{Kollar_Singularities_of_the_minimal_model_program}. So we may assume that $\sO'$ is local.
Then, let $\sO^h$ be the henselization of $\sO$. The fiber product $\sO'\otimes_\sO\sO^h$ is the henselization of $\sO'$ by \cite[0CBF]{Stacks_Project}.
The group $\bZ/p$ acts on the first factor, with invariant ring $\sO^h$. So we may replace $\sO$ and $\sO'$ by their respective henselizations. To summarize, we may assume that $C=\Spec(\sO)$ and $C'=\Spec(\sO')$ are the spectra of henselian DVRs with algebraically closed residue fields, that $G=\bZ/p$ acts on $C'$ with ramification at the closed point, and that $\sO$ is the ring of $G$-invariants.

In what follows, we use the notions and notation introduced in \autoref{appendix:Z/p_qt}.
Since the residue field at the closed point $\mathbf{0}'$ of $C'$ is algebraically closed, the ramification at $\mathbf{0}'$ is wild. 

\begin{claim}\label{claim:ram_theory}
Let $t$ be a uniformizer of $\sO'$. There is an integer $\mathfrak{r}\geq 1$ such that $\ord_{\mathbf{0}'}I(t)=\mathfrak{r}+1$ and $[C'^G]=(\mathfrak{r}+1)\cdot \mathbf{0}'$.
\end{claim}
\begin{proof}\renewcommand{\qedsymbol}{$\lozenge$}
The so-called higher ramification groups
        $$G_i=\{g\in G\mid g \text{ acts trivially on }\sO'/(t^{i+1})\}, \quad i\geq -1,$$
form a descending chain of normal sub-groups of $G$. Since $G$ is cyclic of order $p$, there is $\mathfrak{r}\geq -1$ such that
        $$G_{-1}=G_0=\dots=G_{\mathfrak{r}}=G, 
        \quad \text{and }
        G_{\mathfrak{r}+l}=\{1\} \text{ for all }l\geq 1.$$
This implies that $[C'^G]=(\mathfrak{r}+1)\cdot \mathbf{0}'$. It is easily seen that $\mathfrak{r}+1=\inf_{x\in \sO'}\{\ord_{\mathbf{0}'}I(x)\}$, and since the Galois action is wild at $\mathbf{0}'$ we get $\mathfrak{r}+1=\ord_{\mathbf{0}'}I(t)$. 

It remains to show that $\mathfrak{r}\geq 1$, equivalently that the $G$-action on $\sO'/(t^2)$ is trivial. This $G$-representation splits as $k\oplus (t)/(t^2)$, and the desired triviality follows from \autoref{lemma:wild_ramification_triv_resid_action}.
\end{proof}

\begin{claim}\label{cor:AS_ext_fixed_locus}
$[X'^G]=(\mathfrak{r}+1)X'_{\mathbf{0}'}$.
\end{claim}
\begin{proof}\renewcommand{\qedsymbol}{$\lozenge$}
By definition of the action on $X'$ we have $X'^G=f'^{-1}(C'^G)$ so the statement follows from the previous claim.
\end{proof}

Observe that the ramification index $e(-)$ of every component of $\Delta'$ is $1$, since the $G$-action is unramified on the generic fiber of $f'$ and every component of $\Delta$ dominates $C$. In particular, $\Delta'$ agrees with the divisor $\Delta_{X'}$ defined in the way explained in the paragraph before \autoref{prop:discrepancy_AS_qt}.
Moreover $e(X'_{\mathbf{0}'})=e(\mathbf{0'})=p$.

Let $\varphi\colon Y\to X$ be a birational normal model, $Y'$ be its normalization in $K(X)$ with induced birational morphism $\varphi'\colon Y'\to X'$, $D$ be a $\varphi$-exceptional divisor and $E$ a prime divisor of $Y'$ dominating $D$.
If the $G$-action is unramified at the generic point of $E$, then by \autoref{prop:discrepancy_AS_qt} we have
        $$a(D; X,\Delta+X_{\mathbf{0}})\geq 
        a(E;X',\Delta'+X'_{\mathbf{0}'})\geq -1.$$
So we may assume that the $G$-action is ramified at $E$: this implies that $E$ maps to $\mathbf{0}'$. By \autoref{cor:AS_ext_fixed_locus} we have $[X'^G]=(\mathfrak{r}+1)X'_{\mathbf{0}'}$.
Writing $n=\coeff_E Y'_{\mathbf{0}'}$, the formula of \autoref{prop:discrepancy_AS_qt} reads:
    \begin{eqnarray*}
        e(E)\cdot a(D;X,\Delta+X_\mathbf{0})&=&
            a(E;X',\Delta'+X'_{\mathbf{0}'})
            -(p-1)\coeff_E\varphi'^*X'_{\mathbf{0}'} \\
            && \quad +(p-1)\coeff_E \varphi'^* (\mathfrak{r}+1)X'_{\mathbf{0'}}
            -(p-1)\coeff_E [Y'^G] \\
            &=& a(E;X',\Delta'+X'_{\mathbf{0}'}) 
            +(p-1)\cdot \left[
            \mathfrak{r}n-i(E)
            \right].
    \end{eqnarray*} 
We need to estimate $i(E)$ to conclude. In the DVR $\sO_{Y',E}$ we can write $t=us^n$ where $s$ is a uniformizer and $u\in \sO_{Y',E}^\times$. By \autoref{claim:ram_theory} we have
        $$\ord_EI(t)=(\mathfrak{r}+1)n.$$
We distinguish now several cases:
    \begin{itemize}
        \item \emph{$(n,p)=1$.} Passing to the henselization or the completion of $\sO_{Y',E}$, we may assume that $u=1$. Then by \cite[Lemma 2.3]{Totaro_Terminal_non_CM_3folds}:
                $$I(t)=I(s^n)=I(s)\sum_{i=1}^n \sigma(s)^{i-1}s^{n-i}.$$ 
        As $\ord_E \sum_{i=1}^n \sigma(s)^{i-1}s^{n-i}\geq n-1$ we deduce that $\ord_E I(s)\leq \mathfrak{r}n+1$.
        \begin{itemize}
            \item If the $G$-action is wildly ramified at $E$, then $e(E)=p$ and $i(E)=\ord_E I(s)\leq \mathfrak{r}n+1$. Hence
                    \begin{eqnarray*} 
                    a(D;X,\Delta+X_\mathbf{0}) &=&
                    \frac{1}{p}\big[
                    a(E;X',\Delta'+X'_{\mathbf{0}'}) 
                    +(p-1)(\mathfrak{r}n-i(E))
                    \big] \\
                    &\geq & \frac{1}{p}\big[
                    -1 
                    +(p-1)(-1) \big] \\
                    &= & -1.
                    \end{eqnarray*}

            \item If the $G$-action is fiercely ramified at $E$, then $e(E)=1$ and $i(E)<\ord_E I(s)\leq \mathfrak{r}n+1$, whence $i(E) \leq \mathfrak{r}n$. Hence
                    \begin{eqnarray*} 
                    a(D;X,\Delta+X_\mathbf{0}) &=&
                    a(E;X',\Delta'+X'_{\mathbf{0}'}) 
                    +(p-1)(\mathfrak{r}n-i(E)) \\
                    & \geq  & a(E;X',\Delta'+X'_{\mathbf{0}'}) \\
                    &\geq & -1.
                    \end{eqnarray*}
        \end{itemize}

        \item \emph{$(n,p)=p$.} We have \cite[Lemma 2.3]{Totaro_Terminal_non_CM_3folds}:
            $$I(t)=\sigma(u)I(s)\sum_{i=1}^n \sigma(s)^{i-1}s^{n-i}
            +I(u)s^n.$$
        We have $\ord_E\sum_{i=1}^n \sigma(s)^{i-1}s^{n-i}\geq n-1$. As $I(t)$ has order $(\mathfrak{r}+1)n$ and $s^n$ has order $n$, we see that at least one of $\ord_E I(s)\leq\mathfrak{r} n+1$ or $\ord_E I(u)\leq \mathfrak{r}n$ must happen. Now:
            \begin{itemize}
                \item If the ramification is fierce then $i(E)<\ord_E I(s)$ and $i(E)\leq \ord_E I(u)$. So in any case $i(E)\leq \mathfrak{r}n$.
                \item If the ramification is wild, then $i(E)=\ord_EI(s)$. Moreover by \autoref{lemma:wild_ramification_triv_resid_action} we have:
                    $$\sum_{i=1}^n \sigma(s)^{i-1}s^{n-i} \equiv 
                    \sum_{i=1}^n s^{n-1}\equiv ns^{n-1}\equiv 0 \quad 
                    \mod (s^n).$$
                Hence the order of $\sum_{i=1}^n \sigma(s)^{i-1}s^{n-i}$ is $\geq n$, so in fact $i(E)\leq \mathfrak{r}n$.
            \end{itemize}
         We have found that $i(E)\leq \mathfrak{r}n$ always holds, and so
                \begin{eqnarray*} 
                    a(D;X,\Delta+X_\mathbf{0}) &=&
                    \frac{1}{e(E)}\big[
                    a(E;X',\Delta'+X'_{\mathbf{0}'}) 
                    +(p-1)(\mathfrak{r}n-i(E))
                    \big] \\
                    &\geq & \frac{1}{e(E)}(-1 ) \\
                    &\geq & -1.
                    \end{eqnarray*}
    \end{itemize}
This completes the proof of the proposition.
\end{proof}

\section{Permanence under base-change}\label{section:permanence_under_base_change}

\subsection{When inversion of adjunction applies}\label{section:inv_adjunction}

\begin{terminology}
Let $(X,S+\Delta)$ be a pair, where $S$ is a reduced divisor. 
Let $\lbrack S\rbrack$ be the set of prime divisors corresponding to the irreducible components of $S$.
We say that \emph{inversion of adjunction} holds for $(X,S+\Delta)$ along $S$ if: 
    $$\totdis(S^\nu,\Diff_{S^\nu}(\Delta))=
    \totdis(E\notin \lbrack S\rbrack, \text{center} \cap S\neq \emptyset,
    X,S+\Delta),$$
see \cite[Theorem 4.9]{Kollar_Singularities_of_the_minimal_model_program}.
We say that \emph{inversion of adjunction holds (along Cartier boundaries) in dimension $n$} if inversion of adjunction holds for all pairs $(X, \Delta+S)$ as above (with $S$ Cartier) and $\dim(X)=n$.
\end{terminology}

In equi-characteristic as well as in mixed characteristic, inversion of adjunction implies permanence of strong local stability under finite base-change.

\begin{proposition}\label{lemma:inversion}
Let $R$ be a DVR with fraction field $K$ and algebraically closed residue field $k$. 
Let $(X,\Delta)\to\Spec(R)$ be a strongly locally stable family of relative dimension $n$.
Let $R'$ be a finite normal DVR extension of $R$. If inversion of adjunction along Cartier boundaries holds for pairs of relative dimension $n$ over $R'$, then the base-change 
$(X,\Delta)\times_{\Spec(R)} \Spec(R')\to\Spec (R')$ is strongly locally stable.
\end{proposition}

\begin{proof}
Let $\mathbf{0}$ be the closed point of $\Spec(R)$ and $\mathbf{0}'$ be the closed point of $\Spec(R')$, and write $(X,\Delta)\times_{\Spec(R)} \Spec(R')=(X',\Delta')$.
It suffices to show that $(X',\Delta')\to \Spec(R')$ is locally stable.
By \autoref{prop:demi_normality_loc_stable_family} and \autoref{lemma:Q_cartier_cond_base_changes}, the base-change $X'$ is normal and $K_{X'}+\Delta'$ is $\bQ$-Cartier. Moreover $X'_{\mathbf{0}'}\cong X_\mathbf{0}$ is demi-normal. Applying \cite[Proposition 4.5.(4)+(6)]{Kollar_Singularities_of_the_minimal_model_program} we see that the divisor 
    $$(K_{X'}+\Delta'+X'_{\mathbf{0}'})|_{(X'_{\mathbf{0}'})^\nu}
    =K_{(X'_{\mathbf{0}'})^\nu}+\Diff_{(X'_{\mathbf{0}'})^\nu}(\Delta')$$ 
is equal, under the isomorphism $X'_{\mathbf{0}'}\cong X_\mathbf{0}$, to $(K_X+\Delta+X_\mathbf{0})|_{X_\mathbf{0}^\nu}$ (\footnote{
    Notice that the statement of \cite[Proposition 4.5.(6)]{Kollar_Singularities_of_the_minimal_model_program} requires the pair to be lc: but this is only used to provide a local description of the prime divisor $S$ (in the notations of \emph{op.cit.}), which in our case is the central fiber, and so the argument applies verbatim. 
}). Combining this fact with \autoref{rmk:adjunction}, we obtain that $((X'_{\mathbf{0}'})^\nu,\Diff_{(X'_{\mathbf{0}'})^\nu}(\Delta'))$ is lc.

By inversion of adjunction, we deduce that $a(E;X',\Delta'+X_{\mathbf{0}'}')\geq -1$ for every divisor $E$ whose center on $X'$ is contained in $X'_{\mathbf{0}'}$. If $E$ is a divisor over $X'$ whose center dominates $\Spec(R')$, then
        $$a(E;X',\Delta'+X_{\mathbf{0}'}')=a(E_{\Frac(R')};X'_{\Frac(R')},\Delta'_{\Frac(R')})\geq -1$$
since $(X_{\Frac(R)},\Delta_{\Frac(R)})$ is geometrically lc. This shows that $(X',\Delta')\to \Spec(R')$ is locally stable. 
\end{proof}

At the moment, inversion of adjunction is known in the following cases:
    \begin{itemize}
        \item In equicharacteristic $0$ \cite[Theorem 4.9]{Kollar_Singularities_of_the_minimal_model_program};
     
        \item In dimension $3$ and residue characteristics $p>5$ \cite[Corollary 10.1]{7authors_globally+regular_MMP};
        
        \item In dimension $3$ for plt pairs and equicharacteristic $p>0$ \cite[Corollary 1.5]{Hacon_Witaszek_On_the_relative_MMP_for_threefolds_in_low_char};
        
        \item In dimension $4$ for plt pairs in equicharacteristic $p>5$, assuming the existence of log resolutions \cite[Corollary 4.9]{Hacon_Witaszek_MMP_for_4folds_pos_and_mixed_char} (\footnote{The assumption of standard coefficients in \emph{op.cit.} is not necessary if we assume existence of all log resolutions: see the proof of Corollary 4.8 and Theorem 4.1 in \emph{op.cit.}}).
    \end{itemize}

\begin{remark}
There is also an inversion of adjunction statement for \emph{MJ-log canonical singularities} which can be applied in our context, see \cite{Ishii_Reguera_Singularities_in_arb_char_via_jets}. MJ-log canonical singularities are log canonical: see \cite[Proposition 2.21]{Ein_Ishii_Singularities_wrt_MJ_discrepancies}, which also works in positive characteristic. The converse is true for local complete intersection singularities, but fails dramatically in general: there are examples of three-dimensional terminal quotient singularities (in any characteristic) which are not MJ-lc \cite[Example 3.14]{Ein_Ishii_Singularities_wrt_MJ_discrepancies}.
\end{remark}


There is another important case that is known in every dimension: adjunction for $F$-purity. We discuss it next.

\subsubsection{$F$-locally stable families}

\begin{definition}\label{d-F-locallystable}
Let $C$ be a (germ of) smooth curve over $k$. We say that a family of pairs $f\colon (X,\Delta)\to C$ is \emph{$F$-locally stable} if:
        \begin{enumerate}
            \item $X$ is normal, $K_X+ \Delta$ is $\bQ$-Cartier and $\Delta$ is an effective $\bZ_{(p)}$-divisor, and
            \item for each closed point $c\in C(k)$, the pair $(X,\Delta+X_c)$ is $F$-pure.
        \end{enumerate}
\end{definition}

\begin{proposition}\label{prop:permanence_F_pure_case}
Let $f\colon (X,\Delta)\to C$ be an $F$-locally stable family. Then:
    \begin{enumerate}
        \item $(X,\Delta)\to C$ is strongly locally stable with $F$-pure fibers (in particular with only separable nodes);
        \item for every finite flat morphism $C'\to C$, where $C'$ is a regular curve, the base-change $(X,\Delta)\times_C C'\to C'$ is $F$-locally stable.
    \end{enumerate}
\end{proposition}
\begin{proof}
Let us first prove that $(X,\Delta)\to C$ is locally stable with $F$-pure fibers. For any $c\in C$ closed, the pair $(X,\Delta+X_c)$ is lc by \cite{Hara_Watanabe_F_sing_vs_lt_and_lc}. So $(X,\Delta)\to C$ is locally stable. Since $\Delta+X_c$ is $\bQ$-Cartier by assumption, by \cite[Lemma 3.4]{Polstra_Simpson_Tucker_F-pure_inversion_of_adjunction} we obtain that $\sO_X$ and $\sO_X(-X_c)$ are $S_3$. So $X_c$ is $S_2$ by \cite[Lemma 2.60]{Kollar_Singularities_of_the_minimal_model_program}. Since $X_c$ has at worst nodes in codimension one (by the classification of codimension two lc singularities \cite[Corollary 2.32]{Kollar_Singularities_of_the_minimal_model_program}), we get that the closed fibers $X_c$ are demi-normal. By \cite[Theorem A]{Polstra_Simpson_Tucker_F-pure_inversion_of_adjunction} the pairs $(X_c,\Delta_c)$ are $F$-pure, which implies in particular that every node is separable.

Now let $C'\to C$ be as in the second assertion. By \autoref{lemma:localiz_on_the_base} we may assume that $C$ and $C'$ are spectra of DVRs, with respective closed point $\mathbf{0}$ and $\mathbf{0}'$. Write $(X',\Delta')=(X,\Delta)\times_C C'$. By assumption $\Delta'$ is a $\bZ_{(p)}$-divisor, the fibers $X_\mathbf{0}\cong X'_{\mathbf{0}'}$ are demi-normal as observed above, and by \autoref{prop:demi_normality_loc_stable_family} the fiber product $X'$ is normal. Then, applying \cite[Theorem 2.6]{Polstra_Simpson_Tucker_F-pure_inversion_of_adjunction} first for $F$-pure adjunction on $X_\mathbf{0}\cong X'_{\mathbf{0}'}$, then for $F$-pure inversion of adjunction on $(X',\Delta'+X'_{\mathbf{0}'})$, we see that the germ $(x'\in X',\Delta'+X'_{\mathbf{0}'})$ is $F$-pure for every $x'$ mapping to $\mathbf{0}'$. Since the locus $\{x'\in X'\mid (x'\in X',\Delta'+X'_{\mathbf{0}'}) \text{ is }F\text{-pure}\}$ is open, and $X'\to C$ is closed, we deduce that $(X',\Delta'+X'_{\mathbf{0}'})$ is $F$-pure. This proves the second assertion.

To finish the proof of the first assertion, it remains to show that the generic fiber of $(X,\Delta)\to C$ is geometrically lc. This is true, as we have seen that the generic fiber of $(X',\Delta')\to C'$ is $F$-pure, in particular log canonical, for any $C'\to C$.
\end{proof}

\subsection{Reduction to purely inseparable base-changes 
after \cite{Hu_Zong_Base_change_local-stability_positive_char} and 
\cite{Posva_Local_stability_comparison}
}\label{section:reduction_to_pur_insep}
It is shown in \cite{Hu_Zong_Base_change_local-stability_positive_char} and
\cite{Posva_Local_stability_comparison} that if $(X,\Delta)\to C$ is locally stable and $C_\text{AS}\to C$ is an Artin--Schreier cover, then local stability of $(X,\Delta)\times_CC_\text{AS}\to C_\text{AS}$ holds if all Frobenius base-changes $(X^{(e)},\Delta^{(e)})\to C$ are locally stable (\footnote{
    We remark that our definition of ``local stability" is different from the one used in \cite{Hu_Zong_Base_change_local-stability_positive_char}. We stick to our definitions and reformulate the results of \emph{op.cit.}\ in our terminology. 
}).

In \cite[Theorem 1.6]{Hu_Zong_Base_change_local-stability_positive_char} the authors claim that local stability of arbitrary finite base-changes $X\times_C C'\to C'$ therefore follows from the local stability of base-changes $X\times_C \mathfrak{C}\to \mathfrak{C}$ where $\mathfrak{C}$ is any finite flat cover of the perfection $C_\text{perf}$ of $C$; however no argument is provided, and the subtleties inherent to handling perfect schemes are not dealt with. In this subsection, we clarify their statement. To simplify the formulation, we say that:

\begin{terminology}
\emph{(Strong) local stability in relative dimension $n$ is preserved by base-changes with property $\mathbf{P}$} if, for any (strongly) locally stable family $(X,\Delta)\to C$ of relative dimension $n$, where $C$ is a smooth $k$-curve, and any morphism $C'\to C$ with property $\mathbf{P}$, where $C'$ is also a smooth $k$-curve, it holds that $(X,\Delta)\times_C C'\to C'$ is (strongly) locally stable.  
\end{terminology}

\begin{proposition} \label{p-purelyinsep_vs_wild}
Assume that (strong) local stability in relative dimension $n$ is preserved by base-changes by finite flat purely inseparable morphisms. Then (strong) local stability in relative dimension $n$ is preserved by arbitrary finite flat base-changes.
\end{proposition}
\begin{proof}
Fix a (strongly) locally stable family $(X,\Delta)\to C$ of relative dimension $n$, and a finite flat $C'\to C$ where $C'$ is a smooth curve. The conditions on the geometric fibers are preserved by base-change, so it suffices to prove the locally stable instance of the statement. By \autoref{lemma:localiz_on_the_base} and \cite[04GH]{Stacks_Project} we may assume that $C=\Spec(\sO)$ and $C'=\Spec(\sO')$ where $\sO$ and $\sO'$ are henselian DVRs. 

We adopt the following convention: whenever $\mathscr{F}$ is a finite field extension of $\Frac(\sO)$, we let $C_\mathscr{F}$ be the normalization of $C$ in $\mathscr{F}$: by \cite[04GH]{Stacks_Project}, $C_\mathscr{F}$ is the spectrum of an henselian DVR with residue field $k$. We denote by $(X_\mathscr{F},\Delta_\mathscr{F})=(X,\Delta)\times_CC_\mathscr{F}\to C_\mathscr{F}$ the base-change of $(X,\Delta)\to C$.
We will also use without mention the following fact: if $\mathbf{G}$ is a finite group of $\Frac(\sO)$-automorphisms of $\mathscr{F}$ with fixed field $\mathscr{H}$, then $\mathbf{G}$ acts on on $C_\mathscr{F}$ and $X_\mathscr{F}$, making $X_\mathscr{F}\to C_\mathscr{F}$ equivariant, and the quotient of the family is $(X_\mathscr{H},\Delta_\mathscr{H})\to C_\mathscr{H}$.

That being said, by \cite[V, Proposition 6.6]{Lang_Algebra} the induced extension of fraction fields factors as
        $$\Frac(\sO)\hookrightarrow M\hookrightarrow \Frac(\sO')$$
where $\Frac(\sO) \subseteq M$ is separable and $M \subseteq \Frac(\sO')$ is purely inseparable. If $(X_M,\Delta_M)\to C_M$ is locally stable, then by assumption $(X_{\Frac(\sO')},\Delta_{\Frac(\sO')})\to C'$ is locally stable. So we may assume that $\Frac(\sO')$ is a separable extension of $\Frac(\sO)$.

Fix a normal closure $L$ of $\Frac(\sO)\hookrightarrow \Frac(\sO')$. Then $L$ is a Galois extension of $\Frac(\sO)$ \cite[VI, Corollary 1.6]{Lang_Algebra} and of $\Frac(\sO')$ \cite[VI, Theorems 1.1 and 1.8]{Lang_Algebra}. We will use the main result of \cite{Hu_Zong_Base_change_local-stability_positive_char} to show that $X_L\to C_L$ is locally stable, and then use our descent statements to deduce that $X_{\Frac(\sO')}\to C'$ is locally stable.

\medskip
\textsc{Step 1: $(X_L,\Delta_L)\to C_L$ is locally stable.} Let $G$ be the Galois group of $L/\Frac(\sO)$.
Choose a maximal $p$-Sylow subgroup $G_p<G$, and let $L_0=L^{G_p}$. Then $[L_0:\Frac(\sO)]$ is prime to $p$. As $C_{L_0}$ and $C$ are local rings, the only ramification point of $C_{L_0}\to C$ is the closed point. There the ramification index must divide $[L_0:\Frac(\sO)]$ \cite[09E8]{Stacks_Project}, and so we see that $C_{L_0}\to C$ has only tame ramifications. By \autoref{prop:tame_base_change} it follows that $(X_{L_0},\Delta_{L_0})\to C_{L_0}$ is locally stable.

By \cite[VI, Theorem 1.8]{Lang_Algebra} the extension $L/L_0$ is Galois with group $G_p$. Since $G_p$ is a $p$-group, there exists a series
        $$\{1\}=G^n\lhd G^{n-1} \lhd G^{n-2} \lhd \dots \lhd G^0=G_p$$
such that each $G^{i}/G^{i+1}\cong \bZ/p$ \cite[I, Corollary 6.6]{Lang_Algebra}. If $L_i=L^{G^i}$, we obtain a sequence of finite extensions of henselian DVRs
        $$\sO_{C_{L_0}}\hookrightarrow \sO_{C_{L_1}}\hookrightarrow \dots 
        \hookrightarrow \sO_{C_{L_{n-1}}}\hookrightarrow \sO_{C_{L_{n}}}=\sO_{C_L}$$
such that each extension $\Frac(\sO_{C_{L_i}})\hookrightarrow \Frac(\sO_{C_{L_{i+1}}})$ is of degree $p$, hence an Artin--Schreier extension. Combining the assumption and \cite[Theorem 1.2]{Posva_Local_stability_comparison}, we see inductively that each $(X_{L_i},\Delta_{L_i})\to C_{L_i}$ is locally stable. So $(X_L,\Delta_L)\to C_L$ is locally stable.

\medskip
\textsc{Step 2: $(X_{\Frac(\sO')},\Delta_{\Frac(\sO')})\to C'$ is locally stable.} 
Let $H$ be the Galois group of $L/\Frac(\sO')$. Choose a maximal $p$-Sylow subgroup $H_p<H$ and let $K_0=L^{H_p}$. As in Step 1, we see that we have a commutative diagram
        $$\begin{tikzcd}
        (X_{L},\Delta_L)\arrow[r]\arrow[d] & (X_{K_{m-1}},\Delta_{K_{m-1}})\arrow[r]\arrow[d] & \dots \arrow[r] & (X_{K_1},\Delta_{K_1})\arrow[r]\arrow[d] & (X_{K_0},\Delta_{K_0})\arrow[d] \\
        C_L\arrow[r] & C_{K_{m-1}}\arrow[r] & \dots\arrow[r] & C_{K_1}\arrow[r] & C_{K_0}
        \end{tikzcd}$$
where each small square is the base-change by an Artin--Schreier extension. By \autoref{prop:descent_loc_stab_along_AS} we obtain that $(X_{K_0},\Delta_{L_0})\to C_{K_0}$ is locally stable. Finally, $C_{L_0}\to C'$ is ramified only at the closed point with ramification index dividing $[\Frac(K_0):\Frac(\sO)]=|H|/|H_p|$, which is prime to $p$. So local stability descends to $(X_{\Frac(\sO')},\Delta_{\Frac(\sO')})\to C'$ by \autoref{prop:tame_base_change}, which completes the proof.
\end{proof} 

\begin{remarks}
\begin{enumerate}
    \item As the proof shows, given a (strong) locally stable family $(X,\Delta)\to C$, to show (strong) local stability for arbitrary finite flat base-changes $(X,\Delta)\times_CC'\to C'$, it suffices to assume permanence under Frobenius base-changes of local stability for one-parameter families with geometric fibers amongst $\{(X_{\bar{c}},\Delta_{\bar{c}})\mid c\in C\}$.
    
    \item Suppose that $f\colon (X,\Delta)\to C$ is locally stable, and let $F_{C/k}^e$ be the $e$-th $k$-linear Frobenius of $C$. If $(X,\Delta)\times_{(f,F_{C/k}^n)}C\to C$ is locally stable, then so is $(X,\Delta)\times_{(f,F_{C/k}^m)}C\to C$ for any $m\leq n$ by \autoref{prop:loc_stability_lost_forever}. A potential definition of local stability of $(X,\Delta)\times_C C_\text{perf}\to C_\text{perf}$ could be: there exists a sequence of positive integers $(n_i)$, with $\lim_{i\to \infty}n_i=\infty$, such that each $(X,\Delta)\times_{(f,F_{C/k}^{n_i})}C\to C$ is locally stable. Note that this implies that $(X,\Delta)\to C$ is strongly locally stable.
\end{enumerate}
\end{remarks}

\subsection{Base-change by Frobenius}\label{section:Frob_base_change}
Let $f\colon (X,\Delta)\to C$ be a family of pairs over a smooth (germ of) $k$-curve. We use the notation introduced in \autoref{section:div_and_Frob_base_change}. We make the following assumption for the whole subsection:

\begin{assumption}\label{assumption:nice_family_to_base_change}
The variety $X$ is normal, the closed fibers of $X\to C$ are demi-normal with only separable nodes, the generic fiber is geometrically normal, and $K_X+\Delta$ is $\bQ$-Cartier.
\end{assumption}

This implies that $X^{(1)}$ is normal and that $K_{X^{(1)}}+\Delta^{(1)}$ is $\bQ$-Cartier (\autoref{prop:demi_normality_loc_stable_family} and \autoref{lemma:Q_cartier_cond_base_changes}). We wish to compute the discrepancies of $(X^{(1)},\Delta^{(1)}+X_\mathbf{0}^{(1)})$. Of course, we are interested in the case where $(X,\Delta)\to C$ is strongly locally stable, and for this reason we focus on exceptional divisors that are vertical over $C$. Nonetheless, the formulas that we obtain below are valid more generally under \autoref{assumption:nice_family_to_base_change}.

\begin{remark}
Geometric log canonicity of $(X_{K(C)},\Delta_{K(C)})$ implies that $\Delta^{(e)}_{K(C)}$ has coefficients $\leq 1$ for every $e$. This puts some restrictions on the invariance of $\Supp(\Delta_{K(C)})$ under higher order derivations, but we will not elaborate on this point.
\end{remark}

\subsubsection{Discrepancies of vertical divisors}
Let $E$ be an exceptional divisor over $X$, and $E^{(1)}$ be the corresponding exceptional divisor over $X^{(1)}$. Recall the functions $\delta_{T_{X/C}}(-)$ and $\delta_\mathfrak{F}(-)$ introduced in \autoref{def:deltas}.

\begin{proposition}\label{prop:formula_discrep_qt}
Let $E$ be a divisor over $X$, vertical over $C$. 
Then:
        $$a(E^{(1)};X^{(1)},\Delta^{(1)}+X_\mathbf{0}^{(1)})
        =
        \frac{1}{\delta_{T_{X/C}}(E)}\cdot \left[
        p\cdot a(E;X,\Delta+X_\mathbf{0})-(p-1)\cdot \gamma_w(E)
        \right],$$
where $\gamma_w(E)$ is defined as in \autoref{def:wild_and_tame_divisors}.
\end{proposition} 
\begin{proof}
Let $\phi\colon Y\to X$ be a normal birational model on which $E$ appears. Then we have a commutative diagram
        $$\begin{tikzcd}
            Y\arrow[r, "q'"] \arrow[d, "\phi"] & Y^{(1),\nu}
            \arrow[d, "\phi'"] \\
            X \arrow[r, "q"] & X^{(1)}
        \end{tikzcd}$$
where $q'$ is the quotient by $\phi^*T_{X/C}$ and $\phi'$ is birational. Note that $q^*\Delta^{(1)}=p\Delta$, since by definition $\Delta^{(1)}$ is the pullabck of $\Delta$ along $X^{(1)}\to X$. By throwing away a closed subset of $Y$ that does not contain $E$, we may assume that the generic fiber of $Y\to C$ is geometrically normal.
Then, using \autoref{eqn:can_divisor_rel_log_foliation} and \cite[Proposition 4.2.3]{Posva_Singularities_quotients_by_1_foliations} we get
        \begin{eqnarray*}
            \phi^*q^*(K_{X^{(1)}}+\Delta^{(1)})&=&
        \phi^*(K_X+(p-1)K_{T_{X/C}}+p\Delta) \\
        &=&
        \phi^*(p(K_X+\Delta)-(p-1)f^*K_C) \\
        &=&
        p(K_Y+\phi^{-1}_*\Delta)-(p-1)(f\circ\phi)^*K_C
        -p\cdot a(E;X,\Delta)\cdot E
        \end{eqnarray*}
(where the fact that $K_{T_{X/C}}$ might not be $\bQ$-Cartier is not an issue), and 
        \begin{eqnarray*}
        q'^*\phi'^*(K_{X^{(1)}}+\Delta^{(1)}) &=&
        p(K_Y+\phi^{-1}_*\Delta)-(p-1)(f\circ\phi)^*K_C \\
        && \quad -\big[
        (p-1)(\ord_EY_\mathbf{0}+\gamma_w(E))+
        \delta_{T_{X/C}}(E)\cdot a(E^{(1)};X^{(1)},\Delta^{(1)})
        \big]\cdot E.
        \end{eqnarray*}
Equating the two expressions yields
        $$p\cdot a(E;X,\Delta)-(p-1)(\ord_EY_\mathbf{0}+\gamma_w(E))
        =
        \delta_{T_{X/C}}(E)\cdot a(E^{(1)};X^{(1)},\Delta^{(1)}).$$
Combining this equality with \autoref{lemma:orders_central_fibers}, we can write
        \begin{eqnarray*}
        \delta_{T_{X/C}}(E)\cdot a(E^{(1)};X^{(1)},\Delta^{(1)}+X_\mathbf{0}^{(1)}) &=&
        \delta_{T_{X/C}}(E)\cdot a(E^{(1)};X^{(1)},\Delta^{(1)})-\delta_{T_{X/C}}(E)\cdot \ord_{E^{(1)}}Y_\mathbf{0}^{(1),\nu} \\
        &=& \delta_{T_{X/C}}(E)\cdot a(E^{(1)};X^{(1)},\Delta^{(1)})-\ord_EY_\mathbf{0} \\
        &=& p\cdot \big[a(E;X,\Delta)-  \ord_EY_\mathbf{0}\big]-
        (p-1)\gamma_w(E) \\
        &=& 
        p\cdot a(E;X,\Delta+X_\mathbf{0})-(p-1)\gamma_w(E)
        \end{eqnarray*}
which yields the statement.
\end{proof} 

\begin{remark}
We can also find the formula of \autoref{prop:formula_discrep_qt} starting from \cite[Theorem 3.25]{Benozzo_CBF_in_pos_char} in the case $e=1$.
\end{remark}

Let us also compute the discrepancies of $\mathfrak{F}$:

\begin{proposition}\label{prop:discrepancy_base_change_foliation}
If $E$ is vertical over $C$, then
        $$a(E^{(1)};\mathfrak{F})=
        \frac{1}{\delta_{T_{X/C}}(E)}\big(\gamma_w(E)+\ord_EY_\mathbf{0}\big).$$
\end{proposition}
\begin{proof}
We have
        \begin{eqnarray*}
        a(E;X,\Delta+X_\mathbf{0})&=&
        \delta_\mathfrak{F}(E^{(1)})^{-1}
        \cdot \left[
        a(E^{(1)};X^{(1)},\Delta^{(1)}+X^{(1)}_\mathbf{0})
        \right. \\
        && \quad \quad \quad \quad \quad\quad \left. +(p-1)\cdot \left(a(E^{(1)};\mathfrak{F})
        -\ord_{E^{(1)}}Y^{(1),\nu}_\mathbf{0}\right)
        \right] \\
        &=&
        \delta_\mathfrak{F}(E^{(1)})^{-1}
        \cdot \left[
        \frac{p}{\delta_{T_{X/C}}(E)} \cdot 
        a(E;X,\Delta+X_\mathbf{0}) - \frac{p-1}{\delta_{T_{X/C}}(E)}\cdot \gamma_w(E)
         \right. \\
        && \quad  \quad \quad \quad \quad \quad 
        \left. +(p-1)\cdot a(E^{(1)};\mathfrak{F})
        - \frac{p-1}{\delta_{T_{X/C}}(E)}\cdot \ord_EY_\mathbf{0}
        \right] \\
        &=& 
        a(E;X,\Delta+X_\mathbf{0})-\frac{p-1}{p}\gamma_w(E)
        \\ && \quad \quad 
        +\frac{p-1}{\delta_{T_{X/C}}(E)}\cdot a(E^{(1)};\mathfrak{F}) 
        -\frac{p-1}{p}\ord_EY_\mathbf{0},
        \end{eqnarray*}
where the first equality is \autoref{prop:discrep_qt_by_F}, the second follows from \autoref{prop:formula_discrep_qt} and \autoref{lemma:orders_central_fibers}, and the third one holds thanks to \autoref{lemma:inv_dichotomy}. Using \autoref{lemma:inv_dichotomy} again, we obtain that 
        $$a(E^{(1)};\mathfrak{F})=\frac{\delta_\mathfrak{F}(E^{(1)})}{p}\big(\gamma_w(E)+\ord_EY_\mathbf{0}\big)
        =\frac{1}{\delta_{T_{X/C}}(E)}\big(\gamma_w(E)+\ord_EY_\mathbf{0}\big)$$
as claimed.
\end{proof}

\begin{corollary}\label{cor:wild_coeff_p_div_in_non_inv_case}
If $E$ is vertical, wild and non-$T_{X/C}$-invariant then $\gamma_w(E)\in p\mathbb{Z}_{\geq 0}$.
\end{corollary}
\begin{proof}
In that case we have $\gamma_w(E)=p\cdot a(E^{(1)};\mathfrak{F})-\ord_EY_\mathbf{0}$
and the right-hand side is an integer divisible by $p$, since $K_\mathfrak{F}$ is Cartier by \autoref{lemma:can_div_Frob_is_Cartier} and $\ord_E Y_\mathbf{0}$ is divisible by $p$.
\end{proof}

\begin{example}[Non-zero wild coefficients] \label{e-nonzero_wildcoeffs}
We indicate a way to produce non-$T_{X/C}$-invariant wild divisors with non-zero wild coefficients. 
Start with a non-$T_{X/C}$-invariant divisor $E$ appearing on some model $Y$ (there are many, see \autoref{example:non_inv_wild_div}). By \autoref{prop:discrepancy_base_change_foliation} and \autoref{lemma:orders_central_fibers}, we have
        $$a(E^{(1)};\mathfrak{F})=\frac{\gamma_w(E)}{p}+
        \ord_{E^{(1)}}Y_\mathbf{0}^{(1),\nu}.$$
Take a general regular closed point $\mathbf{z}$ of $E^{(1)}$ that is not contained in any other irreducible component of $Y_\mathbf{0}^{(1),\nu}$, and where the birational pullback of $\mathfrak{F}$ is regular. Then, we can find a local coordinate $y$ such that the birational pullback of $\mathfrak{F}$ is generated by $\partial_y$. Let $x$ be a local equation of $E^{(1)}$ at $\mathbf{z}$: if the images of $x$ and $y$ in the cotangent space 
$\mathfrak{m}_\mathbf{z}/\mathfrak{m}^2_\mathbf{z}$ are linearly equivalent, then $\partial_y(x)\notin \mathfrak{m}_\mathbf{z}$, a contradiction since $E^{(1)}$ must be $\mathfrak{F}$-invariant. So $x,y$ are part of a system of local parameters.

Blow-up the codimension two locus $\mathfrak{c}$ defined by the prime ideal $(x,y)$, and let $D^{(1)}$ be the exceptional divisor. Notice that the image of $\mathfrak{c}$ in $Y$ is contained in $E$ but not in any other irreducible component of $Y_\mathbf{0}$: thus the induced divisor $D$ over $X$ is vertical and wild.
A simple computation in local coordinates shows that $D^{(1)}$ is $\mathfrak{F}$-invariant (so $D$ is non-$T_{X/C}$-invariant by \autoref{lemma:inv_dichotomy}), and the discrepancy along $D^{(1)}$ of the $1$-foliation generated by $\partial_y$ is equal to $1$. Hence
        \begin{eqnarray*}
            a(D^{(1)};\mathfrak{F}) &=&
            a\left(D^{(1)}; \sO_{Y^{(1),\nu},\mathbf{z}}\cdot\partial_y,
            -a(E^{(1)};\mathfrak{F})\cdot E^{(1)}
            \right)\\
            &=&
            1+a(E^{(1)};\mathfrak{F}).
        \end{eqnarray*}
Write $Z^{(1)}=\Bl_\mathfrak{c}Y^{(1),\nu}$. The choice of the center $\mathfrak{c}$ ensures that $\ord_{D^{(1)}}Z^{(1)}_\mathbf{0}=\ord_{E^{(1)}}Y^{(1),\nu}_\mathbf{0}$. Since $D$ is non-$T_{X/C}$-invariant, as above we have
        $$a(D^{(1)};\mathfrak{F})=\frac{\gamma_w(D)}{p}+
        \ord_{E^{(1)}}Z_\mathbf{0}^{(1)}.$$
Comparing the formulas yields $\gamma_w(D)=\gamma_w(E)+p>0$.
\end{example}

We consider now whether we can bound the discrepancy 
$a(E^{(1)};X^{(1)},\Delta^{(1)}+X_\mathbf{0}^{(1)})$ from below.

\begin{proposition}\label{prop:discrepancy_tame_divisor}
Assume that $E$ is vertical and \emph{tame} over $C$. Then $E$ is $T_{X/C}$-invariant and: 
$a(E^{(1)};X^{(1)},\Delta^{(1)}+X^{(1)}_\mathbf{0})+1= 
p\cdot \big[a(E;X,\Delta+X_\mathbf{0})+1
\big]$.
\end{proposition}
\begin{proof}
Let $t$ be a uniformizer of $\sO_{C,\mathbf{0}}$. Let $\phi\colon Y\to X$ be a model of $X$ where $E$ appears: then in the DVR $\sO=\sO_{Y,E}$ we can write $t=us^n$ where $s$ is a uniformizer of $\sO$, $u\in \sO^\times$ and $n$ is prime to $p$. Passing to an \'{e}tale neighbourhood of $\sO$ if necessary, we may assume that $u=1$. Let $x_1,\dots,x_n\in \sO$ be such that $s,x_1,\dots,x_n$ is a differential basis of $\sO$ over $k$ (this is possible since $\sO$ is regular, hence smooth over $k$). As $\phi^*T_{X/C}=T_{Y/C}$ at the generic point of $E$, we see that $\phi^*T_{X/C}$ is generated by the derivations $\partial_{x_i}$, $i=1,\dots,n$. We have $\partial_{x_i}(t)=ns^{n-1}\partial_{x_i}(s)=0$ so we see that $E$ is $T_{X/C}$-invariant. Since $\gamma_w(E)=-1$, by \autoref{prop:formula_discrep_qt} we obtain that
        \begin{eqnarray*}
            a(E^{(1)};X^{(1)},\Delta^{(1)}+X^{(1)}_\mathbf{0}) 
            +1
            &=&
            p \cdot a(E;X,\Delta,X_\mathbf{0})+(p-1)+1 \\
            &=&
            p\cdot\big[a(E;X,\Delta+X_\mathbf{0})+1\big]
        \end{eqnarray*}
as claimed.
\end{proof} 

The wild case is more evasive, because $E$ can be $T_{X/C}$-invariant or not and, in view of of \autoref{prop:formula_discrep_qt}, we have to control the wild coefficient $\gamma_w(E)$ in terms of the discrepancies of $(X,\Delta)$ only. 
On the other hand, if we know for some reason that $(X^{(1)},\Delta^{(1)})\to C$ is locally stable, then \autoref{prop:formula_discrep_qt} implies that discrepancies and wild coefficients of wild divisors are not independent---the most striking consequence being for wild log canonical places, see \autoref{prop:discrep_formula_at_lc_places} below. 

\begin{example}[Snc case]
Suppose that $X_\mathbf{0}$ is snc. We can repeatedly blow-up strata contained in the central fiber to create wild divisors, but their wild coefficients will stay small compared to the discrepancies. Indeed, we know by \cite[Lemma 2.57]{Kollar_Families_of_varieties_of_general_type} that $X^{(1)}\to C$ is still locally stable.

For example, suppose that $X=\bA^{n}_{x_\bullet}$ with $f=\colon X\to \bA^1_t$ given by $f^*(t)=x_1\cdots x_e$ for $e\leq n$. Blow-up $Z=V(x_1,\dots,x_d)$ where $d\geq e$, and denote by $E$ the exceptional divisor. Then $a(E;X,X_\mathbf{0})=-1$. The multiplicity of $E$ in the central fiber is $e$. Assume that $p$ divides $e$, so $E$ is a wild divisor. To compute the other invariants, consider the blow-up chart $\bA^n_{y_\bullet}$ is given by
        $$x_1\mapsto y_1, \quad x_i\mapsto y_1y_i \ (i>1).$$
As $t=y_1^e\cdot (y_2\cdots y_e)$ and $y_1,\dots,y_n$ is a differential basis of $\sO_{\Bl_ZX,E}$, we see that $\gamma_w(E)=0$. Notice that $T_{\Bl_ZX/\bA^1_t}$ is the set of those derivations that annihilate $y_2\cdots y_e$: it contains $\partial_{y_1}$, and so $E$ is \emph{not} $T_{X/C}$-invariant. 
\end{example}

\begin{example}[Family of semi-stable curves]
Assume that $f\colon (X,\Delta)\to C$ is strongly locally stable of relative dimension one. Let $x$ be a closed, singular point of $X$ contained in the central fiber. Then $(x\in X,\Delta)$ is a canonical surface singularity and so, by \cite[Theorem 2.29]{Kollar_Singularities_of_the_minimal_model_program}, $\Delta=0$ around $x$. Let $\mu\colon Y\to X$ be the minimal resolution, with exceptional curves $E_i$: we know that $a(E_i;X)=0$. So we have
        $$-1\leq a(E_i;X,X_\mathbf{0})=-\mult_{E_i}(\mu^*X_\mathbf{0})$$
and hence each $E_i$ has multiplicity one in $Y_\mathbf{0}$: in particular none of them are wild divisors. In most cases the fibers of $Y\to C$ have snc support, see \cite[Corollary 3.31]{Kollar_Singularities_of_the_minimal_model_program}, and then the previous example applies for all further blow-ups.
\end{example}

The following result shows that the situation at wild places  that are not $(1-1/p)$-klt is, in some sense, critical:

\begin{proposition}\label{prop:discrep_formula_at_lc_places}
Assume that $E$ is a wild divisor and that  $a(E;X,\Delta+X_\mathbf{0})\leq -1/p$.
Then the discrepancy $a(E^{(1)};X^{(1)},\Delta^{(1)}+X^{(1)}_\mathbf{0})\leq -1$, with equality if and only if: $E$ is \emph{not} $T_{X/C}$-invariant, and the wild coefficient $\gamma_w(E)=0$.
\end{proposition}
\begin{proof}
This is a direct consequence of \autoref{prop:formula_discrep_qt} and $\gamma_w(E)\geq 0$.
\end{proof}

\begin{example}\label{example:non_lc_family_curves_I}
    In characteristic $p=2$, we consider
    $f\colon X=\bA^{n+2}_{xyz_\bullet}\to \bA^1_t=C$ given by
    $t=x(x+y^r)$ for $r\geq 1$. This is an $F$-locally stable family when $r=1$, but when $r>1$ it is neither $F$-locally stable nor locally stable. We blow-up the ideal $(x,y,z_\bullet)$ on $X$ to obtain a chart $\bA3$ with coordinates $\bar{x},\bar{y},\bar{t},\bar{z}_\bullet$ where $x=\bar{x}$, $y=\bar{x}\bar{y}$ and $z_i=\bar{x}\bar{z}_i$ for every $i$. 
    The exceptional divisor is $E=V(\bar{x})$ and $a(E;X,X_\mathbf{0})=n-1$. We have
            $$t=\bar{x}^2(1+\bar{x}^{r-1}\bar{y}^r),$$
    so $E$ is wild and
        $$\gamma_w(E)=
        \begin{cases}
            r-1 & \text{if }r \text{ is odd,}\\
            r-2 & \text{if }r \text{ is even.}
        \end{cases}$$
    The relative tangent sheaf of the blow-up $Y$ over $C$ is generated by $\partial_{\bar{y}}, \partial_{\bar{z}_\bullet}$ if $r$ is even, and by $\partial_{\bar{x}}, \partial_{\bar{z}_\bullet}$ if $r$ is odd. So $E$ is $T_{Y/C}$-invariant if $r$ is even, and non-$T_{Y/C}$-invariant if $r$ is odd. Hence,
        $$a(E^{(1)};X^{(1)},X^{(1)}_\mathbf{0})
        =\begin{cases}
           n-1-\frac{p-1}{p}(r-1) & \text{if }r\text{ is odd,}\\
           p(n-1)-(p-1)(r-2) & \text{if }r\text{ is even.}
        \end{cases}$$
\end{example}   

\begin{example}\label{example:non_lc_family_curves_II}
    As a variation of the previous example, in any characteristic $p>0$ consider the family
    $f\colon X=\bA^{n+2}_{xyz_\bullet}\to \bA^1_t$ given by 
    $t=x(x^{p-1}+y^r)$ for $r\geq 1$. Let $E$ be the exceptional divisor obtained by blowing-up the ideal $(x,y,z_\bullet)$. Then $E$ is wild, it is non-$T_{X/C}$-invariant if and only if $p$ divides $r-1$, and
        $$\gamma_w(E)=
        \begin{cases}
            r-1 & \text{if }p\text{ divides }r-1, \\
            r-2 & \text{otherwise.}
        \end{cases}$$
\end{example}   



\begin{remark}[Is $T_{X/C}$ log canonical?]\label{rmk:when_relative_tangent_is_lc}
Suppose for simplicity that $\Delta=0$. 
Then it follows from \cite[Theorem 4.2.5]{Posva_Singularities_quotients_by_1_foliations} that $X^{(1)}\to C$ is locally stable if $T_{X/C}$ is log canonical (a slightly weaker property is in fact sufficient, see \cite{Hiroi_Quotients_by_adjoint_foliations}). 
We do not know whether local stability of $X\to C$ implies log canonicity of $T_{X/C}$. Note that for a vertical divisor $E$, it follows from \autoref{prop:can_div_relative_tangent_sheaf} 
that $a(E;T_{X/C})=a(E;X,X_\mathbf{0})-\gamma_w(E)$. 

In characteristic $0$, this implication is always true by \cite[Proposition 3.7]{ACSS_Positivity_moduli_part}. The proof relies on toroidalisation of morphisms in the sense of \cite{AK_Weak_semistable_reduction}. This is problematic in positive characteristic, essentially because we cannot get rid of units appearing in front of monomials---and it transpires that these units play a central role in the local analysis.

Javier Carvajal--Rojas, Samuele Ciprietti, Ghazaleh Fakhrivaighan and Anne Fayolle are developing a notion of $F$-splitting for (higher) foliations \cite{foliations_in_progress}: they can show that an $F$-split foliation is log canonical. However this does not always apply to our situation, see \autoref{rmk:non_T_split_deformations}.
\end{remark}

\subsubsection{Remarks on $T_{X/C}$-invariance}\label{section:invariance}
Let $t$ be a uniformizer of $\sO_{C,\mathbf{0}}$ and let $Y\to X$ be a normal birational model where a vertical wild divisor $E$ appears. Given a uniformizer $s$ of $\sO=\sO_{Y,E}$, we can write 
$t=us^n$
where $u\in \sO^\times$ and $n$ is divisible by $p$. By definition we can write 
        $$du=s^{\gamma_w(E)}\eta, \quad \eta\in \Omega^1_{\sO/k}\setminus \fm\Omega^1_{\sO/k}$$
where $\fm$ is the maximal ideal of $\sO$.
The canonical isomorphism $\Der_k(\sO)\cong \Hom_{\sO}(\Omega^1_{\sO/k},\sO)$ will be denoted by $\psi\mapsto \psi^\text{lin}$.

\begin{lemma}\label{lemma:description_relative_tangent_sheaf}
We have 
    $$\phi^*T_{X/C}\otimes \sO
=T_{Y/C}\otimes \sO=\{\theta\in \Der_k(\sO)\mid \theta(u)=0\}=\{\theta\in \Der_k(\sO)\mid \theta^\mathrm{lin}(\eta)=0\}.$$
\end{lemma}
\begin{proof}
By definition $\phi^*T_{X/C}$ is generically the sheaf of derivations which annihilate $t$, so it is equal to $T_{Y/C}$. As $t=us^n$ and $(n,p)>1$, it is also the set of derivations which annihilate $u$. In other words it is the set of derivations whose linearizations annihilate $du=s^{\gamma_w(E)}\eta$, equivalently $\eta$.
\end{proof}

Let $x_1,\dots,x_d\in \sO$ be such that $s,x_1,\dots,x_d$ is a differential basis of $\sO$ over $k$ (this is possible, since $\sO$ is regular and $k$ is perfect). We can write
        \begin{equation}\label{eqn:decomposition_eta}
        \eta=\alpha ds+\sum_{i=1}^d\beta_i dx_i,
        \quad \alpha,\beta_i\in \sO.
        \end{equation}
Since $\eta\notin \fm\Omega^1_{\sO/k}$, not all the coefficients $\alpha,\beta_i$ belong to $\fm$. We define:
        $$\mathcal{N}=\sO\cdot \eta + \sO\cdot ds
        \subseteq \Omega^1_{\sO/k}.$$
Then $\mathcal{N}$ is an $\sO$-submodule of $\Omega^1_{\sO/k}$, which may or not be saturated, of generic rank one or two.

\begin{proposition}\label{prop:u_s_indep_diff_coordinates}
Assume that $\mathcal{N}\subseteq \Omega^1_{\sO/k}$ is saturated of rank two. Then $E$ is not $T_{X/C}$-invariant.
\end{proposition}
\begin{proof}
By assumption $\Omega^1_{\sO/k}/\mathcal{N}$ is torsion-free, hence free. So $\eta,ds,\omega_2,\dots,\omega_d$ is an $\sO$-basis of $\Omega^1_{\sO/k}$ for some $1$-forms $\omega_2,\dots,\omega_d$. Let $\delta,\partial_s,\zeta_2,\dots,\zeta_d$ be the dual basis of $\Der_k(\sO)$. Then
        $$\delta(u)=
        \delta^\text{lin}\left(s^{\gamma_w(E)}\eta\right)
        =s^{\gamma_w(E)}, \quad 
        \partial_s(u)=\zeta_2(u)=\dots =\zeta_d(u)=0.$$
By \autoref{lemma:description_relative_tangent_sheaf} we see that $T_{Y/C}\otimes \sO$ is (freely) generated by $\partial_s,\zeta_2,\dots,\zeta_d$. As $\partial_s(s)=1$, we see that $E$ is not $T_{Y/C}$-invariant.
\end{proof}

\begin{proposition}\label{prop:inv_case_I}
Assume that $\mathcal{N}$ has generic rank two, but is not saturated. Then $E$ is $T_{X/C}$-invariant.
\end{proposition}
\begin{proof}
First we claim that in \autoref{eqn:decomposition_eta} we have $\beta_i\in \fm$ for all $i$. Indeed, as $\mathcal{N}$ is not saturated there exists a $1$-form $\xi$ such that $s\xi\in \mathcal{N}$ but $\xi\notin \mathcal{N}$. As $\mathcal{N}$ is also generated by $ds$ and $\sum_i \beta_idx_i$, we can write
        $$s\xi = ads+b\sum_i\beta_idx_i \quad 
        \text{for some }a,b\in \sO.$$
Since $s,x_1,\dots,x_d$ is a differential basis, the coefficients $a,b\beta_i$ must all belong to $\fm$. If some $\beta_i$ is invertible, then $b\in \fm$. Writing $a=sa'$ and $b=sb'$ with $a',b'\in \sO$, we obtain
        $$\xi=a'ds+b'\sum_i\beta_idx_i\in \mathcal{N},$$
contradiction. So every $\beta_i$ belongs to $\fm$; note that this implies that $\alpha\in \sO^\times$.

We can therefore write $\sum_i \beta_idx_i=s^r\theta$ where $\theta$ is a $1$-form not divisible by $s$, and $r>0$. We can find an $\sO$-basis $ds,\theta,\omega_2,\dots,\omega_d$ of $\Omega^1_{\sO/k}$, and we let $\partial_s,\delta,\zeta_2,\dots,\zeta_d$ be the dual basis of $\Der_k(\sO)$.
Observe that
        $$s^r=\delta^\text{lin}(s^r\theta)
        =\delta^\text{lin}\left(
            \eta -\alpha ds\right)
            =\delta^\text{lin}(\eta)$$
and that 
        $$\partial_s^\text{lin}(\eta)=
        \partial_s^\text{lin}(s^r\theta+\alpha ds)=\alpha,$$
so $\alpha\delta-s^r\partial_s$ belongs to $T_{Y/C}\otimes \sO$. 
As $\alpha$ is invertible, it follows that a basis of $T_{Y/C}\otimes \sO$ is given by $\alpha\delta-s^r\partial_s,\zeta_2,\dots,\zeta_d$. As
        $$\zeta_i(s)=0 \ \forall i, \quad 
        (\alpha\delta-s^r\partial_s)(s)=-s^r\in \fm,$$
we see that $E$ is $T_{X/C}$-invariant.
\end{proof}

\begin{remark}\label{rmk:simplification_inv_case_1}
In the situation of \autoref{prop:inv_case_I}, up to passing to $\widehat{\sO}$ we may assume that $\alpha=1$ in \autoref{eqn:decomposition_eta}. Indeed, if $\lambda\in \sO^\times$ then using the uniformizer $s'=\lambda s$ instead of $s$ yields
        $$du=s'^{\gamma_w(E)}\eta' \quad 
        \text{where} \quad 
        \eta'=\alpha \lambda^{\gamma_w(E)+1} ds'+ s'^{r'}\theta', \quad r'\geq 1.$$
As $\gamma_w(E)+1$ is not divisible by $p$ (\autoref{lemma:wild_coeff_not-1_mod_p}), after passing to the completion if necessary we may choose $\lambda$ so that $\alpha\lambda^{\gamma_w(E)+1}=1$.
\end{remark}

\begin{proposition}\label{prop:inv_case_II}
Assume that $\mathcal{N}$ has generic rank one. Then $E$ is $T_{X/C}$-invariant.
\end{proposition}
\begin{proof}
Suppose that $\mathcal{N}/(\sO\cdot \eta)$ is not zero: this means that $\sO\cdot \eta$ is not saturated in $\mathcal{N}$. However, it is saturated in $\Omega^1_{\sO/k}$, which is a contradiction: so $\mathcal{N}=\sO\cdot \eta$. This implies that $\eta=\alpha ds$ in \autoref{eqn:decomposition_eta} with $\alpha\in \sO^\times$, so in fact $\sN=\sO\cdot ds$. Take a differential basis $s,x_1,\dots,x_d$ of $\sO$ over $k$ and let $\partial_s,\partial_{x_1},\dots,\partial_{x_d}$ be the dual basis of $\Der_k(\sO)$. By \autoref{lemma:description_relative_tangent_sheaf}, $T_{Y/C}\otimes \sO$ is generated by the $\partial_{x_i}$s. These derivations satisfy $\partial_{x_i}(s)=0$,
so $E$ is $T_{X/C}$-invariant.
\end{proof}

\section{Some results on wild coefficients}\label{section:study_wild_coeff}
The purpose of this section is to study wild coefficients in some specific cases.

\subsection{Deformations of non-$F$-pure canonical surface singularities}\label{section:examples}
The first instances where \autoref{prop:permanence_F_pure_case} does not apply are deformations of some non-$F$-pure canonical surface singularities in low characteristics. In this subsection we give a few basic examples.

Let $A=k\llbracket x,y,z\rrbracket /(f)$ be a germ of a two-dimensional hypersurface singularity with an isolated singularity at the origin. Let $g_1,\dots,g_n$ be a $k$-basis of the (finite-dimensional) Jacobian ring
$A/(\partial_xf, \partial_yf,\partial_zf)$.
The \emph{miniversal deformation} of $A$ is the morphism
        $$\Spec\left(
        \frac{k\llbracket x,y,z,t_1,\dots,t_n\rrbracket}
        {(f+\sum_i t_ig_i)}
        \right)\longrightarrow
        \widehat{\bA}^n_{t_1,\dots,t_n}.$$
Any deformation of $A$ is, \'{e}tale-locally on the base, a base-change of the miniversal deformation. See \cite[\S 4.5]{Kollar_Mori_Birational_geometry_of_algebraic_varieties} for more details. So a typical one-parameter deformation $\mathcal{A}$ over $k\llbracket t\rrbracket$ of $A$ is obtained by selecting a curve 
        $$\widehat{\bA}^1_t\to \widehat{\bA}^n_{t_1,\dots,t_n},
        \quad t_i=\alpha_i t^{m(i)}
        \text{ with } \alpha_i\in k
        \text{ and }m(i)\in \bZ_{\geq 0}$$
in the miniversal deformation space, and pulling back the miniversal family along it. Assume in addition that $A$ is a canonical surface singularity. Then, up to change of coordinates, we can write $f$ in a standard form (see for example \cite[Theorem 0.4.13]{Cossec_Dolgachev_Liedtke_Enriques_surfaces_I}), which allows to write down explicitly the miniversal deformation.

If $p\leq 5$, then the canonical surface singularities $A$ are not necessarily $F$-pure; see for example \cite[Table 1]{TY_Quasi_split}. We consider some deformations of the non-$F$-pure ones. In several cases, blowing-up above the singular point of the special fiber $(t=0)$ eventually produces a wild divisor, but in the examples we computed the wild coefficients are always zero.

\begin{example}[$E^0_8$ singularity, $p=2$]
When $k$ has characteristic $2$, the $E^0_8$-singularity is defined by $f=z^2+x^3+y^5$ and is non-$F-$pure. We consider the deformation
        $$\mathcal{A}=k\llbracket x,y,z,t\rrbracket/(z^2+x^3+y^5+t^jg),$$
where $g=g(x,y,z)$ has non-zero image in the Jacobian ring of the $E^0_8$ singularity, and where $j\geq 1$ and $\mult_{(x,y,z,t)}t^jg\geq 3$. We assume that $g$ is general enough so that $\mathcal{A}$ has an isolated singularity at the origin.
By blowing-up the origin, we obtain a wild, non-$T$-invariant divisor $E$ with $\gamma_w(E)=0$.

Here are some details. A local chart for the blow-up is given by 
        $$Y=V(\bar{z}^2+\bar{x}+\bar{x}^3\bar{y}^5
        +\bar{x}\bar{h})\subset \widehat{\bA}^4, \quad 
        x=\bar{x},\  y=\bar{x}\bar{y},\ 
        z=\bar{x}\bar{z}, \ t=\bar{x}\bar{t},$$
where $\bar{h}$ is a regular function on $\widehat{\bA}^4$; it follows from the fact that $\mult_{(x,y,z,t)}t^jg\geq 3$ that $\bar{h}$ has no constant term. Then $Y$ is a normal threefold at the generic point of the exception divisor $E$.
The divisor $E$ is the trace on $Y$ of $V(\bar{x})$: so
        $$\sO:=\sO_{Y,E}=
        k\llbracket \bar{x},\bar{y},\bar{z},\bar{t}
        \rrbracket_{(\bar{x},\bar{z})}/
        (\bar{z}^2+\bar{x}+\bar{x}^3\bar{y}^5
        +\bar{x}\bar{h}).$$
In $\sO$ we have $\bar{z}^2=\bar{x}(1+\bar{x}^2\bar{y}^5+\bar{h})$, so $\bar{z}$ is a uniformizer of $\sO$ and $\bar{x}$ has order $2$. We have
        $$t=\bar{x}\bar{t}
        =\underbrace{\bar{t}(1+\bar{x}^2\bar{y}^5+\bar{h})^{-1}}_{=u\in\sO^\times}\cdot\bar{z}^2,$$
so $E$ is wild.
If $v=1+\bar{x}^2\bar{y}^5+\bar{h}$ then $\ord_E du=\ord_E d(v^2u)$. One computes $\ord_E\partial_{\bar{t}}(v^2u)=0$, which implies that $\ord_Edu=0$. 

We claim that $E$ is never $T_{\mathcal{A}/k\llbracket t\rrbracket}$-invariant. If $\bar{h}$ does not depend on $\bar{z}$ (equivalently if $g$ does not depend on $z$) then $\partial_{\bar{z}}\in \Der^\text{cont}(\sO)$ kills $u$, which implies that $E$ is non-invariant. More generally, observe that the derivations
$b\partial_{\bar{y}}+c\partial_{\bar{z}}$ on $\widehat{\bA}^4$ that annihilate $\bar{x}^2\bar{y}^5+\bar{h}$ restrict to (continuous) derivations on $\sO$. In particular, 
    $$\psi=\partial_{\bar{z}}(\bar{h})\cdot \partial_{\bar{y}}
    +(\bar{x}^2\bar{y}^4+\partial_{\bar{y}}(\bar{h}))\cdot \partial_{\bar{z}}$$
defines a continuous derivation on $\sO$. We have
$\psi(u)=v^{-2}\psi(v^2u)=0$ while $\psi(\bar{z})\notin \bar{z}\sO$, which shows that $E$ is non-invariant.
\end{example} 

\begin{example}[$E^0_6$ singularity, $p=3$]
When $k$ has characteristic $3$, let us consider the deformation
        $$\mathcal{A}=k\llbracket x,y,z,t\rrbracket/(z^2+x^3+y^4+t^2)$$
of the $E^0_6$-singularity, which is non-$F$-pure. The origin is an isolated singular point. We blow-up four times, along centers that are infinitely near the origin. The last exceptional divisor is wild, non-$T_{\mathcal{A}/k\llbracket t \rrbracket}$-invariant, with wild coefficient $0$.

More precisely, the $y$-chart of the first blow-up contains an isolated singular point, but the exceptional divisor is not wild. Three more blow-ups, the second and third ones at the unique singular point of the $x$-chart, produce a regular chart given by
        $$V(\bar{z}^2+\bar{y}+\bar{y}^2+\bar{t}^2)\subset \widehat{\bA}^4, \quad t=\bar{t}\bar{y}\bar{x}^6.$$
The last exceptional divisor is $E=V(\bar{x})$. Then $u=\bar{t}\bar{y}$, so we see that $E$ is non-invariant with $\gamma_w(E)=0$.
\end{example}

\begin{example}[$E^0_6$ singularity, $p=3$]
We modify slightly the previous example: let us consider
        $$\mathcal{A}=k\llbracket x,y,z,t\rrbracket/(z^2+x^3+y^4+t^3).$$
Changing coordinates $w=x+t$, we write it as
        $$\mathcal{A}=k\llbracket x,y,z,w\rrbracket/(z^2+w^3+y^4), \quad 
        t=w-x.$$
Up to completion $\mathcal{A}$ is the product of the $E^0_6$-singularity with $k[x]$, so there is a resolution of singularities $Y=X\widehat{\otimes}\bA^1_x\to \Spec(\mathcal{A})$ where $X$ is the minimal resolution of the $E^0_6$-singularity. Since $t=w-x$ and the resolution does not depend on the variable $x$, we see that $Y$ has no wild divisor.

Instead, let us blow-up $(x,y,z,w)$ with weights $(3,1,1,3)$. On the schematic chart where the exceptional divisor is $E=V(\bar{y})$, the equation of the blow-up is $\bar{z}^2+\bar{y}^2(1+\bar{y}^5\bar{w}^3)=0$. The function $\bar{v}=\bar{z}/\bar{y}$ belongs to the normalization, which is then given by the spectrum of a suitable completion of
$k[\bar{x},\bar{y},\bar{v},\bar{w}]/(\bar{v}^2+1+\bar{y}^5\bar{w}^3)$. The exceptional divisor is $E=V(\bar{y})$ and 
$t=\bar{v}^3(\bar{w}-\bar{x})\bar{y}^3$. So $E$ is wild, and it is easy to see that $\gamma_w(E)=0$ and that $E$ is non-$T_{\mathcal{A}/k\llbracket w \rrbracket}$-invariant.
\end{example}

\begin{remark}\label{rmk:non_T_split_deformations}
Although one-parameter deformations $X\to C$ of non-$F$-pure singularities are never $F$-locally stable, one can ask whether the foliation $T_{X/C}$ is $F$-split in the sense of \cite{foliations_in_progress}.
Using a foliated Fedder criterion, the authors of \emph{op.cit.} checked in several cases that $T_{X/C}$ is not $F$-split. In fact, one could conjecture that this happens for a general (or indeed every) one-parameter deformation of such surface singularities.
\end{remark}

\subsection{Generalities on pullback of differential operators}\label{section:diff_operators}
If $(X,\Delta)\to C$ is $F$-locally stable (\autoref{d-F-locallystable}), we saw that, thanks to $F$-pure inversion of adjunction, the family $(X^{(1)},\Delta^{(1)})\to C$ is still $F$-locally stable (\autoref{prop:permanence_F_pure_case}). 
Recall that $F$-purity is characterized by the existence of a differential operator on $X$ satisfying a splitting property. We will study in the next sub-section how the existence of such operators influences the values of wild coefficients. To prepare this analysis, we first gather some general facts on differential operators.

In this subsection we let $(X=\Spec(A),\Gamma)$ be a normal pair, $\phi\colon Y\to X$ be a normal birational model, and $E\subset Y$ be a $\phi$-exceptional divisor. We write $\sO=\sO_{Y,E}$, and let $s$ be a uniformizer of this DVR.

We refer to \cite[\S 16.8]{EGA_IV.4} for the general theory of (Grothendieck) differential operators. 
Recall that since $\sO$ is smooth over $k$, the $\sO$-module $\Omega^1_{\sO/k}$ and its $\sO$-dual $\Der_k(\sO)$ are free of dimension $\dim Y$.

To state the following lemma, we use the following notation: 
$\sO\langle X_0,\dots,X_d\rangle$ denotes the ring of polynomials in the \emph{non-commutative} variables $X_0,\dots, X_d$ and with coefficients in $\sO$. 

\begin{lemma}\label{lemma:decomposition_diff_operator}
Let $\zeta_0,\dots,\zeta_{d}$ be any $\sO$-basis of $\Der_k(\sO)$. If $D\in \Diff_k(\sO)$ is an $\sO^p$-linear differential operator, then we can write $D$ as a polynomial
        \begin{equation}
        D=P(\zeta_0,\dots,\zeta_d), \quad
        P\in \sO\langle X_0,\dots,X_d\rangle \text{ where every power is }<p.
        \end{equation}
Moreover, the constant coefficient of $P$ is zero if and only if $D|_{\sO^p}=0$.
\end{lemma}
\begin{proof}
Since $\sO$ is smooth over $k$, we can find a differential basis $(x_0,\dots,x_{d})=\mathbf{x}$ of $\sO$ over $k$. By \cite[\S 16.11]{EGA_IV.4}, any differential operator $L$ of $\sO$ over $k$ of order at most $m$ can be written in a unique way
        $$L=
        \sum_{|\mathbf{i}|\leq m} \alpha(L,\mathbf{x})_\mathbf{i}
        \cdot D^{\mathbf{x}}_\mathbf{i}, \quad 
        \alpha(L,\mathbf{x})_\mathbf{i}\in \sO,$$
where by definition we have
        $$D^{\mathbf{x}}_\mathbf{i}(\mathbf{x}^\mathbf{j})
        =\binom{\mathbf{j}}{\mathbf{i}}\mathbf{x}^{\mathbf{j}-\mathbf{i}}, \quad \text{where }\binom{\mathbf{j}}{\mathbf{i}}=0 \text{ if }\mathbf{i}\not\leq \mathbf{j}.$$
We claim that $L$ is $\sO^p$-linear if and only if $\alpha(L,\mathbf{x})_\mathbf{i}=0$ whenever $i_j\geq p$ for some $j$. 
The \emph{if} direction is clear, because in case $(i_0,\dots,i_d)<(p,\dots,p)$ we have $D^\mathbf{x}_\mathbf{i}=\frac{1}{\mathbf{i}!}\prod_{j=1}^d\partial_{x_j}^{i_j}$ where $\partial_{x_1},\dots,\partial_{x_d}$ is the basis of $\Der_k(\sO)$ dual to $dx_1,\dots,dx_d\in \Omega_{\sO/k}^1$, and therefore $L$ is a linear combination of $\sO^p$-linear operators. Conversely, suppose that $L$ is $\sO^p$-linear. The operator
        $$\sum_{i_0,\dots,i_d<p} \alpha(L,\mathbf{x})_\mathbf{i}
        \cdot D^{\mathbf{x}}_\mathbf{i}$$
is $\sO^p$-linear by what we have just said; subtracting it from $L$, we may assume that $\alpha(L,\mathbf{x})_\mathbf{i}=0$ when $i_0,\dots,i_d<p$. This implies that $L(\mathbf{x}^\mathbf{i})=0$ whenever $(i_0,\dots,i_d)< (p,\dots,p)$. For the sake of contradiction, suppose that after these reductions $L\neq 0$, and take $\mathbf{j}$ minimal for the lexicographic order with the property $\alpha(L,\mathbf{x})_\mathbf{j}\neq 0$. We can write
        $$\mathbf{x}^\mathbf{j}=v^p\cdot \mathbf{x}^\mathbf{i} \quad \text{where }
        v\in \sO \  \text{ and } \ 0\leq i_0,\dots,i_d<p.$$
Then $L(\mathbf{x}^\mathbf{j})=\alpha(L,\mathbf{x})_\mathbf{j}\neq 0$ by minimality, while by $\sO^p$-linearity
        $$L(\mathbf{x}^\mathbf{j})=v^pL(\mathbf{x}^\mathbf{i})=0,$$
which yields a contradiction. This proves our claim.

We can therefore write
        $$D=\sum_{0\leq i_0,\dots,i_d<p}\alpha_\mathbf{i}D^{\mathbf{x}}_\mathbf{i},
        \quad \alpha_\mathbf{i}\in \sO.$$
Since $\zeta_0,\dots,\zeta_d$ is also an $\sO$-basis of $\Der_k(\sO)$, each $\partial_{x_i}$ is an $\sO$-linear combination of the $\zeta_j$s, and therefore each $D^\mathbf{x}_\mathbf{i}$ with $i_0,\dots,i_d<p$ can be written as a polynomial in the $\zeta_j$s. Substituting into the above expression for $D$ yields an expression
        $$D=P(\zeta_0,\dots,\zeta_d)\quad 
        \text{for some }P\in \sO\langle X_0,\dots,X_d\rangle.$$
Suppose that $\zeta_i^{p+n}$, for some $i$ and $n\geq 0$, appears in this expression. Since $\zeta_i^p\in \Der_k(\sO)$, we can write it as a linear combinations of the $\zeta_j$s, and modify the polynomial $P$ accordingly. After a finite number of such modifications, all the powers of the variables in $P$ are $<p$: this proves the first assertion. The second one is immediate.
\end{proof}

The abelian group $\Diff_k(\sO)$ is an $\sO$-bimodule: given $L\in \Diff_k(\sO)$ and $x\in \sO$, we have 
        $$\text{left structure:} \ \ \ xL\colon y\mapsto x\cdot L(y), \quad 
        \text{right structure:} \ \ \ L* x\colon y\mapsto L(xy).$$
If $L$ is $\sO^p$-linear and $x\in \sO^p$ then $L*x=xL$, but the two module structures are quite different.

Recall that $\Diff_k(A)=\bigcup_{e\geq 0}\Hom_{A^{p^e}}(A,A)$, see for example \cite[\S 2.5]{Smith_cdBergh_Simplicity_diff_operators_in_pos_char}. So any $k$-differential operator of $A$ is $A^{p^e}$-linear for some $e\geq 0$, and therefore it extends uniquely to a $\Frac(A)^{p^e}$-linear endomorphism of $\Frac(A)$. Since $\Frac(A)=\Frac(\sO)$, it can be considered as a \emph{rational} differential operator of $\sO$, i.e.\ as an element of $\Frac(\sO)\otimes_\sO\Diff_k(\sO)$ where the tensor product is taken with respect to the left $\sO$-module structure of $\Diff_k(\sO)$.

\begin{definition}\label{def:discrepancy_diff_operator}
Let $D\in \Diff_k(A)$ be a non-zero differential operator of $A$ over $k$. We let 
$a(E;D)\in \bZ$ be the minimal integer such that 
$s^{a(E;D)}D\in \Diff_k(\sO)$, where $s$ is a uniformizer of $\sO$.
\end{definition}

Let us fix some notations: let $e\geq 1$, $q=p^e$ and $X^e=(|X|,\sO_X^{q})$. We define $\phi^e\colon Y^e\to X^e$ similarly. Let also $\Gamma^e$ be the divisor corresponding to $\Gamma$ on $X^e$.
The following result is a well-known consequence of duality along the $k$-linear Frobenius $F^e\colon X\to X^e$:

\begin{lemma}\label{lemma:finite_duality_along_Frobenius}
Assume that $(q-1)\Gamma$ is integral. Then there is an isomorphism
        $$\alpha_e\colon \Hom_{\sO_{X^{e}}}(\sO_X((q-1)\Gamma), \sO_{X^{e}})
        \cong \left(
        \sO_{X^e}\big((1-q)(K_{X^{e}}+\Gamma^e\big)
        \right)^{1/q}$$
where we consider $\sO_{X^e}((1-q)(K_{X^{e}}+\Gamma^e)$ as a sub-sheaf of $K(X^e)$ to apply the exponent $1/q$.
\end{lemma}

An element of $\Hom_{\sO_{X^{e}}}(\sO_X((q-1)\Gamma), \sO_{X^{e}})$ induces an $\sO_X^{p^e}$-linear differential operator of $\sO_X$, by restricting its domain along the inclusion $\sO_X\hookrightarrow \sO_X((q-1)\Gamma)$. These differential operators have the property, remarkable in our context, that their discrepancy $a(E;-)$ can be compared with $a(E;X,\Gamma)$: this is the content of \autoref{prop:bound_on_pole} below. We indicate three proofs of that result. The first is implicitly contained in \cite[First proof of Theorem 3.3]{Hara_Watanabe_F_sing_vs_lt_and_lc}, and while it is longer, we feel that it is the most transparent about the relations between the discrepancies.

\begin{proposition}\label{prop:preparation_bound_on_pole}
Assume that $(q-1)\Gamma$ is integral, and let $\varrho\in \Hom_{\sO_{X^e}}(\sO_X((q-1)\Gamma),\sO_{X^e})$ be non-zero. Consider $\varrho$ as a rational section of $\sHom_{\sO_{Y^e}}(\sO_Y,\sO_{Y^e})$ on $Y$ and let $i\in \bZ$ be the smallest integer for which $\varrho*s^i$ is a regular section of $\sHom_{\sO_{Y^e}}(\sO_Y,\sO_{Y^e})$ at the generic point of $E$.

Let $r>0$ be the Cartier index of $K_X+\Gamma$, and write $q-1=m_er+i_e$ with $0\leq i_e<r$. Then there exists a non-negative integer $o$, independent of $\varrho$, such that
        $$i\leq o+m_er\cdot a(E;X,\Gamma).$$
\end{proposition}
\begin{proof}
We start with some preparations:
\begin{itemize}
    \item We may assume that $Y=\Spec(\sO)$. Then $s^q=s_e$ is a uniformizer of $\sO_{Y^e}$. By finite duality we have an isomorphism
    $$\sHom_{Y^e}(\sO_Y, \sO_{Y^e})=
        \sO_{Y^e}((1-q)K_{Y^e})^{1/q}$$
compatible with the isomorphism of \autoref{lemma:finite_duality_along_Frobenius}. By abuse of notation, we call this isomorphism $\alpha_e$ as well.
As $Y$ is Gorenstein, these sheaves are invertible as $\sO_Y$-modules. The $\sO_Y$-module structure on $\sHom_{\sO_{Y^e}}(\sO_Y, \sO_{Y^e})$ is given by pre-composition, i.e.\ by the right structure $[(x,L)\rightsquigarrow L*x]$ mentioned before.

    \item Choose an element $b\in A$ such that 
$b\cdot H^0(X,\sO(-j (K_X+\Gamma)))\subseteq H^0(Y,\omega_Y^{-j})$ for all $0\leq j<r$. 
Then $b^q\cdot H^0(X^e, \sO(-j(K_{X^e}+\Gamma^e)))\subseteq H^0(Y^e,\omega_{Y^e}^{-j})$ for all $0\leq j<r$. 
We define $o=\ord_Eb=\ord_{E^e}b^q$, and write $b^q=b_e$ below. Notice that $o$ depends only on $(X,\Gamma)$ and $\sO_{Y,E}$.
    
    \item Let us fix an embedding of $\sO_X(K_X+\Gamma)$ inside the function field of $X$, so that we can find $w\in K(X)$ with
                $$\sO_X(r(K_X+\Gamma))=w\sO_X.$$
Then $\sO_{X^e}(r(K_{X^e}+\Gamma^e))=w^q\sO_{X^e}$. We write $w^q=w_e$ below.

    \item Finally, we have $a(E^e;X^e,\Gamma^e)=a(E;X,\Gamma)$: we denote this rational number by $a$.
\end{itemize}

We can now start the proof. By definition of the discrepancy we have
        $$\sO(rK_{Y^e})=\phi^{e, *}\sO(r(K_{X^e}+\Gamma^e)+raE^e)
        =s_e^{-ra}w_e\sO_{Y^e}.$$
It follows that
        \begin{eqnarray*}
            \sO((1-q)K_{Y^e})^{1/q} &=&
            \sO(-m_erK_{Y^e}-i_eK_{Y^e})^{1/q} \\
            &=& \left( 
            s_e^{m_era}w_e^{-m_e}\cdot \omega_{Y^e}^{-i_e}
            \right)^{1/q}.
        \end{eqnarray*}
Therefore, if $\theta_e$ is a generator of the invertible $\sO_Y$-module
$\sHom_{\sO_{Y^e}}(\sO_Y,\sO_{Y^e})$ then the previous equation implies that
        \begin{equation}\label{eqn:1st_approximation}
        \sO_{Y}\cdot \alpha_{e}(\theta_e)
        =\left(
        s_e^{m_era}w_e^{-m_e}\cdot \omega_{Y^{e}}^{-i_e}
        \right)^{1/q}
        =
        s^{m_era}\cdot \left(
        w_e^{-m_e}\cdot \omega_{Y^{e}}^{-i_e}
        \right)^{1/q}.
        \end{equation}
If $\varrho$ is as in the statement, then 
        \begin{eqnarray*} 
        \alpha_e(\varrho)\in 
        \phi^{e,*}\sO_{X^e}((-m_er-i_e)(K_{X^e}+\Gamma^e))^{1/q} &=&
        \left(w_e^{-m_e}\phi^{e, *}\sO_{X^e}(-i_e(K_{X^e}+\Gamma^e))\right)^{1/q} \\
        &\subseteq & 
        \left(w_e^{-m_e}b_e^{-1}H^0(Y^e,-i_eK_{Y^e})\right)^{1/q}.
        \end{eqnarray*}
This implies
        $$\alpha_e(\varrho)\in \left(
        s_e^{-o}w_e^{-m_e}\cdot \omega_{Y^e}^{-i_e}
        \right)^{1/q},$$
and thus by the $\sO_{Y}$-linearity of $\alpha_e$:
        \begin{equation*}\label{eqn:2nd_approximation}
        \alpha_e(\varrho*s^i)\in 
        s^{-o+i}\cdot
        \left(
        w_e^{-m_e}\cdot \omega_{Y^e}^{-i_e}
        \right)^{1/q}.
        \end{equation*}
By minimality of $i$, we have 
$\varrho*s^i=\text{(unit)}\cdot \theta_e$.
Therefore
        \begin{equation}\label{eqn:2nd_approximation}
        \alpha_e(\theta_e)\in 
        s^{-o+i} \cdot \left(
        w_e^{-m_e}\cdot \omega_{Y^e}^{-i_e}
        \right)^{1/q}.
        \end{equation}
By comparing the powers of $s$ in \autoref{eqn:1st_approximation} and \autoref{eqn:2nd_approximation} we see that necessarily $-o+i\leq m_era$. This yields the desired inequality.
\end{proof}

Now suppose that we are given 
$\varrho\in \Hom_{\sO_{X^e}}(\sO_X((q-1)\Gamma),\sO_{X^e})$. 
Let $i\in \bZ$ be minimal such that $\varrho*s^i$ belongs to $\Hom_{\sO^q}(\sO,\sO^q)$. The image of $\varrho*s^i\colon \sO_{Y,E}\to \sO_{Y^e,E^e}$ is an ideal, by the $\sO_{Y^e}$-linearity of $\varrho$. If it is contained in $s_e\sO_{Y^e,E^e}$ (where $s_e=s^q$, as in the proof of \autoref{prop:preparation_bound_on_pole}) then 
$s^{-q}(\varrho*s^i)=\varrho*s^{i-q}$ would be regular at the generic point of $E$, contradicting the minimality of $i$. Therefore $\varrho*s^i\colon \sO_{Y,E}\to \sO_{Y^e,E^e}$ is surjective. 

More generally, for $n\geq 1$, we denote by
$\varrho_n\in\Hom_{\sO_{X^{e(n+1)}}}(\sO_{X^{ne}},\sO_{X^{e(n+1)}})$ the operator defined by 
        $$\varrho_n(x)=\varrho\left(x^{1/q^n}\right)^{q^n}, \quad x\in \sO_{X^{ne}}=\sO_X^{q^n}.$$
Then $i$ is the minimal integer such that $\varrho_n*s_{ne}^i$ is regular at the generic point of $E^{ne}$, and sends $\sO_{Y^{ne},E^{ne}}$ onto $\sO_{Y^{e(n+1)},E^{e(n+1)}}$.

Consider the composition 
    $$\boldsymbol{\varrho}=\varrho_{ne}\circ\dots\circ \varrho_n\circ\varrho
    \in \Hom_{\sO_{X^{e(n+1)}}}(\sO_X, \sO_{X^{e(n+1)}}).$$
By what we have said above, the minimal integer $v$ such that $\boldsymbol{\varrho}*s^v$ is a regular element of
$\sHom_{\sO_{Y^{e(n+1)}}}(\sO_Y,\sO_{Y^{e(n+1)}})$ at the generic point of $E$ is
$v=i+qi+\dots+q^{n}i$, since
    $$\boldsymbol{\varrho}*s^{i+\dots+q^{n}i}
    =(\varrho_{n}*s_{ne}^i)\circ \dots 
    (\varrho_1*s_e^i)\circ 
    (\varrho*s^i).$$
By \autoref{prop:bound_on_pole} we obtain
        $$i(1+q+\dots+q^n)\leq o+m_{e(n+1)}r\cdot a(E;X,\Gamma).$$
By writing down the closed form of the geometric series, we get
        $$\frac{i}{q-1}\leq \frac{o}{q^{n+1}-1}+
        \lfloor \frac{q^{n+1}-1}{r}\rfloor \cdot \frac{r}{q^{n+1}-1}\cdot a(E;X,\Gamma).$$
By letting $n$ go to infinity, we have therefore proved:

\begin{proposition}\label{prop:bound_on_pole}
Assume that $(p^e-1)\Gamma$ is integral and let $0\neq \varrho\in \Hom_{\sO_{X^e}}(\sO_X((p^e-1)\Gamma),\sO_{X^e})$. Let $i\in \bZ$ be minimal such that $\varrho*s^i\in \Hom_{\sO_{Y^e,E^e}}(\sO_{Y,E},\sO_{Y^e,E^e})$. Then:
        \begin{equation}\label{eqn:bound_on_pole_2}
            \frac{i}{p^e-1}\leq a(E;X,\Gamma).
        \end{equation}
\end{proposition}

The second proof of \cite[Theorem 3.3]{Hara_Watanabe_F_sing_vs_lt_and_lc} gives a quicker and more ``geometric'' approach to \autoref{prop:bound_on_pole}:


\begin{proof}[Second proof of \autoref{prop:bound_on_pole}]
As in the previous proof of \autoref{prop:bound_on_pole}, we let $\alpha_e$ be the isomorphism defined by finite duality $$\sHom_{Y^e}(\sO_Y, \sO_{Y^e})=
        \sO_{Y^e}((1-q)K_{Y^e})^{1/q}$$
compatible with the isomorphism of \autoref{lemma:finite_duality_along_Frobenius}.
Note that $\alpha_e(\varrho)$ is a rational section of $\left( \sO_{Y^e}((1-q)(K_{Y^e}+\phi^{e,-1}_*\Gamma^e)) \right)^{1/q}$, therefore it defines an integral, not necessarily effective divisor on $Y^e$, that we denote by $\mathrm{div}(\alpha_e(\varrho))$.
On the other hand, $\alpha_e(\varrho*s^{i}) \in H^0(Y^e, ((1-q)(K_{Y^e}+\phi^{e,-1}_*\Gamma^e))^{1/q}) \cong H^0(Y^e, ((1-q)(K_{Y^e}+\phi^{e,-1}_*\Gamma^e)))$, and, since $s\in \cO_Y$, we have $\mathrm{div}(\alpha_e(\varrho))=\mathrm{div}(\alpha_e(\varrho*s^{i}))-iE$.
Let $D:= \mathrm{div}(\alpha_e(\varrho))=\mathrm{div}(\alpha_e(\varrho*s^{i}))-iE$ and $B:=D/(q-1)$.
    By our assumption, $B \sim_{\bQ} -(K_Y+\phi^{-1}_*\Gamma)$, hence $\phi_*B\sim_{\bQ}-(K_X+\Gamma)$ is $\bQ$-Cartier.
    Note that $\phi_*B \geq 0$ since $B$ is effective outside $E$.
    In particular, the $\bQ$-divisor $B-\phi^*\phi_*B+aE\sim_{\bQ} -\phi^*(K_X+\Gamma)-\phi^*\phi_*B$ is $\phi$-trivial and $\phi$-exceptional (where $a= a(E;X,\Gamma)$).
    By the negativity Lemma, we conclude that it is $0$.
    However, $\phi^*\phi_*B\geq 0$, yielding that $i/(q-1) \leq a$.
\end{proof}

\begin{remark}
One can also give a short third proof of \autoref{prop:bound_on_pole} by applying \cite[Chapter 6, Lemma 4.15]{Schwede_Smith_book}: we leave the details to the reader.
\end{remark}

\subsection{Wild coefficients in the $F$-split case}\label{section:wild_coeff_F-split_case}

\begin{definition}[{\cite[Chapter 5, Definition 3.7]{Schwede_Smith_book}}]\label{def:F-split_along_div}
    Let $(X=\Spec(A),\Delta)$ be a normal pair and $D$ be an effective $\bQ$-divisor.
    We say that \emph{$(X,\Delta)$ is $e$-$F$-split along $D$} if the natural injection
    \[
    A^{p^e} \hookrightarrow A\left(
    \lceil (p^e-1)\Delta\rceil +D
    \right)
    \]
    splits.
\end{definition}

\begin{remark}\label{caution}
 If $(X,\Delta)$ is $e$-$F$-split along $D$ and $e' <e$, then it is not necessarily true that $(X,\Delta)$ is $e'$-$F$-split along $\lceil(p^{e'}-1)\Delta\rceil$ (see \cite[Caution 4.8]{Schwede_Smith_book}).
\end{remark}

Let us now consider $f\colon (X=\Spec(A),\Delta)\to C=\Spec(\sO_{C,\mathbf{0}})$ a strongly locally stable family of pairs. Let $\phi\colon Y\to X$ be a normal birational model, $E\subset Y$ a wild divisor. We let $t$ be a uniformizer of $\sO_{C,\mathbf{0}}$ and $s$ be a uniformizer of $\sO=\sO_{Y,E}$: we can write 
$t=us^n$ where $u\in \sO^\times$ and $n$ is divisible by $p$ and $du=s^{\gamma_w(E)}\eta$ in $\Omega^1_{\sO/k}$, where the $1$-form $\eta$ is not divisible by $s$.
We assume:

\begin{assumption}\label{assumption:F_split_along_div}
The pair $(X,\Delta)$ is $1$-$F$-split along $X_\mathbf{0}$. 
\end{assumption}

This assumption implies that $(X, \Delta)$ is $F$-pure and that the $F$-pure threshold\footnote{For the definition of $F$-pure threshold see \cite[Definition 5.4.29]{Schwede_Smith_book}.} $\mathrm{fpt}(X_\mathbf{0};X,\Delta) \geq 1/(p-1)$.

Then, we have an $A^p$-linear operator $\psi\colon A(\lceil(p-1)\Delta\rceil+ X_\mathbf{0}))\to A^p$ that splits the natural injective map $A^p \hookrightarrow A(\lceil(p-1)\Delta\rceil+X_\mathbf{0}))$ as in \autoref{def:F-split_along_div}. 
Consider the composition
        $$\varphi\colon A\hookrightarrow A\left( 
        \lceil(p-1)\Delta\rceil + X_\mathbf{0}\right)
        \overset{\psi}{\longrightarrow}
        A^p\hookrightarrow A,
        \quad \text{and let} \quad 
        \varsigma=\varphi-\id.$$
Note that by $A^p$-linearity we have $\varsigma|_{A^p}=0$.

\begin{lemma} \label{l-7.1.2b}
We have $\varsigma(t^{p-1})=\lambda t^{p-1}$ where $\lambda\in A$ is invertible in a neighbourhood of $X_\mathbf{0}$.
\end{lemma}

\begin{proof}
The element $1/t$ belongs to $A((p-1)\Delta +X_\mathbf{0}))$. Hence by $A^p$-linearity
        $$\varsigma(t^{p-1})=\varphi(t^{p-1})-t^{p-1}
        =t^p\varphi(1/t)-t^{p-1}
        = t^{p-1}\cdot \underbrace{(-1+t\varphi(1/t))}_{=\lambda}$$
and by definition $\lambda\equiv -1$ modulo $t$.
\end{proof}

We consider $\varsigma$ as a rational differential operator on $\sO$: we can write
        $$\varsigma=s^{-a(E;\varsigma)}D, \quad 
        D\in \End_{\sO^p}(\sO),$$
where $a(E;\varsigma)\in \bZ$ is minimal for this property (see \autoref{def:discrepancy_diff_operator}). Note that $D|_{\sO^p}=0$.

\begin{lemma} \label{lemma:N_is_p_div}
We have $D(u^{p-1})=s^{a(E;\varsigma)}\lambda u^{p-1}$, $\lambda\in \sO^\times$, and $a(E;\varsigma)=a(E;\varphi) \in p\bZ_{\geq 0}$.
\end{lemma}

\begin{proof}
By \autoref{l-7.1.2b} and $\sO^p$-linearity of $D$ we have
        \begin{eqnarray*} 
        s^{a(E;\varsigma)}\lambda t^{p-1} = s^{a(E;\varsigma)}(\phi^*\varsigma)(t^{p-1})
        &=& D(t^{p-1})\\
        &=& s^{(p-1)n}D(u^{p-1}) \\
        &=& u^{-p+1}t^{p-1}D(u^{p-1})
        \end{eqnarray*}
which simplifies into the first claim. 
Since the center of $E$ on $X$ is contained in $X_\mathbf{0}$, the element $\lambda$ is invertible in $\sO$. Since $D$ is regular, we must have $a(E;\varsigma)\geq 0$. As $\phi^*\varsigma=\phi^*\varphi-\id$, this implies that $a(E;\varsigma)$ is also the smallest integer for the property that $s^{a(E;\varsigma)}\phi^*\varphi$ is regular at the generic point of $E$. Since $\varphi$ takes values in $A^{p}$, we have $\phi^*\varphi(\sO)\subseteq \Frac(\sO^{p})$. As $a(E;\varsigma)=a(E;\varphi)=-\min\{\ord_\sO\varphi(h)\mid h\in \sO\}$, we deduce that $a(E;\varsigma)\in p\bZ_{\geq 0}$.
\end{proof}

We can bound $\gamma_w(E)$ in terms of $a(E;\varsigma)$:

\begin{lemma}\label{lemma:bound_non_inv_case}
Assume that $E$ is non-$T_{Y/C}$-invariant, or more generally that $\gamma_w(E)\in p\mathbb{Z}_{\geq 0}$ (see \autoref{cor:wild_coeff_p_div_in_non_inv_case}). Then
$\gamma_w(E)\leq a(E;\varsigma)$. 
\end{lemma}

\begin{proof}
The element $\eta\in \Omega^1_{\sO/k}$, defined as in \autoref{section:invariance}, is not divisible by $s$, so it is part of an $\sO$-basis $\eta,\omega_1,\dots,\omega_d$ of $\Omega^1_{\sO/k}$. Let $\delta,\zeta_1,\dots,\zeta_d$ be the dual basis of $\Der_k(\sO)$. Then $\delta(u)=s^{\gamma_w(E)}$ and $\zeta_i(u)=0$ for all $i$. Furthermore, as $\gamma_w(E)$ is divisible by $p$ we have $\zeta_i(s^{\gamma_w(E)})=0$ for all $i$ and $\delta^j(u)=0$ for $j>1$. 

By \autoref{lemma:decomposition_diff_operator} we can write
$D=P(\delta,\zeta_1,\dots,\zeta_d)$ where $P$ is a polynomial in non-commutative variables with coefficients in $\sO$. 
Moreover, $P$ has no constant term since $D|_{\sO^p}=0$. By what we have said, when we evaluate at $u^{p-1}$ and we get
        $$D(u^{p-1})=P(\delta,\zeta_1,\dots,\zeta_d)(u^{p-1})
        =\alpha \delta(u^{p-1})=-\alpha u^{p-2}s^{\gamma_w(E)} 
        \quad \text{for some }\alpha\in \sO.$$
Thus 
        $$\lambda u^{p-1}s^{a(E;\varsigma)}=D(u^{p-1})
        =-\alpha u^{p-2} s^{\gamma_w(E)}$$
which implies that $\gamma_w(E)\leq a(E;\varsigma)$.
\end{proof}

\begin{lemma}\label{lemma:bound_inv_case}
Assume that $E$ is $T_{Y/C}$-invariant. Then 
$\gamma_w(E)\leq a(E;\varsigma)+p-2$.
\end{lemma}

\begin{proof}
The invariant case occurs precisely in the cases described in \autoref{prop:inv_case_I} and \autoref{prop:inv_case_II}. 
    \begin{itemize}
        \item Suppose that we are in the situation of \autoref{prop:inv_case_I}: we use the notation of its proof. By \autoref{rmk:simplification_inv_case_1}, after passing to $\widehat{\sO}$ if necessary, we may also assume that $\alpha=1$. Then we have a basis $\partial_s,\delta,\zeta_2,\dots,\zeta_d$ of $\Der_k(\sO)$, and these derivations satisfy
            $$\partial_s(u)= s^{\gamma_w(E)},
            \quad \delta(u)=s^r, \quad 
            \zeta_i(u)=0 \text{ for all } i$$
        where $r>0$, and
            $$\partial_s(s)=1, \quad \delta(s)=0, 
            \quad \zeta_i(s)=0 \text{ for all }i.$$
        Thanks to \autoref{lemma:decomposition_diff_operator}, we can write 
        $D=P(\partial_s,\delta,\zeta_2,\dots,\zeta_d)$ where $P$ is a polynomial in non-commutative variables
        with powers $<p$ and
        without constant term. Evaluating at $u^{p-1}$ yields:
            \begin{eqnarray*}
        \lambda u^{p-1} s^{a(E;\varsigma)}
        &=& D(u^{p-1})\\ &=&
        \sum_{l= 1}^{p-1} \alpha_{l} \partial_s^l(u^{p-1})
        +\sum_{j= 0}^{p-1}
        \beta_j
        (\partial_s^j\circ \delta) (u^{p-1}) \\
        &=& 
        \sum_{l= 1}^{p-1} \alpha_{l} \partial_s^l(u^{p-1})-\sum_{j= 0}^{p-1}
        \beta_j
        \partial_s^j\left(s^r u^{p-2} \right) \\
        &=& 
        \sum_{l= 1}^{p-1} \alpha_{l} \partial_s^l(u^{p-1})-\sum_{j= 0}^{p-1}
        \beta_j\sum_{k=0}^j \binom{j}{k} \partial_s^k(s^r)\partial_s^{j-k}(u^{p-2})
        \end{eqnarray*}
        for some $\alpha_l,\beta_j\in \sO$.
        Note that $\partial_s^l(u^{p-1})$ is divisible by $s^{\gamma_w(E)-l+1}$. Moreover, whenever $r\geq k$, $\partial_s^k(s^r)\partial_s^{j-k}(u^{p-2})$ is divisible by $s^{r-k+\gamma_w(E)-j+k+1}$ and it is $0$ if $r<k$.
        All in all, the right hand side is divisible by $s^{r-p+2+\gamma_w(E)}$.
        As $\lambda u^{p-1}\in \sO^\times$ and $r>0$, this yields $\gamma_w(E) \leq a(E;\varsigma)+p-3$.

        \item Suppose that we are in the case of \autoref{prop:inv_case_II}. Then there is a differential basis $s,x_1,\dots,x_d\in \sO$ of $\sO$, with dual basis $\partial_s,\partial_{x_1},\dots,\partial_{x_d}$ of $\Der_k(\sO)$, such that
                $$\partial_s(u)=\alpha s^{\gamma_w(E)}, \quad 
                \partial_{x_i}(u)=\alpha\partial_{x_i}(s)=0 \text{ for all }i$$
        where $\alpha\in \sO^\times$. Notice that the elements of the dual basis commute with each other.
        So evaluating $D=P(\partial_s,\partial_{x_1},\dots,\partial_{x_d})$ at $u^{p-1}$ yields
                $$\lambda u^{p-1} s^{a(E;\varsigma)}
                =\sum_{j= 1}^{p-1}\alpha_j\partial_s^j(u^{p-1})$$
        Note that $\partial_s^j(u^{p-1})$ is divisible by $s^{\gamma_w(E)-j+1}$, whence the right hand side is divisible by $s^{\gamma_w(E)-p+2}$.
       Since $\lambda u^{p-1} \in \sO^\times$, we obtain that $\gamma_w(E)\leq a(E;\varsigma)+p-2$.
    \end{itemize}
The proof is complete.
\end{proof}

\begin{remark}
The previous stream of arguments (\footnote{
    Except possibly for the fact that $a(E;\varsigma)\in p\mathbb{Z}$ in \autoref{lemma:N_is_p_div}, which requires the fact that the codomain of $\varphi$ is $A^p$.
}) applies in more general situations: it suffices to have an intermediate ring $A^p\subset B \subset A$ and an $A^p$-linear map
$\psi'\colon A(\lceil (p-1)\Delta\rceil +X_\mathbf{0})\to B$ that (co-)restricts to the identity on $A^p$. This applies, for example, to foliated variants of $F$-purity as developed in \cite{foliations_in_progress}.
\end{remark}

We introduce the following function:
        $$\prescript{p}{}\lceil -\rceil\colon \mathbb{R}\to \mathbb{Z}, \quad 
        x\mapsto \min\{n\in p\mathbb{Z}\mid n\geq x\}.$$

\begin{proposition}\label{prop:bound_wild_coeff_in_F_pure_case}
If \autoref{assumption:F_split_along_div} holds, then for $E$ wild we have
        $$\gamma_w(E)\leq \begin{dcases}
            \ \prescript{p}{}\Lceil (p-1)\cdot a(E;X,\Delta+X_\mathbf{0})\Rceil & \text{if }E\text{ is not }T_{X/C}\text{-invariant}, \\
            \ \prescript{p}{}\Lceil (p-1)\cdot a(E;X,\Delta+X_\mathbf{0})\Rceil +p-2 & \text{otherwise}.
        \end{dcases}$$
\end{proposition}
\begin{proof}
Let $i$ be the minimal integer such that $\varsigma*s^i\in \Hom_{\sO^p}(\sO,\sO)$. Then $i\leq a(E;\varsigma)$, and as $a(E;\varsigma)$ is divisible by $p$ (\autoref{lemma:N_is_p_div}), we have $a(E;\varsigma)\leq \prescript{p}{}\lceil i\rceil$.
By \autoref{prop:bound_on_pole} it follows that
        $$a(E;\varsigma)\leq \ \prescript{p}{}\Lceil (p-1)\cdot a(E;X,\Delta+X_\mathbf{0})\Rceil$$
The result follows by combining this inequality with \autoref{lemma:bound_non_inv_case} and \autoref{lemma:bound_inv_case}.
\end{proof}

In view of \autoref{prop:formula_discrep_qt}, the upper bound we have just obtained is too large by roughly a factor of $p$. There are however a few cases where our methods yields a satisfactory result:

\begin{corollary}\label{cor:application_non_inv_div}
Suppose that $(X,\Delta)\to C$ is strongly locally stable, that \autoref{assumption:F_split_along_div} holds, and let $E$ be a wild non-$T_{X/C}$-invariant divisor. If either
    \begin{itemize}
        \item $a(E;X,\Delta+X_\mathbf{0})<1/(p-1)$, or
        \item $p=2$,
    \end{itemize}
then $a(E^{(1)};X^{(1)},\Delta^{(1)}+X_\mathbf{0}^{(1)})\geq -1$.
\end{corollary}
\begin{proof}
Suppose first that $a(E;X,\Delta+X_\mathbf{0})<1/(p-1)$. Then $i<1$ by \autoref{prop:bound_on_pole}, so as $i$ is an integer we get $i\leq 0$. We saw in the proof of \autoref{prop:bound_wild_coeff_in_F_pure_case} that $0\leq a(E;\varsigma)\leq \prescript{p}{}\lceil i\rceil $, so we must have $a(E;\varsigma)=0$. It follows from \autoref{lemma:bound_non_inv_case} that $\gamma_w(E)=0$, and \autoref{prop:formula_discrep_qt} yields the desired inequality.

Next, assume $p=2$. Then $a(E;\varsigma)\leq i+1\leq a(E;X,\Delta+X_\mathbf{0})+1$, so by \autoref{prop:formula_discrep_qt} we get
    $$a(E^{(1)};X^{(1)},\Delta^{(1)}+X_\mathbf{0}^{(1)})
    \geq \frac{a(E;X,\Delta+X_\mathbf{0})}{2}-\frac{1}{2}\geq -1$$
as desired.
\end{proof}

\begin{remark}[Some invariants cases]\label{rmk:application_inv_div}
Suppose that $(X,\Delta)\to C$ is strongly locally stable, that \autoref{assumption:F_split_along_div} holds, and let $E$ be a wild $T_{X/C}$-invariant divisor. Let us assume $a(E;X,\Delta+X_\mathbf{0})<1/(p-1)$, so that as in the previous proof we find $a(E;\varsigma)=0$. From \autoref{lemma:bound_inv_case} and \autoref{prop:formula_discrep_qt} we get
        $$a(E^{(1)};X^{(1)},\Delta^{(1)}+X_\mathbf{0}^{(1)})
        \geq p\cdot a(E;X,\Delta+X_\mathbf{0})-(p-1)(p-2).$$
The right-hand side is $\geq -1$ if and only if
        $$a(E;X,\Delta+X_\mathbf{0})\geq \frac{p^2-3p+1}{p}.$$
This is consistent with $a(E;X,\Delta+X_\mathbf{0})<1/(p-1)$ if and only if $p\in \{2,3\}$. 
\begin{itemize}
    \item If $p=3$ we find that $a(E^{(1)};X^{(1)},\Delta^{(1)}+X_\mathbf{0}^{(1)})\geq -1$ when $a(E;X,\Delta+X_\mathbf{0})\in [1/3; 1/2)$.

    \item If $p=2$ we find that $a(E^{(1)};X^{(1)},\Delta^{(1)}+X_\mathbf{0}^{(1)})\geq -1$ when $a(E;X,\Delta+X_\mathbf{0})\in [-1/2; 1)$.
\end{itemize}
\end{remark}

\begin{remark}\label{rmk:F-splitting_higher_order}
Assume that $(X, \Delta)\to C$ is $F$-locally stable. Then $(X,\Delta)$ is $e$-$F$-split along $X_\mathbf{0}$ for some $e\geq 1$. If $e=1$ then the analysis of this subsection applies; however, when $e>1$ our set-up is not appropriate. The main reason is that the splitting operator is only $A^{p^e}$-linear, and to generalize \autoref{lemma:decomposition_diff_operator} we must consider differential operators that are not compositions of derivations. We have used \autoref{lemma:decomposition_diff_operator} to estimate $\gamma_w(E)=\ord_Edu$ by looking at $\varsigma(u)$. So if $\varsigma$ is replaced by an $A^{p^e}$-linear differential operator, we will be estimating the order of $\delta_nu=d^nu-u$ in the left $\sO$-module of principal parts $\mathscr{P}^{n}_{\sO/k}$ with $n=p^{e-1}$ (see \cite[\S 16.7]{EGA_IV.4}). This invariant is independent of $\ord_E du$ in general (\footnote{
    For example, say that $u=x^p +v$ where $x$ is a differential coordinate of $\sO$ and $v\in \sO$ is any function that does not cancel $x^p$. Then $\gamma_w(E)$ depends on $v$ only and can assume any non-negative value, while $\delta_pu=\delta_px^p+\delta_pv$ has order $0$ since $\delta_px^p$ is not divisible by $s$.
}). 

In fact, in that situation we should work with an appropriate notion of $e$-foliations (in the spirit of \cite{Grabowski_Power_towers} 
and \cite{foliations_in_progress}) and study directly $(X^{(e)},\Delta^{(e)})\to C$. Once the relevant definitions are in place, the statements obtained above should generalize with few changes; however, it is not apparent at the moment whether this more general framework offers a better grip on the evaluation of discrepancies.
\end{remark}


\appendix
\section{Discrepancies of $\bZ/p$-quotients}\label{appendix:Z/p_qt}
In this appendix we derive a formula to compute discrepancies of $\bZ/p$-quotients (\autoref{prop:discrepancy_AS_qt}).

\medskip
We start by recalling some basic facts about $\bZ/p$-actions and quotients, following \cite[\S 3]{Totaro_Terminal_non_CM_3folds}. 
Let $X$ be a normal $k$-scheme of finite type, and suppose that $G=\bZ/p=\langle \sigma\rangle$ acts non-trivially on $X$. We say that an integral divisor on $X$ is $G$-invariant if it is sent to itself by the action of $G$. 
We denote by $X^G$ the scheme-theoretic fixed locus of the $G$-action 
(\footnote{
    If $\Delta\subset X\times X$ is the diagonal and $\Gamma$ is the graph of $\sigma\in \Aut(X)$, then $X^G$ is equal to the pullback under $X\cong \Delta$ of the 
    scheme-theoretic intersection $\Delta\cap \Gamma$. 
}), and $[X^G]$ the $1$-cycle associated to the codimension one part of $X^G$:
\[
[X^G]=
\sum_{\eta \in X^G, \ \dim\sO_{X,\eta}=1} \mathrm{length}_{\sO_{X,\eta}}(\cO_{X^G, \eta})\cdot \overline{\{\eta\}}.
\]
By definition $[X^G]$ is an effective Weil divisor on $X$.
We let $q\colon X\to Y=X/G$ be the scheme-theoretic quotient of $G$: here $q$ is a finite morphism, $Y$ is normal and of finite type over $k$.
We have
        $$K_X=q^*K_Y+(p-1)[X^G].$$
Let $D$ be an irreducible divisor on $Y$. Let $E$ be the reduced preimage of $D$ in $X$. If the generic points of $E=\sum_i E_i$ do not belong to $X^G$, then $\coeff_{E_i}[X^G]=0$ for every $i$ and $q^*D=E$. Otherwise $E$ is a prime $G$-invariant divisor, and we say that the $G$-action is ramified at $E$ (\footnote{
    Caveat: if $E$ is $G$-invariant, the action is not necessarily ramified at $E$.
}). In that case we record the following notations and facts:
    \begin{itemize}
        \item The automorphism $\sigma\colon X\to X$ induces an automorphism $\sigma\colon \sO_{X,E}\to \sO_{X,E}$, and we write
        $I(s)=\sigma(s)-s$ for any $s\in \sO_{X,E}$.
        \item We let the \emph{ramification index} of $E$ be 
            $$e(E)=\frac{p}{[k(E):k(D)]} \in \{1,p\},$$
        where $k(-)$ denotes the function field.
        With this notation we have $q^*D=e(E)\cdot E$. We say that the ramification at $E$ is \emph{wild} if $e(E)=p$, and \emph{fierce} if $e(E)=1$.
        \item We let the \emph{Artin ramification index} of $E$ be (\footnote{
            Note that while the function $I(-)$ depends on a choice of generator of $G$, the function $i(-)$ does not.
        })
                $$i(E)=\coeff_E [X^G]
                =\min_{s\in \sO_{X,E}} \ord_E I(s).$$
            We also record that if $\pi$ is a uniformizer of $\sO_{X,E}$, then the ramification at $E$ is wild (i.e.\ $e(E)=p$) if and only if $\ord_EI(\pi)=i(E)$.
    \end{itemize}
Notice that the $G$-action extends uniquely to the completion of $\sO_{X,E}$ (because the action descends to each quotient $\sO_{X,E}/\fm^n$) or to its henselization (by \cite[04GS]{Stacks_Project}), and that the above invariants can be computed there.
    
We extend the function $e(-)$ to all prime divisors on $X$ by setting $e(E)=1$ if the $G$-action is unramified at $E$. So in all cases we have $q^*F=\sum_{E:\  q(E)=F}e(E)\cdot E$.

\begin{lemma}\label{lemma:wild_ramification_triv_resid_action}
Suppose that the $G$-action is wildly ramified at $E$. If $\fm$ is the maximal ideal of $\sO_{X,E}$, then the induced action on $\fm^i/\fm^{i+1}$ is trivial for every $i\geq 0$.
\end{lemma}
\begin{proof}
Each $\fm^i/\fm^{i+1}$ is a one-dimensional $k(E)$-vector space.
A linear one-dimensional representation of $G$ is always trivial, so it suffices to know that the induced actions of $\sigma$ are $k(E)$-linear.
This is equivalent to saying that the action on $k(E)$ is trivial. Since the ramification is wild, the inclusion of residue fields $k(D)\hookrightarrow k(E)$ induced by $\sO_{Y,D}=\sO_{X,E}^G\hookrightarrow \sO_{X,E}$ is an equality. So the statement holds. 
\end{proof}

\begin{corollary}\label{cor:wild_ram_implies_ord_is_bigger}
Assume that the $G$-action is wildly ramified at $E$. Then for every $s\in \sO_{X,E}$ we have $\ord_EI(s)>\ord_E s$.
\end{corollary}
\begin{proof}
We have $i:=\ord_E\sigma(s)=\ord_Es$, so it suffices to show that the image of $I(s)$ in $\fm^i/\fm^{i+1}$ is zero. This follows from \autoref{lemma:wild_ramification_triv_resid_action}. 
\end{proof}
    
Let us apply this theory to the computation of discrepancies of $\bZ/p$-quotients. We let $X$ and $Y$ be as above. Let $\varphi\colon Y'\to Y$ be a birational finite type morphism from a normal $k$-scheme. We let $X'$ be the normalization of $Y'$ in the function field $K(X)$. Then $G$ acts on $X'$ with quotient $Y'$ and we have a commutative diagram
        $$\begin{tikzcd}
        X'\arrow[d, "\varphi'"] \arrow[r, "q'"] & 
        Y'\arrow[d, "\varphi"] \\
        X \arrow[r, "q"] & Y
        \end{tikzcd}$$
where $\varphi'$ is birational and $G$-equivariant. Let $\Delta=\sum_ic_i\Delta_i$ be a $\bQ$-divisor on $Y$, and $\Delta_X=\sum_i c_i\Delta_{X,i}$, where $\Delta_{X,i}$ is the reduced divisor with support $q^{-1}\Delta_i$. \emph{Note that $\Delta_X\neq q^*\Delta$ in general, so here we depart from our usual way of pulling back divisors (see \autoref{def:base_change_divisors})}.

\begin{theorem}\label{prop:discrepancy_AS_qt}
In the situation above, let $D$ be a $\varphi$-exceptional divisor, and $E$ be a $\varphi'$-exceptional divisor such that $q'(E)=D$. Assume that 
        $$K_Y+\Delta, \quad K_X+\Delta_X
        \quad \text{and}\quad 
        \sum_i(e(\Delta_{X,i})-1) c_i\Delta_{X,i}$$
are $\bQ$-Cartier. Then:
            \begin{eqnarray*} 
    a(D;Y,\Delta)&=&
    \frac{1}{e(E)}\Big[
    a(E;X,\Delta_X) 
    -\coeff_E\left(\sum_i(e(\Delta_{X,i})-1) \varphi'^*c_i\Delta_{X,i} \right) \\
    && \quad\quad\quad \left. +(p-1)\cdot \coeff_E\left(
    \varphi'^*[X^G]-[X'^G]
    \right)
    \right].
    \end{eqnarray*}
In particular, if the $G$-actions are unramified along $E$ and $\Delta_X$, then $a(D;Y,\Delta)\geq a(E;X,\Delta_X)$.
\end{theorem}
\begin{proof}
We pullback $K_Y+\Delta$ along $\varphi\circ q'=q\circ\varphi'$, using the remarks made above. Let $\Delta_{X'}=\varphi'^{-1}_*\Delta_X$. On one hand we have around the generic point of $E$:
        \begin{eqnarray*}
            \varphi'^*q^*(K_Y+\Delta) &=&
            \varphi'^*\left(
                K_X+\Delta_X-(p-1)[X^G] +\sum_i (e(\Delta_{X,i})-1)c_i\Delta_{X,i}
            \right) \\
            &=&
            K_{X'}+\Delta_{X'}-a(E;X,\Delta_X)\cdot E \\
            && \quad
            +\coeff_E \left(
            -(p-1)\varphi'^*[X^G]
            +\sum_i (e(\Delta_{X,i})-1)\varphi'^*c_i\Delta_{X,i}
            \right)\cdot E.
        \end{eqnarray*}
Let $\Delta_{Y'}=\varphi^{-1}_*\Delta$ and observe that $q'^*\Delta_{Y'}=\sum_i (e(\Delta_{X,i})-1)\varphi'^{-1}_*c_i\Delta_{X,i}$. 
Then around the generic point of $E$ we have:
        \begin{eqnarray*}
            q'^*\varphi^*(K_Y+\Delta) &=& 
            q'^*(K_{Y'}+\Delta_{Y'}-a(D;Y,\Delta)\cdot D) \\
            &=&
            K_{X'}+\Delta_{X'}-(p-1)[X'^G]-e(D)a(D;Y,\Delta)\cdot E.
        \end{eqnarray*}
By equating the two expressions we find the desired formula. To conclude, assume that the $G$-actions are unramified along $E$ and $\Delta$. Then $\coeff_E[X'^G]=0$, $\coeff_E\varphi'^*[X^G]=0$, $e(E)=1$ and $e(\Delta_{X,i})=1$ for every $i$. So
        $$a(D;Y,\Delta)=a(E;X,\Delta_X)+
        (p-1)\coeff_E\varphi'^*[X^G]
        \geq a(E;X,\Delta_X)$$
as claimed.
\end{proof}

\begin{remark}
Suppose that $\psi\colon Z'\to Z$ is $G$-equivariant of schemes. Then while $\psi(Z'^G)\subseteq Z^G$, it is not true in general that $\psi^{-1}(Z^G)\subseteq Z'^G$. So in \autoref{prop:discrepancy_AS_qt} the support of $\varphi'^*[X^G]$ is not necessarily contained in the support of $[X'^G]$.
\end{remark}

\bibliographystyle{alpha}
\bibliography{Bibliography}

@article {EGA_IV.1,
    AUTHOR = {Grothendieck, A.},
     TITLE = {\'{E}l\'{e}ments de g\'{e}om\'{e}trie alg\'{e}brique. {IV}. \'{E}tude locale des
              sch\'{e}mas et des morphismes de sch\'{e}mas. {I}},
   JOURNAL = {Inst. Hautes \'{E}tudes Sci. Publ. Math.},
  FJOURNAL = {Institut des Hautes \'{E}tudes Scientifiques. Publications
              Math\'{e}matiques},
    NUMBER = {20},
      YEAR = {1964},
     PAGES = {259},
      ISSN = {0073-8301},
   MRCLASS = {14.05 (13.10)},
  MRNUMBER = {173675},
MRREVIEWER = {H. Hironaka},
       URL = {http://www.numdam.org/item/PMIHES_1964__20__5_0},
}

@article {EGA_IV.2,
    AUTHOR = {Grothendieck, A.},
     TITLE = {\'{E}l\'{e}ments de g\'{e}om\'{e}trie alg\'{e}brique. {IV}. \'{E}tude locale des
              sch\'{e}mas et des morphismes de sch\'{e}mas. {II}},
   JOURNAL = {Inst. Hautes \'{E}tudes Sci. Publ. Math.},
  FJOURNAL = {Institut des Hautes \'{E}tudes Scientifiques. Publications
              Math\'{e}matiques},
    NUMBER = {24},
      YEAR = {1965},
     PAGES = {231},
      ISSN = {0073-8301},
   MRCLASS = {14.00},
  MRNUMBER = {199181},
MRREVIEWER = {H. Hironaka},
       URL = {http://www.numdam.org/item?id=PMIHES_1965__24__231_0},
}

@article {EGA_IV.3,
    AUTHOR = {Grothendieck, A.},
     TITLE = {\'{E}l\'{e}ments de g\'{e}om\'{e}trie alg\'{e}brique. {IV}. \'{E}tude locale des
              sch\'{e}mas et des morphismes de sch\'{e}mas. {III}},
   JOURNAL = {Inst. Hautes \'{E}tudes Sci. Publ. Math.},
  FJOURNAL = {Institut des Hautes \'{E}tudes Scientifiques. Publications
              Math\'{e}matiques},
    NUMBER = {28},
      YEAR = {1966},
     PAGES = {255},
      ISSN = {0073-8301},
   MRCLASS = {14.55},
  MRNUMBER = {0217086},
MRREVIEWER = {J. P. Murre},
       URL = {http://www.numdam.org/item?id=PMIHES_1966__28__255_0},
}

@article {EGA_IV.4,
    AUTHOR = {Grothendieck, A.},
     TITLE = {\'{E}l\'{e}ments de g\'{e}om\'{e}trie alg\'{e}brique. {IV}. \'{E}tude locale des
              sch\'{e}mas et des morphismes de sch\'{e}mas {IV}},
   JOURNAL = {Inst. Hautes \'{E}tudes Sci. Publ. Math.},
  FJOURNAL = {Institut des Hautes \'{E}tudes Scientifiques. Publications
              Math\'{e}matiques},
    NUMBER = {32},
      YEAR = {1967},
     PAGES = {361},
      ISSN = {0073-8301},
   MRCLASS = {14.55},
  MRNUMBER = {238860},
MRREVIEWER = {J. P. Murre},
       URL = {http://www.numdam.org/item?id=PMIHES_1967__32__361_0},
}

@book {Bourbaki_AC_V-VII,
    AUTHOR = {Bourbaki, Nicolas},
     TITLE = {\'{E}l\'{e}ments de math\'{e}matique},
      NOTE = {Alg\`ebre commutative. Chapitres 5 \`a 7.,
              Reprint},
 PUBLISHER = {Masson, Paris},
      YEAR = {1985},
     PAGES = {351},
      ISBN = {2-225-80269-6},
   MRCLASS = {13-02},
  MRNUMBER = {782297},
}

@article{Stacks_Project,
	TITLE = {The {S}tacks {P}roject},
      JOURNAL={\url{https://stacks.math.columbia.edu}},
}

@book {Matsumura_Commutative_Ring_Theory,
    AUTHOR = {Matsumura, Hideyuki},
     TITLE = {Commutative ring theory},
    SERIES = {Cambridge Studies in Advanced Mathematics},
    VOLUME = {8},
   EDITION = {Second},
      NOTE = {Translated from the Japanese by M. Reid},
 PUBLISHER = {Cambridge University Press, Cambridge},
      YEAR = {1989},
     PAGES = {xiv+320},
      ISBN = {0-521-36764-6},
   MRCLASS = {13-01},
  MRNUMBER = {1011461},
}

@book {Kollar_Mori_Birational_geometry_of_algebraic_varieties,
    AUTHOR = {Koll{\'a}r, J{\'a}nos and Mori, Shigefumi},
     TITLE = {Birational geometry of algebraic varieties},
    SERIES = {Cambridge Tracts in Mathematics},
    VOLUME = {134},
      NOTE = {With the collaboration of C. H. Clemens and A. Corti,
              Translated from the 1998 Japanese original},
 PUBLISHER = {Cambridge University Press},
   ADDRESS = {Cambridge},
      YEAR = {1998},
     PAGES = {viii+254},
      ISBN = {0-521-63277-3},
   MRCLASS = {14E30},
  MRNUMBER = {MR1658959 (2000b:14018)},
MRREVIEWER = {Mark Gross},
}

@book {Kollar_Singularities_of_the_minimal_model_program,
    AUTHOR = {Koll\'{a}r, J\'{a}nos},
     TITLE = {Singularities of the minimal model program},
    SERIES = {Cambridge Tracts in Mathematics},
    VOLUME = {200},
      NOTE = {With a collaboration of S\'{a}ndor Kov\'{a}cs},
 PUBLISHER = {Cambridge University Press, Cambridge},
      YEAR = {2013},
     PAGES = {x+370},
      ISBN = {978-1-107-03534-8},
   MRCLASS = {14E30 (14B05)},
  MRNUMBER = {3057950},
MRREVIEWER = {Tommaso De Fernex},
       DOI = {10.1017/CBO9781139547895},
       URL = {https://doi.org/10.1017/CBO9781139547895},
}

@book{Cossec_Dolgachev_Liedtke_Enriques_surfaces_I,
 author = {Cossec, Fran{\c{c}}ois and Dolgachev, Igor and Liedtke, Christian},
 title = {Enriques surfaces {I}},
 edition = {2nd edition},
 isbn = {978-981-961213-0; 978-981-961216-1; 978-981-961214-7},
 year = {2025},
 publisher = {Singapore: Springer},
 language = {English},
 doi = {10.1007/978-981-96-1214-7},
 zbMATH = {8017469},
 Zbl = {1566.14075}
}

@article{Arvidsson_Bernasconi_Patakfalvi_Properness_moduli_surfaces,
 author = {Arvidsson, Emelie and Bernasconi, Fabio and Patakfalvi, Zsolt},
 title = {On the properness of the moduli space of stable surfaces over {{\(\mathbb{Z} [1/30]\)}}},
 fjournal = {Moduli},
 journal = {Moduli},
 issn = {2949-7647},
 volume = {1},
 pages = {37},
 note = {Id/No e3},
 year = {2024},
 language = {English},
 doi = {10.1112/mod.2024.1},
 zbMATH = {7969267}
}

@book {Kollar_Families_of_varieties_of_general_type,
    AUTHOR = {Koll\'{a}r, J\'{a}nos},
     TITLE = {Families of varieties of general type},
    SERIES = {Cambridge Tracts in Mathematics},
    VOLUME = {231},
      NOTE = {With the collaboration of Klaus Altmann and S\'{a}ndor J. Kov\'{a}cs},
 PUBLISHER = {Cambridge University Press, Cambridge},
      YEAR = {2023},
     PAGES = {xviii+471},
      ISBN = {978-1-009-34610-8},
   MRCLASS = {14J10 (14D20 14E30 14J29)},
  MRNUMBER = {4566297},
MRREVIEWER = {Chenyang Xu},
}

@incollection {Kollar_Quotients_by_finite_equivalence_relations,
    AUTHOR = {Koll\'{a}r, J\'{a}nos},
     TITLE = {Quotients by finite equivalence relations},
 BOOKTITLE = {Current developments in algebraic geometry},
    SERIES = {Math. Sci. Res. Inst. Publ.},
    VOLUME = {59},
     PAGES = {227--256},
      NOTE = {With an appendix by Claudiu Raicu},
 PUBLISHER = {Cambridge Univ. Press, Cambridge},
      YEAR = {2012},
   MRCLASS = {14A15 (14E20)},
  MRNUMBER = {2931872},
MRREVIEWER = {Leovigildo M. Alonso Tarrio},
}

@article{Hacon_Witaszek_On_the_relative_MMP_for_threefolds_in_low_char,
 author = {Hacon, Christopher and Witaszek, Jakub},
 title = {On the relative minimal model program for threefolds in low characteristics},
 fjournal = {Peking Mathematical Journal},
 journal = {Peking Math. J.},
 issn = {2096-6075},
 volume = {5},
 number = {2},
 pages = {365--382},
 year = {2022},
 language = {English},
 doi = {10.1007/s42543-021-00037-7},
 zbMATH = {7597118},
 Zbl = {1509.14034}
}

@article {Posva_Gluing_for_surfaces_and_threefolds,
    AUTHOR = {Posva, Quentin},
     TITLE = {Gluing theory for slc surfaces and threefolds in positive
              characteristic},
   JOURNAL = {Ann. Sc. Norm. Super. Pisa Cl. Sci. (5)},
  FJOURNAL = {Annali della Scuola Normale Superiore di Pisa. Classe di
              Scienze. Serie V},
    VOLUME = {25},
      YEAR = {2024},
    NUMBER = {2},
     PAGES = {811--886},
      ISSN = {0391-173X},
   MRCLASS = {14J17 (14E30 14G17 14J29 14J30)},
  MRNUMBER = {4778466},
}

@article {Hara_Watanabe_F_sing_vs_lt_and_lc,
    AUTHOR = {Hara, Nobuo and Watanabe, Kei-Ichi},
     TITLE = {F-regular and {F}-pure rings vs. log terminal and log
              canonical singularities},
   JOURNAL = {J. Algebraic Geom.},
  FJOURNAL = {Journal of Algebraic Geometry},
    VOLUME = {11},
      YEAR = {2002},
    NUMBER = {2},
     PAGES = {363--392},
      ISSN = {1056-3911},
   MRCLASS = {13A35 (14B05 14J17)},
  MRNUMBER = {1874118},
MRREVIEWER = {Karen E. Smith},
       DOI = {10.1090/S1056-3911-01-00306-X},
       URL = {https://doi.org/10.1090/S1056-3911-01-00306-X},
}

@incollection{Polstra_Simpson_Tucker_F-pure_inversion_of_adjunction,
 author = {Polstra, Thomas and Simpson, Austyn and Tucker, Kevin},
 title = {On {{\(F\)}}-pure inversion of adjunction},
 booktitle = {Higher dimensional algebraic geometry. A volume in honor of V. V. Shokurov to his 70th birthday. Based on the Japan-US Mathematics Institute (JAMI) conference, Baltimore, MD, USA May 3--8, 2022},
 isbn = {978-1-00-939624-0; 978-1-00-939623-3},
 pages = {319--344},
 year = {2025},
 publisher = {Cambridge: Cambridge University Press},
 language = {English},
 doi = {10.1017/9781009396233.019},
 zbMATH = {8085511}
}

@article {Kollar_Families_of_3folds_in_pos_char,
    AUTHOR = {Koll\'{a}r, J\'{a}nos},
     TITLE = {Families of stable 3-folds in positive characteristic},
   JOURNAL = {\'{E}pijournal G\'{e}om. Alg\'{e}brique},
  FJOURNAL = {\'{E}pijournal de G\'{e}om\'{e}trie Alg\'{e}brique. EPIGA},
    VOLUME = {7},
      YEAR = {2023},
     PAGES = {Art. 6, 11},
   MRCLASS = {14J10 (14D22 14E30 14G17 14J30)},
  MRNUMBER = {4582879},
MRREVIEWER = {Jingshan Chen},
       DOI = {10.46298/epiga.2023.volume7.9730},
       URL = {https://doi.org/10.46298/epiga.2023.volume7.9730},
}

@article{Posva_Gluing_surfaces_mixed_char,
	TITLE={Gluing for stable families of surfaces in mixed characteristic},
	AUTHOR={Posva, Quentin},
	YEAR={2021},
	JOURNAL={To appear in Algebraic Geometry (Foundation Compositio Mathematica)},
}

@article{Posva_Pathological_MMP_sing_alpha_p_qts,
 author = {Posva, Quentin},
 title = {Pathological {MMP} singularities as {{\(\alpha_p\)}}-quotients},
 fjournal = {Forum of Mathematics, Sigma},
 journal = {Forum Math. Sigma},
 issn = {2050-5094},
 volume = {13},
 pages = {29},
 note = {Id/No e185},
 year = {2025},
 language = {English},
 doi = {10.1017/fms.2025.10129},
 zbMATH = {8120920}
}

@article{Posva_Singularities_quotients_by_1_foliations,
 author = {Posva, Quentin},
 title = {On the singularities of quotients by 1-foliations},
 fjournal = {Nagoya Mathematical Journal},
 journal = {Nagoya Math. J.},
 issn = {0027-7630},
 volume = {261},
 pages = {41},
 note = {Id/No e6},
 year = {2026},
 language = {English},
 doi = {10.1017/nmj.2025.10072},
 zbMATH = {8135533}
}

@article{Posva_Resolution_of_1_foliations,
	TITLE={Resolution of 1-foliations on surfaces and threefolds},
	AUTHOR={Posva, Quentin},
	YEAR={2024},
	journal={To appear in Algebra and Number Theory},
}

@misc{Totaro_Terminal_non_CM_3folds,
	TITLE={Terminal 3-folds that are not {C}ohen--{M}acaulay},
	AUTHOR={Totaro, Burt},
	YEAR={2024},
	howpublished={Preprint, arXiv:2407.02608v2 [math.AG]},
}

@article {Hacon_Witaszek_MMP_for_4folds_pos_and_mixed_char,
    AUTHOR = {Hacon, Christopher and Witaszek, Jakub},
     TITLE = {On the relative minimal model program for fourfolds in
              positive and mixed characteristic},
   JOURNAL = {Forum Math. Pi},
  FJOURNAL = {Forum of Mathematics. Pi},
    VOLUME = {11},
      YEAR = {2023},
     PAGES = {Paper No. e10, 35},
   MRCLASS = {14E30 (14G17)},
  MRNUMBER = {4565409},
MRREVIEWER = {Artie Prendergast-Smith},
       DOI = {10.1017/fmp.2023.6},
       URL = {https://doi.org/10.1017/fmp.2023.6},
}

@misc{Hu_Zong_Base_change_local-stability_positive_char,
 author = {Hu, Zhi and Zong, Runhong},
 title = {On {Base} {Change} of {Local} {Stability} in {Positive} {Characteristics}},
 year = {2020},
 howpublished = {Preprint, {arXiv}:2001.04083 [math.{AG}]},
 url = {https://arxiv.org/abs/2001.04083},
 arXiv = {arXiv:2001.04083}
}

@article{Schroeer_Fibrations_with_nonreduced_geom_fibers,
 author = {Schr{\"o}er, Stefan},
 title = {On fibrations whose geometric fibers are nonreduced},
 fjournal = {Nagoya Mathematical Journal},
 journal = {Nagoya Math. J.},
 issn = {0027-7630},
 volume = {200},
 pages = {35--57},
 year = {2010},
 language = {English},
 doi = {10.1215/00277630-2010-011},
 zbMATH = {5858598},
 Zbl = {1217.14004}
}

@incollection{Ishii_Reguera_Singularities_in_arb_char_via_jets,
 author = {Ishii, Shihoko and Reguera, Ana J.},
 title = {Singularities in arbitrary characteristic via jet schemes},
 booktitle = {Hodge Theory and \(L^2\)-analysis},
 isbn = {978-1-57146-351-7},
 pages = {419--449},
 year = {2017},
 publisher = {Somerville, MA: International Press; Beijing: Higher Education Press},
 language = {English},
 zbMATH = {6817383},
 Zbl = {1379.14022}
}

@incollection{Ein_Ishii_Singularities_wrt_MJ_discrepancies,
 author = {Ein, Lawrence and Ishii, Shihoko},
 title = {Singularities with respect to {Mather}-{Jacobian} discrepancies},
 booktitle = {Commutative algebra and noncommutative algebraic geometry. Volume II: Research articles},
 isbn = {978-1-107-14972-4},
 pages = {125--168},
 year = {2015},
 publisher = {Cambridge: Cambridge University Press},
 language = {English},
 zbMATH = {6703995},
 Zbl = {1359.14020}
}

@book{Lang_Algebra,
 author = {Lang, Serge},
 title = {Algebra.},
 edition = {3rd revised ed.},
 fseries = {Graduate Texts in Mathematics},
 series = {Grad. Texts Math.},
 issn = {0072-5285},
 volume = {211},
 isbn = {0-387-95385-X},
 year = {2002},
 publisher = {New York, NY: Springer},
 language = {English},
 zbMATH = {1703931},
 Zbl = {0984.00001}
}

@article{Benozzo_CBF_in_pos_char,
 author = {Benozzo, Marta},
 title = {On the canonical bundle formula in positive characteristic},
 year = {2026},
journal = {To appear in Journal of the London Mathematical Society},
 url = {https://arxiv.org/abs/2305.19841},
 arXiv = {arXiv:2305.19841}
}

@article{Bernasconi_Tanaka_On_dP_fibrations_in_pos_char,
 author = {Bernasconi, Fabio and Tanaka, Hiromu},
 title = {On del {Pezzo} fibrations in positive characteristic},
 fjournal = {Journal of the Institute of Mathematics of Jussieu},
 journal = {J. Inst. Math. Jussieu},
 issn = {1474-7480},
 volume = {21},
 number = {1},
 pages = {197--239},
 year = {2022},
 language = {English},
 doi = {10.1017/S1474748020000067},
 zbMATH = {7463762},
 Zbl = {1483.14030}
}

@article{7authors_globally+regular_MMP,
 author = {Bhatt, Bhargav and Ma, Linquan and Patakfalvi, Zsolt and Schwede, Karl and Tucker, Kevin and Waldron, Joe and Witaszek, Jakub},
 title = {Globally {{\(+\)}}-regular varieties and the minimal model program for threefolds in mixed characteristic},
 fjournal = {Publications Math{\'e}matiques},
 journal = {Publ. Math., Inst. Hautes {\'E}tud. Sci.},
 issn = {0073-8301},
 volume = {138},
 pages = {69--227},
 year = {2023},
 language = {English},
 doi = {10.1007/s10240-023-00140-8},
 zbMATH = {7780366},
 Zbl = {1533.14009}
}

@misc{Bernasconi_Posva_Equising_lifting_K-trivial_slc_surfaces,
 author = {Bernasconi, Fabio and Posva, Quentin},
 title = {Equisingular lifting of semi-log canonical ${F}$-split ${K}$-trivial surfaces},
 year = {2025},
 howpublished = {Preprint, {arXiv}:2506.01007v2 [math.{AG}]},
 url = {https://arxiv.org/abs/2506.01007},
 arXiv = {arXiv:2506.01007v2}
}

@article{Greco_Traverso_Seminormal_schemes,
 author = {Greco, S. and Traverso, C.},
 title = {On seminormal schemes},
 fjournal = {Compositio Mathematica},
 journal = {Compos. Math.},
 issn = {0010-437X},
 volume = {40},
 pages = {325--365},
 year = {1980},
 language = {English},
 url = {https://eudml.org/doc/89442},
 zbMATH = {3639757},
 Zbl = {0412.14024}
}

@article{Sato_General_hyperplane_section_of_lc_3folds_pos_char,
 author = {Sato, Kenta},
 title = {General hyperplane sections of log canonical threefolds in positive characteristic},
 fjournal = {Journal of the Institute of Mathematics of Jussieu},
 journal = {J. Inst. Math. Jussieu},
 issn = {1474-7480},
 volume = {24},
 number = {5},
 pages = {1867--1894},
 year = {2025},
 language = {English},
 doi = {10.1017/S1474748025000118},
 zbMATH = {8095524}
}

@article{LMMPJoeZsolt,
 author = {Patakfalvi, Zsolt and Waldron, Joe},
 title = {Singularities of general fibers and the {LMMP}},
 fjournal = {American Journal of Mathematics},
 journal = {Am. J. Math.},
 issn = {0002-9327},
 volume = {144},
 number = {2},
 pages = {505--540},
 year = {2022},
 language = {English},
 doi = {10.1353/ajm.2022.0009},
 zbMATH = {7496843},
 Zbl = {1498.14036}
}

@article{Smith_cdBergh_Simplicity_diff_operators_in_pos_char,
 author = {Smith, Karen E. and Van den Bergh, Michel},
 title = {Simplicity of rings of differential operators in prime characteristic},
 fjournal = {Proceedings of the London Mathematical Society. Third Series},
 journal = {Proc. Lond. Math. Soc. (3)},
 issn = {0024-6115},
 volume = {75},
 number = {1},
 pages = {32--62},
 year = {1997},
 language = {English},
 doi = {10.1112/S0024611597000257},
 zbMATH = {1053384},
 Zbl = {0948.16019}
}

@article{Bernasconi_Brivio_Filipazzi_Deformations_of_3folds,
 author = {Bernasconi, Fabio and Brivio, Iacopo and Filipazzi, Stefano},
 title = {Arithmetic and geometric deformations of threefolds},
 fjournal = {Bulletin of the London Mathematical Society},
 journal = {Bull. Lond. Math. Soc.},
 issn = {0024-6093},
 volume = {56},
 number = {1},
 pages = {423--443},
 year = {2024},
 language = {English},
 doi = {10.1112/blms.12942},
 zbMATH = {7798523},
 Zbl = {1543.14005}
}

@book {Schwede_Smith_book,
    AUTHOR = {Schwede, Karl and Smith, Karen},
     TITLE = {Singularities defined by the Frobenius map},
 PUBLISHER = {Book in progress, available at \url{https://github.com/kschwede/FrobeniusSingularitiesBook}},
      YEAR = {2026},
       URL = {https://github.com/kschwede/FrobeniusSingularitiesBook},
}

@misc{Grabowski_Power_towers,
 author = {Przemys{\l}aw Grabowski},
 title = {Power {Towers}: {Purely} {Inseparable} {Galois} {Theory} and {Foliations} in {Positive} {Characteristic}},
 year = {2025},
 howpublished = {Preprint, {arXiv}:2510.06058 [math.{AG}]},
 url = {https://arxiv.org/abs/2510.06058},
 arXiv = {arXiv:2510.06058}
}

@misc{Hiroi_Quotients_by_adjoint_foliations,
 author = {Yutaro Hiroi},
 title = {Quotients by {$(p-1)/p$}-klt {Foliations} on {Surfaces}},
 year = {2026},
 howpublished = {Preprint, {arXiv}:2602.20703 [math.{AG}]},
 url = {https://arxiv.org/abs/2602.20703},
 arXiv = {arXiv:2602.20703}
}

@article{foliations_in_progress,
 author = {Carvajal-Rojas, Javier and Ciprietti, Samuele and Fakhrivaighan, Ghazaleh and Fayolle, Anne},
 title = {{$F$}-singularities of foliations},
 year = {2026},
 journal = {Work in progress, private correspondence}
}

@misc{ACSS_Positivity_moduli_part,
 author = {Florin Ambro and Paolo Cascini and Vyacheslav Shokurov and Calum Spicer},
 title = {Positivity of the {Moduli} {Part}},
 year = {2021},
 howpublished = {Preprint, {arXiv}:2111.00423 [math.{AG}]},
 url = {https://arxiv.org/abs/2111.00423},
 arXiv = {arXiv:2111.00423}
}

@article{AK_Weak_semistable_reduction,
 author = {Abramovich, D. and Karu, K.},
 title = {Weak semistable reduction in characteristic 0},
 fjournal = {Inventiones Mathematicae},
 journal = {Invent. Math.},
 issn = {0020-9910},
 volume = {139},
 number = {2},
 pages = {241--273},
 year = {2000},
 language = {English},
 doi = {10.1007/s002229900024},
 zbMATH = {1443405},
 Zbl = {0958.14006}
}

@book{Raynaud_Anneaux_locaux_henseliens,
 author = {Raynaud, Michel},
 title = {Anneaux locaux hens{\'e}liens},
 fseries = {Lecture Notes in Mathematics},
 series = {Lect. Notes Math.},
 issn = {0075-8434},
 volume = {169},
 year = {1970},
 publisher = {Springer, Cham},
 language = {French},
 doi = {10.1007/bfb0069571},
 zbMATH = {3322176},
 Zbl = {0203.05102}
}

@misc{TY_Quasi_split,
 author = {Teppei Takamatsu and Shou Yoshikawa},
 title = {Quasi-{$F^{\infty}$}-split height versus quasi-{$F$}-regular height for rational double points and graded rings},
 year = {2026},
 howpublished = {Preprint, {arXiv}:2601.03491 [math.{AG}]},
 url = {https://arxiv.org/abs/2601.03491},
 arXiv = {arXiv:2601.03491}
}

@article{Hartshorne_Generalized_div_on_Gor_schemes,
 author = {Hartshorne, Robin},
 title = {Generalized divisors on {Gorenstein} schemes},
 fjournal = {\(K\)-Theory},
 journal = {\(K\)-Theory},
 issn = {0920-3036},
 volume = {8},
 number = {3},
 pages = {287--339},
 year = {1994},
 language = {English},
 doi = {10.1007/BF00960866},
 zbMATH = {682612},
 Zbl = {0826.14005}
}

@misc{Posva_Local_stability_comparison,
 author = {Quentin Posva},
 title = {Local stability of {Frobenius} and {Artin}--{Schreier} base-change: a comparison},
 year = {2026},
 howpublished = {Preprint, {arXiv}:2608.24256 [math.{AG}] (2026)},
 url = {https://arxiv.org/abs/2608.24256},
 arXiv = {arXiv:2608.24256}
}

\end{document}